\documentclass{amsart}
\usepackage{color}

\usepackage{enumerate}
\usepackage{enumitem}
\usepackage{moreenum}
\usepackage{amsmath} 
\usepackage{amssymb}
\usepackage{mathrsfs}

\newtheorem{theorem}{Theorem}[section] 
\newtheorem{claim}[theorem]{Claim}
\newtheorem{sclaim}[theorem]{Subclaim}
\newtheorem{lemma}[theorem]{Lemma} 
\newtheorem{proposition}[theorem]{Proposition} 
\newtheorem{observation}[theorem]{Observation} 
 
\newtheorem{fact}[theorem]{Fact} 

\newcommand{\thistheoremname}{}
\newtheorem*{genericthm*}{\thistheoremname}
\newenvironment{namedthm*}[1]
{\renewcommand{\thistheoremname}{#1}%
	\begin{genericthm*}}
	{\end{genericthm*}}

\theoremstyle{definition}
\newtheorem{definition}[theorem]{Definition}

\newtheorem{problem}[theorem]{Problem}

\newtheorem{convention}[theorem]{Convention}
\newtheorem{cor}[theorem]{Corollary}

\theoremstyle{remark}
\newtheorem{remark}[theorem]{Remark}
\newtheorem{notation}[theorem]{Notation}

\newcommand{\lh}{{\ell g}}

\newcommand{\rest}{{\restriction}}
\newcommand{\dom}{{\rm dom}} 

\newcommand{\ran}{{\rm ran}}

\newcommand{\clpr}{{\rm cl}^{\rm pr}}
\newcommand{\cll}{{\rm cl}}
\newcommand{\suc}{{\rm succ}}

\newcommand{\then}{{\underline{then}}}
\newcommand{\when}{{\underline{when}}}
\newcommand{\where}{{\underline{where}}}
\newcommand{\Then}{{\underline{Then}}}

\newcommand{\If}{{\underline{if}}}

\newcommand{\mn}{{\medskip\noindent}}
\newcommand{\sn}{{\smallskip\noindent}}

\newcommand{\varp}{{\varepsilon}}

\newcommand{\cB}{{\mathcal B}}

\newcommand{\cF}{{\mathcal F}}

\newcommand{\sG}{{\mathscr G}}
\newcommand{\cI}{{\mathcal I}}

\newcommand{\gp}{{\mathfrak p}}
\newcommand{\gq}{{\mathfrak q}}

\newcommand{\bbQ}{{\mathbb Q}}
\newcommand{\bbZ}{{\mathbb Z}}

\newcommand{\tieconcat}{%
	\mathbin{\mathpalette\dotieconcat\relax}%
}
\newcommand{\dotieconcat}[2]{% auxiliary macro, don't use it directly
	\text{\raisebox{.8ex}{$\smallfrown$}}%
}

\newcount\skewfactor
\def\mathunderaccent#1#2 {\let\theaccent#1\skewfactor#2
\mathpalette\putaccentunder}
\def\putaccentunder#1#2{\oalign{$#1#2$\crcr\hidewidth
\vbox to.2ex{\hbox{$#1\skew\skewfactor\theaccent{}$}\vss}\hidewidth}}
\def\name{\mathunderaccent\tilde-3 }

\newenvironment{PROOF}[2][\proofname.]
   {\begin{proof}[#1]\renewcommand{\qedsymbol}{\textsquare$_{\rm #2}$}}
   {\end{proof}}

\usepackage{hyperref}

\begin{document}

\title {Between Whitehead groups and uniformization \\
 Sh486}
\author {M\'ark Po\'or}
\address{Einstein Institute of Mathematics\\
	Edmond J. Safra Campus, Givat Ram\\
	The Hebrew University of Jerusalem\\
	Jerusalem, 91904, Israel\\}
\email{sokmark@gmail.com}

\author {Saharon Shelah}
\address{Einstein Institute of Mathematics\\
Edmond J. Safra Campus, Givat Ram\\
The Hebrew University of Jerusalem\\
Jerusalem, 91904, Israel\\
 and \\
 Department of Mathematics\\
 Hill Center - Busch Campus \\ 
 Rutgers, The State University of New Jersey \\
 110 Frelinghuysen Road \\
 Piscataway, NJ 08854-8019 USA}
\email{shelah@math.huji.ac.il}
\urladdr{http://shelah.logic.at}
\thanks{$^\dag$The first author was supported by the Excellence Fellowship Program for International Postdoctoral Researchers of The Israel Academy of Sciences and Humanities, and by the National Research, Development and Innovation Office
	- NKFIH, grants no. 124749, 129211. \\
	$^\ast$The second author was supported	by the Israel Science Foundation grants
	1838/19 and 2320/23. Paper 486 on Shelah's list. \\
	References like \cite[Th0.2=Ly5]{Sh:950} means the label
	of Th.0.2 is y5.  The reader should note that the version on the second author's website is usually more updated than the one in the mathematical archive.}

%Previous Revision -  10/Mar/2

%First version; Oct.91
%Pub. 486, done fall 91, revised 9/92, 9/93 \null\newline
%First Typed/previous version at RU - 01/Dec/20
%(converted from LaTex to Ams-TeX)
% Previous revision: 02/Jan/28
%last version: 29/9/93
%last corrections to last version: 12/93
% Original ams version moved to f9822?

\subjclass{Primary 	03E35, 03E05; Secondary: 03E50, 20K20, 20K35}

\keywords{set theory, ladder system uniformization, infinite abelian groups, Whitehead problem}
\date{13. March, 2025}

\begin{abstract}
For a given stationary set $S$ of countable ordinals we prove (in $\mathbf{ZFC}$)
that the assertion ``every $S$-ladder system has $\aleph_0$-uniformization''
is equivalent to ``every strongly $\aleph_1$-free abelian group of
cardinality $\aleph_1$ with non-freeness invariant $\subseteq S$
is $\aleph_1$-coseparable, i.e. Ext$(G, \oplus_{i=0}^{\infty} \mathbb Z)=0$ (in particular Whitehead, i.e.\ Ext$(G,  \mathbb Z)=0$)``.
This solves problems B3 and B4 from the monograph of Eklof and Mekler \cite{EM02}.
\end{abstract}

\maketitle
\numberwithin{equation}{section}
\setcounter{section}{-1}

\section{Introduction}

On the subject and all background needed see Eklof-Mekler \cite{EM02}.

By \cite{Sh:44} there may be a non-free Whitehead group of cardinality
$\aleph_1$, by \cite{Sh:64} this may occur even under $\mathbf{CH}$,
in both proofs the main idea is that the group theoretic problem is
similar to the combinatorial problem of uniformization.
This intuition materialized as a theorem stating that the existence of a non-free Whitehead ($\aleph_1$-coseparable, resp.) group of cardinality $\aleph_1$ is equivalent to the existence of a ladder system on a stationary subset of $\omega_1$ that admits $2$-uniformization ($\aleph_0$-uniformization, resp.); see \cite{Sh:98} and
\cite[XIII.0, XIII.2.]{EM02} where also a comprehensive list of problems is presented. More is done in Eklof-Mekler-Shelah \cite{Sh:441},
\cite{Sh:442}, Eklof-Shelah \cite{Sh:505}; see the updated book
Eklof-Mekler \cite{EM02}. 
In fact if a group $G$ of cardinality $\aleph_1$ is Whitehead, then it is $\aleph_1$-free (i.e.\ each countable subgroup of $G$ is free), and $\mathbf{MA} + \neg \mathbf{CH}$ proves that for an abelian group $G$ of cardinality $\aleph_1$ being a ``Shelah'' group (a strengthening of the notion of $\aleph_1$-free groups, see Definition $\ref{0.2}$) is equivalent to being Whitehead \cite[XII. 2.5, XIII. 3.6]{EM02}.
Therefore it is  natural to ask whether uniformization principles themselves together with a strengthening of $\aleph_1$-freeness  imply the Whitehead property. 
P.C. Eklof, A.H. Mekler and S. Shelah  explored the close connection of $\aleph_0$-uniformization and the characterization of  $\aleph_1$-free abelian groups of cardinality $\aleph_1$ in terms of the Whitehead property and the following important problem remained open:
\begin{namedthm*}{Problem}
	Are the following two assertions equivalent?
	\begin{enumerate}[label = ($\alph*$), ref = ($\alph*$)]
		\item Every strongly $\aleph_1$-free abelian group of cardinality $\aleph_1$ is Whitehead,
		\item every ladder system $\langle \eta_\delta: \ \delta < \omega_1 \text{ limit} \rangle$ has $\aleph_0$-uniformization.
	\end{enumerate}	
\end{namedthm*}
This problem was first stated explicitly more than 30 years ago (in fact it arose implicitly more than 40 years ago), and is the crux of the  
problems listed in the Appendix  of \cite{EM02}.
Our main theorem here settles it, and (the main case of, i.e.~, the case of a strongly $\aleph_1$-free group, of) \textbf{B3} and \textbf{B4}:

\begin{namedthm*}{B3}
	If we have, say, all uniformization results that can be deduced from $\mathbf{MA} + \neg \mathbf{CH}$, then is every strongly $\aleph_1$-free (every Shelah) group of cardinality $\aleph_1$ a Whitehead group?
\end{namedthm*}
\begin{namedthm*}{B4}
	If every strongly $\aleph_1$-free group is a Whitehead group, are they also all $\aleph_1$-coseparable?
\end{namedthm*}
Although the case of Shelah groups is certainly of interest, historically the stress has always been on strongly $\aleph_1$-free groups. Model theoretic considerations suggest that the general case of Shelah groups might be more intricate than the more natural (and more familiar to  algebraists) property  of being strongly $\aleph_1$-free \cite{eklof1974infinitary}.
We also remark that by the already cited theorem \cite[XII. 2.5, XIII. 3.6]{EM02} we cannot replace strongly $\aleph_1$-free by $\aleph_1$-free.

 In \S\ref{s1}, in the proof of Theorem \ref{2.1} the 
group theoretic problem is translated to a combinatorial one which
is proved to be equivalent to the uniformization of ladder systems in \S\ref{s2} (Theorem
\ref{1.1}). We have to remark that in some sense \S\ref{s1} is enough for answering \textbf{B3}, as $(B_S)$ from Theorem $\ref{1.1}$ is a combinatorial principle which follows from $\mathbf{MA}_{\aleph_1}$ by Claim $\ref{1.3A}$. Still it is better to show that $(A_S)$ from Theorem $\ref{1.1}$ (i.e. the ladder system uniformization with parameter $S$) can serve as a sufficient (and in fact  necessary) condition, since that is a classical, old principle, and  it is nicer to phrase. 

Moreover, in Proposition $\ref{ekp}$ we prove an equivalence between uniformization of all ladder systems $\bar \eta^*$ that are ``very similar" to a fixed ladder system $\bar \eta$ on a fixed stationary set $S$, and all strongly $\aleph_1$-free groups of cardinality $\aleph_1$ whose structure in terms of non-freeness in some sense can be described by $\bar \eta^*$ (see $(B)_{S, \bar \eta^*}$ from the proposition) have the Whitehead property: We will argue that the Whitehead property of the groups in $K_{S, \bar \eta}$ (see Definition $\ref{GpKp}$) implies $\aleph_0$-uniformization of $\bar \eta$.
This class is related to the more general (but less convenient to use) class where the equation is $k_n y_{\delta,n+1} = y_{\delta,n} + x_{\eta_\delta(n)}$ (related to the ones used  in \cite{Sh:125}).
\bigskip

\centerline {$* \qquad * \qquad *$}

\begin{notation}{\ } \\
\label{0.1}
1) Let $\alpha$, $\beta$, $\gamma, \delta$ be ordinals
(here always $<\omega_1$),
$m$, $n$, $k$, $\ell$ are natural numbers,
$\Omega$ is the set of countable limit ordinals. 

\noindent
2) For sets $u, v$ of ordinals, OP$_{u, v}$ is the following function:
OP$_{u,v}(\alpha) =\beta$ iff $\alpha\in v$,
$\beta\in u$ and otp$(\alpha\cap v)= \text{ otp}(\beta\cap u)$. 
Here we let $u$, $v$, $w$ vary on finite subsets of $\omega_1$. 

\noindent 
3) By a sequence we mean a function on an ordinal, where for a sequence $ \bar{s}= \langle s_\alpha: \ \alpha < \dom(\bar{s}) \rangle$ the length of $\bar{s}$ (in symbols $\lh(\bar{s})$) denotes $\dom(\bar{s})$. Moreover, for sequences $\bar{s}$, $\bar{t}$ let $\bar{s} \tieconcat \bar{t}$ denote the natural concatenation (of length $\lh(\bar{s}) + \lh(\bar{t})$).

\noindent
4) Let $\forall^* n$ mean  for every $n$ large enough; let us define $x \subseteq^* y$, if $|x \setminus y| < \aleph_0$,
so for the functions $f_1$, $f_2$ $f_1\subseteq^* f_2$ iff $\{x \in \text{ dom}(f_1):x \notin \text{
dom}(f_2)$ or $f_1(x) \ne f_2(x)\}$ is finite.

\noindent
5) $\bbZ $ is the additive group of integers, $\bbZ_\omega$ is 
the direct sum of $\aleph_0$-many copies of $\bbZ$.
The word group here always means abelian group.
\end{notation}

\begin{definition}{\ } \\
\label{0.2}
1) An abelian group $G$ is $\aleph_1$-free if all of its countable subgroups are free abelian groups, i.e. isomorphic to the direct sum $\oplus_{i < \mu} \bbZ$ for some cardinal $\mu$ (where of course for a countable free abelian group $\mu \leq \omega$). By classical results \cite[Thm. 19.1.]{Fu} it is equivalent to demanding that each finite rank subgroup is a free abelian group.

\noindent
2) An $\aleph_1$-free group $G$ is a Shelah group, if for every countable subgroup $H \leq G$ there exists a countable subgroup $H' \leq G$ with $H \leq H'$, which moreover satisfies that for every countable group $F \leq G$ 
\[ F \cap H' = H \ \Rightarrow \ F / H \text{ is free}. \]

\noindent
3) An $\aleph_1$-free group $G$ is strongly $\aleph_1$-free, if for each countable subgroup $H \leq G$ there exists a countable subgroup $H' \leq G$ with $H \leq H'$, and $G/H'$ being $\aleph_1$-free.

\noindent
4) $G$ is a $W$-group means $G$ is a Whitehead group 
which means: if $\bbZ\subseteq H$,
$H/\bbZ\cong G$ then $\bbZ$ is a direct summand of $H$.

\noindent
5) $G$ is $\aleph_1$-coseparable means: if $\bbZ_\omega\subseteq H$ and
$H/\mathbb Z_\omega \cong G$ then $\bbZ_\omega$ is a direct summand of
$H$ where $\bbZ_\omega$ is the direct sum of $\omega$-many copies of $\bbZ$. We call an $\aleph_1$-coseparable group also a $W_\omega$-group.
\end{definition}

\begin{definition}{\ } \\
\label{0.3}
1) For $S\subseteq\Omega$ we say that
$\bar \eta$ is an $S$-ladder system \If  \,
$\bar \eta=\langle \eta_\delta:\delta\in S\rangle$,
$\eta_\delta$ a strictly increasing $\omega$-sequence of ordinals
$<\delta$ with limit $\delta$.

\noindent
2) We say $\bar \eta$ has $\kappa$-uniformization when: \underline{for every}
$\bar c=\langle c_\delta:\delta\in S\rangle$,
$c_\delta$ a function from $\ran(\eta_\delta)$
to $\kappa$ \underline{there is} a function $c$ from $\omega_1$ to $\kappa$,
uniformizing $\bar c$, i.e., for every $\delta\in S$,
$c_\delta\subseteq^*c$, i.e.\
for every $n<\omega$ large enough $c (\eta_\delta(n))= 
c_\delta(\eta_\delta(n))$.

\noindent
3) 	We say that the ladder systems $\bar \eta^1,\bar \eta^2$ on the stationary set $S$ are very
similar \when \, $\bar \eta^\ell = \langle \eta^\ell_\delta:\delta \in S
\rangle$, and for each $\delta \in S$ for some $k_1,k_2 < \omega$ we
have 
$$\{\eta^1_\delta(k_1 + n) + \omega:n < \omega\} 
= \{\eta^2_\delta(k_2 +n) + \omega:n < \omega\}$$  
(which is equivalent to stating that
$$\{\eta^1_\delta( n) + \omega:n < \omega\} 
=^* \{\eta^2_\delta(n) + \omega:n < \omega\}.$$
)
\end{definition}

\begin{definition}\label{0.4}{\ }\\
1) We say $S$ is stationary when for every closed unbounded $C\subseteq \Omega,
S\cap C\not=\emptyset$.  

\noindent
2) Call $S$ simple if $S\subseteq 
\{\alpha<\omega_1: \omega^2 \text{ divides }\alpha\}$ (i.e. $S \cap \{ \alpha + \omega: \ \alpha < \omega_1 \} = \emptyset$). 
%(this corresponds to the case of strongly $\aleph_1$-free group).

\end{definition}

%\begin{discussion}
%We have to show that \wilog \, $\alpha = \eta_\delta(n) \Rightarrow
%\eta_\delta \rest n = \nu_\alpha$ and by induction on $k$, predict the
%future choose $\bar c_k = \langle c^k_\delta:\delta \in
%S\rangle,c^k_\delta(n)$, the $(k,n)$-type for $\delta$ and $f_n$
%uniformize it and $\langle m^k_{n,\delta}:\delta \in S\rangle$ measure
%the success and $\langle m^k_\alpha:\alpha < \omega_1\rangle$
%uniformize the $m$'s.

%Let $c^\omega_\delta = \langle \langle
%c^k_\delta(n),m^k_{\eta_\delta(n)}:k \le k_n\rangle:n < \omega\rangle$
%where $k_n = n$ or $k_n$ is maximal such that for $k \le k_n$ for
%$m^k_0 < n$.  Now we uniformize $\langle c^\omega_\delta:\delta \in
%S\rangle$.

%We may consider a depth of the uniformization frame: rk$(\delta) =
%\text{ min}\{\sup\{\text{rk}(\eta_\delta(n))+1:n \in [k,\omega)\}:k <
%\omega\}$.  Note that this function may give us $\omega_1 =
%\{\text{rk}(\delta):\delta \in S\}$.  But \wilog \, $\eta_\delta((((n)
%= \alpha \pm \nu_\alpha \cap n$.

%We may code ``future history" in $\omega$ rounds \then \, code when
%each start.
%\end{discussion}

\section{Abelian Groups}\label{s1}

Before stating our main theorem we need an array of definitions.
\begin{definition}{\ }\\
	\begin{enumerate}[label = $\alph*)$, ref = $\alph*)$]
		\item  If $G$ is an abelian group, $A\subseteq G$, we let $\cll_G(A)$ be the subgroup of $G$ that $A$ generates 	(if $G$ is fixed, we may suppress it).
		\item 	The subgroup $H \leq G$ is a pure subgroup if for every $g \in G$, $n \in \bbZ$, $ng \in H$ implies $g \in H$.
		As the groups in this paper are mostly torsion free (except for some quotient groups of the $\aleph_1$-free groups, when we divide by a non-pure subgroup), for every such group $G$ and set $A\subseteq G$
		we have that the pure closure $\clpr_G(A) := \cap \{K:A \subseteq K, K$ a 
		pure subgroup of $G\}$ of $A$ is equal to  
		$\{x:$ for some $n>0$, $nx \in \ \ \cll_G( A)\}$.
		\item
		For $K_1,K_2\subseteq G$, let $K_1+K_2=\{x+y:x\in K_1, y\in K_2\}$.
		If $K_1\cap K_2=\{0\}$ we write $K_1\oplus K_2$.
	\item  For $K \leq G$ we say that $\{g_i: \ i \in I\} \subseteq G$ is independent over $K$ if there exists no non-trivial $\bbZ$-linear combination of $\{g_i + K: \ i \in I\}$ being equal to $0$ in the quotient group $G/K$. 
	\item We call a set independent if it is independent over the trivial subgroup $\{0\}$.
	\item If $G$ is an abelian group of cardinality $\aleph_1$, then we say $\bar G' = \langle G'_i: \ i < \omega_1 \rangle$ is a filtration of $G$, when
	\begin{enumerate}[label = $f\arabic*)$, ref = $f\arabic*)$]
		\item $\bar G' = \langle G'_i: \ i < \omega_1 \rangle$ is increasing, continuous,
		\item for each $\alpha$ $G'_\alpha$ is a countable pure subgroup of $G$,
		\item $G = \bigcup_{ \alpha < \omega_1} G'_\alpha$.
	\end{enumerate}
	\end{enumerate}

\end{definition}

\begin{definition}\label{ABCdf}
	For a fixed stationary set $S$ we will consider the following assertions.
	\begin{enumerate}[label = $(\Alph*)_S \equiv$, ref = $(\Alph*)_S$]
		\item \label{AS}   every $S$-ladder 
		system $\langle \eta_\delta:\delta \in S \rangle$ 
		has $\aleph_0$-uniformization,
		\sn
		\item \label{BS}  if $G$, $\bar G'$ satisfy $(*)_S(G, \bar G')$ below \then \, 
		$G$ is a Whitehead group ($W$-group in short)
		\begin{enumerate}
			\item[$(*)_S(G, \bar G')$:] 
			
			\begin{enumerate}[label = (\greek*)$^{*}$, ref = (\greek*)$^{*}$]
				\item  $G$ is an abelian group of cardinality $\aleph_1$,
				\sn
				\item  $G$ is $\aleph_1$-free, i.e. every 
				countable subgroup of it is free,
				\sn
				\item \label{gam} $\bar G' = \langle G'_i: \ i < \omega_1 \rangle$ is a filtration of $G$, and
				\[
				\{\delta\in \Omega: G/G'_\delta \text{ is not }
				\aleph_1 \text{-free}\} \subseteq_{\text{NS}} S
				\]
				(which means that the set $\{ \delta\in \Omega: G/G'_\delta \text{ is not }
				\aleph_1 \text{-free}\}$ is contained in $S$ up to a non-stationary set),
				\item $G$ is strongly $\aleph_1$-free, i.e.,
				for every countable $H \leq G$ there is a countable subgroup $H' \leq
				G$ extending $H$ such that $G/H'$ is $\aleph_1$-free (follows from
				$(\gamma)$ if $\omega_1 \backslash S$ is stationary),
			\end{enumerate}
		\end{enumerate}
		\item \label{CS}   if $G$, $\bar G'$ satisfy $(*)_S(G, \bar G')$ above, \then \,
		$G$ is $\aleph_1$-coseparable (one of the equivalent 
		definitions is Ext$(G,\bbZ_\omega) = 0$ where 
		$\bbZ_\omega$ is the direct sum of $\omega$ copies of $\bbZ$, i.e. $G$ is a W$_\omega$-group).
	\end{enumerate}
	
\end{definition}

%Note that if we restrict ourselves to strongly $\aleph_1$-free abelian
%groups then simple ${\gp}$ (in \ref{1.1}) suffice.
\begin{theorem}
	\label{2.1}
	Let $S$ be a stationary subset of $\omega_1$. Then
		\[ \text{ \ref{AS} } \iff \text{ \ref{BS} } \iff \text{ \ref{CS}. } \]
	
\end{theorem}
\noindent Before proving this we state its immediate corollary:
\begin{cor}
	The following are equivalent:
	\begin{enumerate}[label = ($\alph*$), ref = ($\alph*$)]
		\item every ladder system $\langle \eta_\delta: \ \delta < \omega_1 \text{ limit} \rangle$ has $\aleph_0$-uniformization.
		\item Every strongly $\aleph_1$-free abelian group of cardinality $\aleph_1$ is Whitehead,
		
	\end{enumerate}	
\end{cor}

\noindent However, unfortunately the following falls out of the scope of the present work:
\begin{problem}
	Are the following two assertions equivalent?
	\begin{enumerate}[label = ($\alph*$), ref = ($\alph*$)]
		\item Every ladder system $\langle \eta_\delta: \ \delta < \omega_1 \text{ limit} \rangle$ has $\aleph_0$-uniformization.
		\item Every  Shelah group of cardinality $\aleph_1$ is Whitehead,
	\end{enumerate}	
\end{problem}

\begin{PROOF}{Theorem \ref{2.1} \ assuming \ Theorem \ref{1.1}}(Theorem \ref{2.1})  Note that \ref{CS}$ \Rightarrow $\ref{BS} is immediate.   For the implication \ref{BS}$ \Rightarrow $\ref{AS} see \cite[XIII. Proposition 2.9.]{EM02}. However, this citation only proves that \ref{BS} implies that  every $S$-ladder 
	system $\langle \eta_\delta:\delta \in S \rangle$ 
	has $2$-uniformization, and the statement that 
	$$  (\textrm{every }S \textrm{-ladder 
	system }\langle \eta_\delta:\delta \in S \rangle
	\textrm{ has }2\textrm{-uniformization}) \Rightarrow $$
	 $$ \Rightarrow  (\textrm{every }S \textrm{-ladder 
		system }\langle \eta_\delta:\delta \in S \rangle
	\textrm{ has }\aleph_0\textrm{-uniformization})$$
	is only an exercise, so a possible reference for this is \cite[Lemma 1.4]{Sh:80}. Because of some inaccuracies in this Lemma, here (in Lemma \ref{megl}) we provide a cleaned-up proof for the sake of completeness.
	
	When we will have finished Lemma \ref{megl} we  turn to our main theorem, i.e.~ that \ref{AS} implies \ref{CS}.
	Towards this Proposition \ref{ekp} is the key, and we will need a sequel of lemmas to be able to reduce the task to the proposition. After having proved all the ingredients for the second clause of Proposition \ref{ekp} and put them together (right before Definition \ref{GpKp}), we will argue that \ref{AS} implies \ref{CS}.
	\begin{lemma}{\cite[Lemma 1.4]{Sh:80}} \label{megl}
		Suppose that $S \subseteq \Omega$ is stationary. If every $S$-ladder system has $2$-uniformization, then every $S$-ladder system has $\aleph_0$-uniformization.
	\end{lemma}
	\begin{PROOF}{Lemma \ref{megl}}
	    Let $\langle \eta^0_\delta: \ \delta \in S \rangle$ be an $S$-ladder system, $\langle c^0_\delta: \ \delta \in S\rangle$ be a coloring with $\aleph_0$-many colors, so $c^0_\delta$ is defined on $\{\eta^0_\delta(i): \ i \in \omega\}$. We define a ladder system $\bar \eta^1 = \langle \eta^1_\delta: \ \delta \in S \rangle$ and a coloring
	    $\langle c^1_\delta: \ \delta \in S\rangle$ with two colors.
	    Define the function $\varp: \omega_1 \times \omega \to \omega_1$ so that it is an injection, and $\varp(\alpha,n) + \omega = \alpha + \omega$, i.e.~ for each limit $\beta$, $\varp(\beta+m,n)$ is of the form $\beta + l$ for some finite $l$,  moreover $\varp(\alpha,n) < \varp(\alpha,n+1)$ for each $\alpha$, $n$.
	    Fix a $\delta \in S$, we are going to construct $\eta^1_\delta$. Define $k_0 = c^0_\delta(\eta^0_\delta(0))+1 \in \omega$, and let 
	    $$\begin{array}{rcl}
	    	\eta^1_{\delta}(0)&  = &  \varp(\eta^0_{\delta}(0),0),\\
	    	\eta^1_{\delta}(1) & = & \varp(\eta^0_{\delta}(0),1), \\
	    	\vdots & & \\
	    	\eta^1_{\delta}(k_0-1) & =& \varp(\eta^0_{\delta}(0),k_0-1),
	    \end{array}$$
	   and 
	    $c^1_\delta(\eta^1_{\delta}(j)) = 0$ for $j < k_0-1$, $c^1_\delta(\eta^1_{\delta}(k_0-1)) = 1$.
	     We define the sequence $k_i$ by induction with the rule $k_{i+1} = k_i + c^0_\delta(\eta^0_\delta(i+1))+1$, and parallel to that the function $\eta^1_\delta \rest [k_i,k_{i+1})$, and $c^1_\delta(\eta^1_\delta(\ell))$ for $k_i \leq \ell < k_{i+1}$ as follows. Let
	    $\eta^1_{\delta}(k_i+j) =  \varp(\eta^0_{\delta}(i+1),j)$ for $0 \leq j \leq c^0_\delta(\eta^0_\delta(i+1))$. As for the coloring, set 
	    \begin{equation} \label{c1D}
	    	\begin{array}{lll}
	    		c^1_\delta(\eta^1_{\delta}(k_i+j)) & = 0 & \textrm{ for }j < c^0_\delta(\eta^0_\delta(i+1)), \\
	    		c^1_\delta(\eta^1_{\delta}(k_i+j) &= 1 &  \textrm{ for }j =k_{i+1}-k_i - 1= c^0_\delta(\eta^0_\delta(i+1)))) . 
	    	\end{array}	
	    \end{equation} 
	    
	    Suppose that the function $d^1: \omega_1 \to \{0,1\}$ is a uniformization for $\langle c^1_\delta: \ \delta \in S \rangle$. We let $d^0(\alpha)$ be 
	    \begin{equation} \label{d0} d^0(\alpha)= \min \{n \in \omega: \ d^1(\varp(\alpha,n))=1\} \end{equation} (or $d^0(\alpha) = 0$, if the set in question is empty).
	    Fix a $\delta \in S$, suppose that for all $\ell \geq m$ $d^1(\eta^1_\delta(\ell)) = c^1_\delta(\eta^1_\delta(\ell))$ holds, and let $i \in \omega$ be minimal such that $k_{i} \geq m$ (where $k_i$ is the auxiliary sequence used for constructing $\eta^1_\delta$ and $c^1_\delta$). We claim that for $n > i$ the equality $c^0_\delta(\eta^0_\delta(n)) = d^0(\eta^0_\delta(n))$ holds true. So fix such an $n$, and let $\alpha = \eta^0_\delta(n)$, hence 
	    \begin{equation}\label{laddef} \eta^1_\delta(k_{n-1} + j) = \varp(\alpha,j) \textrm{ for }j \leq c^0(\alpha) = c^0(\eta^0_\delta(n)). \end{equation} Note that $m \leq k_{n-1}$, so $d^1(\eta^1_\delta(\ell)) = c^1_\delta(\eta^1_\delta(\ell))$
	    for $\ell \in [k_{n-1},k_n)$, which means that
	    $$(\forall j \leq k_{n}-k_{n-1}-1 (= c^0_\delta(\eta^0_\delta(n)))): \ d^1(\varp(\alpha,j)) = c^1_\delta(\varp(\alpha,j)).$$
	    Recalling \eqref{laddef} we obtain $d^1(\varp(\eta^0_\delta(n),j)) = d^1(\eta^1_\delta(k_{n-1}+j)) = c^1_\delta(\eta^1_\delta(k_{n-1}+j))$ for $j \leq c^0_\delta(n)$, so by \eqref{c1D} this translates to $d^1(\varp(\eta^0_\delta(n),j)) = d^1(\eta^1_\delta(k_{n-1}+j)) = 0$ if $j <c^0_\delta(\eta^0_\delta(n))$, and $d^1(\varp(\eta^0_\delta(n),c^0_\delta(\eta^0_\delta(n)))) = d^1(\eta^1_\delta(k_{n-1}+c^0_\delta(\eta^0_\delta(n)))) = 1$. Therefore, by the way we defined $d^0$ \eqref{d0}, $d^0(\eta^0_\delta(n)) = c^0_\delta(\eta^0_\delta(n))$, as desired.
	\end{PROOF}

	\begin{proposition} \label{ekp} Fix a stationary set $S \subseteq \omega_1$:
		\begin{enumerate}[label = $(\roman*)$, ref = (\roman*)]
			\item \label{eki0} If   $\bar \eta = \langle \eta_\delta: \ \delta \in S \rangle$ is a ladder system on $S$ and $\bar \eta$ satisfies $\ran(\eta_\delta) \subseteq \Omega$ for each $\delta \in S$  (so necessarily $S$ is simple),  $\bar \eta^*$ is very similar to $\bar \eta$, then  $$(A)_{S,\bar \eta} \ \Rightarrow \ (A)_{S, \bar \eta^*},$$
			\item \label{eki}
			if   $\bar \eta = \langle \eta_\delta: \ \delta \in S \rangle$ is a ladder system on $S$ and $\bar \eta$ satisfies $\ran(\eta_\delta) \subseteq \Omega$ for each $\delta \in S$,
			 and the ladder system $\bar \eta^*$ is very similar to $\bar \eta$, then  $$(A)_{S,\bar \eta} \ \Rightarrow \ (C)_{S, \bar \eta^*},$$
			\item \label{ekii} for every ladder system $\bar \eta^*$ on $S$
			$$ (B)_{S, \bar \eta^*} \ \Rightarrow  \ (A)_{S, \bar \eta^*},$$
		\end{enumerate}
		\where \ under $(A)_{S,\bar \eta}$, $(B)_{S,\bar \eta^*}$, $(C)_{S,\bar \eta^*}$  we mean the following:
		\begin{enumerate}
			\item[$(A)_{S,\bar \eta}\equiv$]    the ladder 
				system $\langle \eta_\delta:\delta \in S \rangle$ 
				has $\aleph_0$-uniformization,
			\item[$(B)_{S, \bar \eta^*}\equiv$]    whenever $G$ and the ladder system $\bar \eta^* = \langle \eta^*_\delta: \ \delta \in S \rangle$ satisfy $(*)^+_{S, \bar \eta^*}(G, \bar G')$   below for some $\bar G'$ \then \, 
				$G$ is a $W$-group,
					\begin{enumerate}
					\item[$(*)^+_{S, \bar \eta^*}(G, \bar G')$:] \label{S+eta} $\quad$   
					\begin{enumerate}[label = (\greek*)$^+$, ref = (\greek*)$^+$]
						\item \label{alp0}  $G$ is an abelian group of cardinality $\aleph_1$,
						\sn
						\item  $G$ is $\aleph_1$-free, i.e. every 
						countable subgroup of it is free,
						\sn
						\item \label{gamp0}  the sequence $\bar G' = \langle G'_\alpha: \alpha < \omega_1 \rangle$ is a filtration of $G$,
						where for each $\alpha < \omega_1$ $G'_{\alpha+1}/G'_\alpha$ is of rank $1$, and						
						\[
						\{\delta\in \omega_1: G/G'_\delta \text{ is not }
						\aleph_1 \text{-free}\} \subseteq S,
						\]
						
						\item \label{delp0} $G$ is strongly $\aleph_1$-free, i.e.,
						for every countable $H \leq G$ there is a countable $H' \leq
						G$ extending $H$ such that $G/H'$ is $\aleph_1$-free (follows from
						$(\gamma)^+$ if $\omega_1 \backslash S$ is not bounded in $\omega_1$),
						\item \label{varpp0}for each $\delta \in S$ for some sequence $\langle y_{\delta+\ell}: \ \ell < \omega \rangle \in \ ^\omega G'_{\delta+\omega}$ which is  a maximal independent sequence over $G'_\delta$, we have for each $n \in \omega$:
						\[  \left\{ \begin{array}{ll} \alpha<\delta: & \clpr\left( G'_{\alpha+1} \cup
							\{y_{\delta+\ell}:\ell<n\} \right) \ne \\ &
							\clpr\left( G'_\alpha \cup
							\{ y_{\delta+\ell}: \ \ell < n \} \right) +G'_{\alpha+1} \end{array} \right\} \subseteq^* \ran(\eta^*_\delta).\]
					\end{enumerate}
					\end{enumerate}
				\item[$(C)_{S, \bar \eta^*} \equiv$]    if $G$ satisfies $(*)^+_{S, \bar \eta^*}(G, \bar G')$ above for some $\bar G'$  \then \, 	$G$ is a $W_\omega$-group,
			
		\end{enumerate}
	\end{proposition}
	\begin{remark}
			Probably we cannot state the alternative version of Proposition $\ref{ekp}$ defining $(A)_{S,\bar \eta}$, $(B)_{S,\bar \eta}$ modulo nonstationary sets. 
			However, $(A)_{S,\bar \eta}$ $\iff$ $(A)_{S',\bar \eta'}$ if $(S, \bar \eta)$, $(S', \bar \eta')$ are club similar, see the definition below.
	\end{remark}
	\begin{definition}
		We call $(S^0, \bar \eta^0)$, $(S^1, \bar \eta^1)$ (where $S^0, S^1$ are stationary subsets of $\omega_1$, $\bar \eta^i$ is a ladder system on $S^i$) club similar, when 
		there exists a club set $C$, and a bijection $\pi : \omega_1 \to \omega_1$ such that
		\begin{enumerate}
			\item $C \cap S^0 = C \cap S^1$,
			\item $\pi \rest (C \cap S^0) = \text{id} \rest (C \cap S^0)$,
			\item for each $\delta \in \omega_1$ we have $\delta \in S^0$ $\iff$ $ \pi(\delta) \in S^1$, and if they hold, then
			\[ \{ \pi(\eta^0_\delta(n)): \ n \in \omega\} =^* \{ \eta^1_{\pi(\delta)}(n): \ n \in \omega\}.\]
		\end{enumerate}
	\end{definition}
	Lemma \ref{verisi} proves \ref{eki0}.
	Claim $\ref{fil1}$, Claim $\ref{hulyetekkl}$, Claim $\ref{bef}$ together will prove clause \ref{eki} in Proposition $\ref{ekp}$. (After finishing the proof of Claim $\ref{bef}$, i.e. having proved Lemma \ref{befej}, we will put the pieces together and formally verify \ref{eki}.)
	
	After that we will turn back to the proof of \ref{ekii} at the end of the section (starting with Definition \ref{GpKp}, concluding with Lemma \ref{megold2}).
	\begin{lemma} \label{verisi}
		If $S \subseteq \omega_1 \cap \Omega$ is a stationary set, and the ladder system $\bar \eta = \langle \eta_\delta: \ \delta \in S \rangle$ is such that for every $\delta \in S$ $\ran(\eta_\delta) \subseteq \Omega$, and $\bar \eta$ has $\aleph_0$-uniformization, then every ladder system $\bar \eta^*$ very similar to $\bar \eta$ has $\aleph_0$-uniformization.
	\end{lemma}
	\begin{PROOF}{Lemma \ref{verisi}}
	
		First observe that since $\bar \eta$ and $\bar \eta^*$ are very similar, for each $\delta$ for large enough $n$ (depending on $\delta$) we have $\eta_\delta(n))+ \omega = \eta^*_\delta(n^*) + \omega$ for some $n^*$ (depending on $n$ and $\delta$ of course),
		moreover, since $\eta_\delta(n)$ is limit this means that $\eta^*_\delta(n^*) = \eta_\delta(n)) + l$ for some $l \in \omega$.
		In fact, for such $\delta$, for some $N_\delta$, there exists a strictly increasing sequence $\langle N^*_{\delta,n}: \ n \geq N_\delta \rangle$  such that 
		$$\eta_\delta(n))+ \omega =  \eta^*_\delta(n^*) + \omega \text{ whenever } n^* \in [N^*_{\delta,n}, N^{*}_{\delta,n+1}),$$
		hence for all $(\forall n \geq N_\delta)$ for some $l_0<l_1 < \dots < l_{N^*_{\delta,n+1}-N^*_{\delta,n}-1}$
		 \begin{equation} \label{eggy}  \left.\begin{array}{rl}
		 	\eta_\delta(n))+ l_0  = & \eta^*_\delta(N^*_{\delta,n}),  \\
		 	 \eta_\delta(n))+ l_1 = & \eta^*_\delta(N^*_{\delta,n} +1) ,  \\
		 	 \dots & \\
		 	 \eta_\delta(n))+  l_{N^*_{\delta,n+1}-N^*_{\delta,n}-1}  = & \eta^*_\delta(N^*_{\delta,n+1}-1) .
		 	 \end{array} \right\} (\forall n \geq N_\delta)
		 \end{equation}
		 Now define the coloring $\langle c_\delta: \ \delta \in S  \rangle$ of the ladder system $\bar \eta$ by 
		 $$c_\delta(\eta_\delta(n)) = \{ l \in \omega: \ \eta_\delta(n)+l \in \ran(\eta^*_\delta) \} \in [\omega]^{<\omega},$$
		 and fix $\langle H_\alpha: \ \alpha < \omega_1 \rangle$, $\langle m_\delta: \ \delta \in S \rangle$ such that for every $\delta \in S$, $\forall n \geq m_\delta$ $c_\delta(\eta_\delta(n)) = H_{\eta_\delta(n)}$. We may as well assume that $m_\delta \geq N_\delta$, so that 
		 \begin{equation} \label{kettto} \emptyset \neq \{ l \in \omega: \ \eta_\delta(n)+l \in \ran(\eta^*_\delta) \} = H_{\eta_\delta(n)}  \ (\forall n \geq m_\delta) \end{equation}
		
		Suppose that $\langle c'_\delta: \ \delta \in S  \rangle$ is a coloring of $\bar \eta^*$ with $\aleph_0$-many colors, define the coloring $c''_\delta: \ran(\eta_\delta \rest (\omega \setminus m_\delta))$ on $\bar \eta$ by
		\begin{equation} \label{kozep} \begin{array}{ll} c''_\delta(\eta_\delta(n)) = \langle c'_\delta(\eta_\delta(n) + l_j): \ j <  |H_{\eta_\delta(n)}| \rangle, \\ \text{where } H_{\eta_\delta(n)} = \{ l_0 < l_1 < \dots < l_{|H_{\eta_\delta(n)}|-1}\}. \end{array}  \end{equation}
		(Note that, since $n \geq m_\delta$, $\eta_\delta(n) + l \in \dom(c'_\delta) = \ran(\eta^*_\delta)$ for all $l \in H_{\eta_\delta(n)}$.)
		It remains to check that a uniformization of $\bar c''$ gives rise to a uniformization of $\bar c'$.
		Let the function $c''$ uniformize $\bar c''$; we define $c'$ as follows. The domain of $c'$ is defined for any limit $\beta$  so that
		$$\dom(c') \cap [\beta, \beta+\omega) = \{ \beta +l_0,\beta+ l_1, \dots \beta+ l_{i-1} \},$$
		where $H_\beta = \{l_0,l_1, \dots, l_{i-1}\}$ (possibly $\dom(c') \cap [\beta, \beta+\omega) = \emptyset$).
		Set $c'(\beta+l)$'s so that
		\begin{equation} \label{veg}
		  c''(\beta) = \langle c'(\beta+l_j): \ j < |H_\beta| \rangle 
		\end{equation}
		hold (and so $c''(\beta)$ is a finite sequence).
		
		 Fix $\delta \in S$, (recalling that $c''$ is a uniformization for $\bar c''$) set $m''_\delta \geq m_\delta$ so that
		\begin{equation} \label{m''delta} \text{for }n \geq m''_\delta: \ c''(\eta_\delta(n)) = c''_\delta(\eta_\delta(n))% = \langle c'_\delta(\eta_\delta(n) + l_j): \ j <  |H_{\eta_\delta(n)}| \rangle. 
		\end{equation}
		
	We claim that setting $m'_\delta = N^*_{\delta, m''_\delta}$, whenever $m$ satisfies $m \geq m'_\delta$, then $c'(\eta^*_\delta(m)) = c'_\delta(\eta^*_\delta(m))$, which will finish the proof of the Lemma. (Note that, recalling \eqref{eggy}, since $m''_\delta \geq m_\delta$, and $m_\delta \geq N_\delta$, we clearly have $m''_\delta \geq N_{\delta}$, so $N^*_{\delta, m''_\delta}$ is defined.)
	Fix $m \geq m'_\delta$, let $n \geq m''_\delta$ be the unique natural number for which
	$$ m \in [N^*_{\delta, n}, N^*_{\delta,n+1}).$$
	It suffices to show that
	\begin{equation} \label{goa} \begin{array}{r} \langle c'_\delta(\eta^*_\delta(N^*_{\delta, n})), c'_\delta(\eta^*_\delta(N^*_{\delta, n}+1)), \dots, c'_\delta(\eta^*_\delta(N^*_{\delta, n+1}-1)) \rangle =  \\ \langle c'(\eta^*_\delta(N^*_{\delta, n})),  c'(\eta^*_\delta(N^*_{\delta, n}+1)), \dots ,  c'(\eta^*_\delta(N^*_{\delta, n+1}-1)) \rangle. \end{array} \end{equation}
	Now, since $n \geq N_\delta$, recalling \eqref{eggy} we obtain
	that for some $l_0<l_1 < \dots < l_{N^*_{\delta,n+1}-N^*_{\delta,n}-1}$:
	\begin{equation} \label{hhar} \begin{array}{c}  \eta_\delta(n))+ l_j =  \eta^*_\delta(N^*_{\delta,n} +j) \ \ (\forall j < N^*_{\delta,n+1}-N^*_{\delta,n}), \\
		\text{ and } \\
		\{l \in \omega: \ \eta_\delta(n)+l \in \ran(\eta^*_\delta)\} = \{l_0,l_1, \dots, l_{N^*_{\delta,n+1}-N^*_{\delta,n}-1} \}.
		\end{array}
		 \end{equation}
		 	Now \eqref{hhar} together with \eqref{kettto} (and $n \geq m_\delta$) give us
		 \begin{equation} \label{neeg} H_{\eta_\delta(n)} = \{l_0,l_1, \dots, l_{N^*_{\delta,n+1}-N^*_{\delta,n}-1} \}. \end{equation}
		 Then \eqref{goa} (what we want to verify) is equivalent to
		 $$ \langle c'_\delta(\eta_\delta(n) + l_j): \ j < |H_{\eta_\delta(n)}| \rangle = \langle  c'(\eta_\delta(n) + l_j): \ j < |H_{\eta_\delta(n)}| \rangle,$$
		 which is equivalent to (reformulating both the left-hand side and the right-hand side by \eqref{kozep} and \eqref{veg}):
		 \begin{equation} \label{cel}  c''_\delta(\eta_\delta(n)) = c''(\eta_\delta(n)). \end{equation}
		 Finally, remembering that $c''$ uniformizes $\bar{c}''$, $n \geq m''_\delta$, \eqref{m''delta} implies that \eqref{cel} holds true (and so \eqref{goa}, too), as desired.
	\end{PROOF}

	%Let $G$ be given satisfying \ref{BS} from Theorem $\ref{2.1}$, Definition $\ref{ABCdf}$).	Fix a filtration $\langle  G^*_\alpha: \ \alpha < \omega_1 \rangle$ witnessing \ref{BS} / \ref{gam} with $G^*_0 = \{0\}$.
	\newcounter{bocou}
	\setcounter{bocou}{0}
%	\begin{enumerate}[label = $\boxplus_{\arabic*}$, ref = $\boxplus_{\arabic*}$]
%		\setcounter{enumi}{\value{bocou}}
%		\item \label{bo1} for some club $C^* \subseteq \omega_1$ if $\alpha \notin S \cup C^*$, %then $G/G^*_\alpha$ is $\aleph_1$-free.
%		\stepcounter{bocou}
%	\end{enumerate}

	 The following will be useful for the proof of Theorem  \ref{2.1} (as Proposition \ref{ekp} has stronger premises).
	\begin{claim} \label{fil1}	
			 Let $G$ be a group, $\bar G^*$, $S$  satisfy $(*)_S(G, \bar G^*)$ from Definition $\ref{ABCdf}$, and suppose  that $C^0 \subseteq \omega_1$ is a club, \then \  $(**)_{S \cap C^0\setminus \{\alpha+ \omega: \ \alpha <  \omega_1 \}}(G, \bar G')$ holds for some $\bar G'$, where the symbol $(**)_{S}(G, \bar G')$ denotes the  following assertion:
		\begin{enumerate}
			\item[$(**)_S(G, \bar G')$:] $\quad$ 
			\begin{enumerate}[label = (\greek*)$^{**}$, ref = (\greek*)$^{**}$]
				\item \label{alp}  $G$ is an abelian group of cardinality $\aleph_1$,
				\sn
				\item  $G$ is $\aleph_1$-free, i.e. every 
				countable subgroup of it is free,
				\sn
				\item \label{gamp} the sequence $\bar G' = \langle G'_i:i<\omega_1\rangle$ is a filtration of $G$ for which
				\[
				\{\delta\in \Omega: G/G'_\delta \text{ is not }
				\aleph_1 \text{-free}\} \subseteq S,
				\]
				and for each $\alpha < \omega_1$ $G_{\alpha+1}/G_\alpha$ is of rank $1$,
				\item \label{delp} $G$ is strongly $\aleph_1$-free.
			
	    	\end{enumerate}
		\end{enumerate}
	\end{claim}
	\begin{PROOF}{Claim \ref{fil1}}
	We can choose the sequence $\langle  G^0_\alpha: \ \alpha \in \omega_1 \cap (\Omega \cup \{0\}) \rangle$ such that:
	\mn

	\begin{enumerate}[label = $\boxplus_{\arabic*}$, ref = $\boxplus_{\arabic*}$]
		\setcounter{enumi}{\value{bocou}}
		\item \label{G0}
		\begin{enumerate}[label = (\alph*), ref = (\alph*)]
			\item $G^0_0 = \{0\}$,
			\item each $G_\alpha^0$ is a pure subgroup of $G$,
			\item \label{b} for each $\alpha$ the quotient group $G/G^0_{\alpha+\omega}$ is $\aleph_1$-free (hence strongly $\aleph_1$-free), and $G^0_{\alpha+\omega}/G^0_\alpha$ is of rank $\aleph_0$,
			\item \label{c} the sequence $\langle G^0_\alpha: \ \alpha \in \omega_1 \cap \Omega \rangle$ is continuous, increasing, $\bigcup_{\alpha \in \omega_1 \cap \Omega} G_\alpha^0 = G$.
			
		\end{enumerate}
		\stepcounter{bocou}
	\end{enumerate}
	
	\begin{enumerate}[label = $\boxplus_{\arabic*}$, ref = $\boxplus_{\arabic*}$]
		\setcounter{enumi}{\value{bocou}}
		\item \label{assC} So for a suitable club  $C$ of $\omega_1$ with $0 \in C$, for each $i \in C$ we have $G_i^0 = G_i^*$, and
		\begin{equation}\label{afree} \forall \alpha \in C \setminus S: \ G/G_\alpha^0 = G / G_\alpha^* \text{ is }\aleph_1\text{-free} \end{equation}
		(as $\langle G^*_i: \ i \in \omega_1 \rangle$ was chosen to be a witness of $(*)_S(G,\bar G^*)$). Moreover, we can assume that
		\begin{equation} \label{afraa} \{ \alpha+\omega: \alpha <\omega_1 \} \cap C = \emptyset, \end{equation}
		i.e. each element of $C$ is divisible by $\omega^2$,
		and 
		 \[ C \subseteq C^0.\]
			 
		\stepcounter{bocou}

	\end{enumerate}
	We are going to define the filtration $\bar G'= \langle G'_i: \ i \in \omega_1 \rangle$ witnessing $(**)_{S \cap C}(G,\bar G')$ with $G^*_\alpha = G^0_\alpha = G'_\alpha$ for each $\alpha \in C$. $\eqref{afree}$ and $\eqref{afraa}$ imply that this suffices.
%	Then it is straightforward to check that
%	\begin{enumerate}[label = $\boxplus_{\arabic*}$, ref = $\boxplus_{\arabic*}$]
%		\setcounter{enumi}{\value{bocou}}
%		\item 
%		the filtration $\bar G^0 = \langle G^0_i: \ i \in \omega_1 \rangle$ satisfies the condition $(\gamma)$ from $(*)_S(G)$.
%		\stepcounter{bocou}
%	\end{enumerate}

	For each $\alpha \in C$ let $\{y_{\alpha+n} : \ n \in \omega \} \subseteq G^0_{\alpha+\omega}$ be a maximal independent set over $G^0_\alpha$ so that 
	\begin{enumerate}[label = $\boxplus_{\arabic*}$, ref = $\boxplus_{\arabic*}$]
		\setcounter{enumi}{\value{bocou}}
		\item \label{base} if $G^0_{\alpha+\omega} / G^0_\alpha$ is free, then $\{ y_{\alpha+n} + G^0_\alpha: \ 		n \in \omega \}$ is a basis of $G^0_{\alpha+\omega} / G^0_\alpha$.
		\stepcounter{bocou}
	\end{enumerate}
	
	Now letting $\text{succ}_C(\alpha) = \min(C \setminus (\alpha+\omega))$ 
	\begin{enumerate}[label = $\boxplus_{\arabic*}$, ref = $\boxplus_{\arabic*}$]
		\setcounter{enumi}{\value{bocou}}
		\item \label{ass--} we choose (for each $\alpha \in C$) $\{ y_{\gamma}: \ \gamma \in \left[\alpha+\omega, \text{succ}_C(\alpha) \right) \}$ so that $\{ y_{\gamma} + G^0_{\alpha+\omega}: \ \gamma \in \left[\alpha+\omega, \text{succ}_C(\alpha) \right) \}$ forms a basis of the free group $G^0_{\text{succ}_C(\alpha)} / G^0_{\alpha+\omega}$, moreover (as $\suc_C(\alpha) $ is not of the form $\beta + \omega$ for any $\beta$ by \ref{assC}),
		%\begin{enumerate}[label = $\boxplus_{5a}$, ref = $\boxplus_{5a}$]
			\item[$\boxplus_{5a}$] for some strictly increasing sequence $\langle \delta_n: \ n \in \omega \rangle$ in $[\alpha+\omega, \suc_C(\alpha))$ cofinal in $\suc_C(\alpha)$ with $\delta_0 = \alpha$ (equivalently, $\langle \delta_n+ \omega: \ n \in \omega \rangle \in \ ^\omega[\alpha+\omega+\omega, \suc_C(\alpha))$ is cofinal) we demand also that
			\[ \{ y_{\gamma} +  G^0_{\delta_n+\omega}: \ \gamma \in \left[\delta_n+\omega, \delta_{n+1}+\omega \right) \} \text{ is a base of } G^0_{\delta_{n+1}+\omega} / G^0_{\delta_n+\omega} \]
		%\end{enumerate}
		\stepcounter{bocou}
	\end{enumerate}
	
	Now it is easy to see by induction on $\alpha \in C$, that
	\begin{enumerate}[label = $\boxplus_{\arabic*}$, ref = $\boxplus_{\arabic*}$]
		\setcounter{enumi}{\value{bocou}}
		\item \label{ass-} for each $\alpha \in C$ the set $\{y_j: \ j < \alpha\}$ is a maximal independent set in $G^0_\alpha$, and as $G^0_\alpha$ is a pure subgroup of $G$, necessarily $G^0_\alpha = \ \clpr (\{y_j: \ j < \alpha \} )$, moreover from the choice of $\{y_j: \ \alpha \leq j < \alpha +\omega\}$, $\alpha \in C$ implies that $G^0_{\alpha+\omega} = \ \clpr(\{ y_j: \ j < \alpha+\omega \})$,
		
		\stepcounter{bocou}
		\item \label{asse} letting $G'_i = \ \clpr( \{ y_j: \ j < i \} )$ ($i < \omega_1$) it is easy to see that
		the sequences $G'_i$, $y_i$ ($i < \omega_1$) satisfy the following.
		\begin{enumerate} [label = (\alph*), ref = (\alph*)]
			\item \label{assaa}	 $G = \bigcup\limits_{i<\omega_1} G'_i$, and $\langle G'_i: \ i < \omega \rangle$ 
			is an increasing continuous chain of countable pure subgroups of $G$ such that
			$G'_0=\{0\},G'_{i+1}/G'_i$ is of rank 1,
			\item  $\{y_i:i<\alpha\}$ is a 
			maximal independent family in $G'_\alpha$ for every ordinal $\alpha<\omega_1$,
			\item \label{assb} for each $i \in C$ we have $G'_i = G_i^0 = G_i^*$, and $G'_{i+\omega} = G^0_{i+\omega}$,
			\item \label{assc} for each limit ordinal $\alpha \notin S \cap C$ the quotient	$G'_{\omega_1}/G'_{\alpha}$  is $\aleph_1$-free (hence for every $\alpha = \beta+\omega$ as the elements of $C$ are divisible by $\omega^2$ \ref{assC}),

		\end{enumerate}
		\stepcounter{bocou}
	\end{enumerate}
	\begin{PROOF}{}(\ref{asse}) \ref{assaa} follows from the definition of the $G'_i$'s, while \ref{assb} follows from \ref{ass-}. Now for \ref{assc} fix $\alpha<\omega_1$ limit, $\alpha \notin S \cap C$. If $\alpha \in C$, then $G'_{\alpha+\omega} = G^0_{\alpha+\omega}$ by \ref{assb} (as well as $G'_{\alpha} = G^0_{\alpha}$), and using $\alpha \notin S$ and $\eqref{afree}$ we are done. If $\alpha\in [\beta+\omega , \suc_C(\beta))$ for some $\beta \in C$ (so $\alpha+\omega \in [\beta+\omega , \suc_C(\beta))$ by $\eqref{afraa}$), then recalling the construction of $y_\gamma$'s for $\gamma \in [\beta+\omega , \suc_C(\beta))$ in \ref{ass--}  there is some $\delta_n$ such that $\delta_n+\omega \geq \alpha$.
		Now it  clearly suffices to argue that for each countable group $H$ with $G'_\alpha \subseteq H$  both $H / (H \cap G'_{\delta_n+\omega})$ and $(H \cap G'_{\delta_n+\omega}) / G'_\delta$ are free, as extending a free group by a free group always yields a free group.
		As \ref{ass--} implies that 
		$$G'_\alpha = G'_{\beta+\omega} \oplus (\oplus_{\gamma \in [\beta+\omega, \alpha)} \bbZ y_\gamma), \text{ and}$$
		$$G'_{\delta_n+\omega} = G'_{\beta+\omega} \oplus (\oplus_{\gamma \in [\beta+\omega, \delta_n+\omega)} \bbZ y_\gamma),$$
		clearly $G'_{\delta_n+\omega} / G'_\alpha = \oplus_{\gamma \in [\alpha, \delta_n+\omega)} \bbZ (y_\gamma + G'_\alpha)$, so  $G'_{\delta_n+\omega} / G'_\alpha$ is a free group, but then
		$(H \cap G'_{\delta_n+\omega}) / G'_\alpha$ is free as well, as every subgroup of a free group is free.

		Now recall the so called Second Isomorphism Theorem.
		\begin{lemma}\label{isose}(Second Isomorphism Theorem)
			If $H$, $N$ are subgroups of a (not necessarily abelian) group $G$, and $N$ is normal, then $H/(H\cap N) \simeq HN/N$ through the mapping $h(H \cap N) \mapsto hN$.
		\end{lemma}
		This implies that $H / (H \cap G'_{\delta_n+\omega}) \simeq (H + G'_{\delta_n+\omega}) / G'_{\delta_n+\omega}$. But we have already seen that $G'_{\delta_n+\omega} = G^0_{\delta_n+\omega}$ in \ref{assb}, and $G/ G^0_{\delta_n+\omega}$ is $\aleph_1$-free by \ref{G0} / \ref{b}, so its countable subgroup $(H + G'_{\delta_n+\omega}) / G^0_{\delta_n+\omega} = (H + G'_{\delta_n+\omega}) / G'_{\delta_n+\omega}$ must be free.

	\end{PROOF}
	\end{PROOF}

	Lemma $\ref{isose}$ immediately implies the following.
	\begin{enumerate}[label = $\boxplus_{\arabic*}$, ref = $\boxplus_{\arabic*}$]
		\setcounter{enumi}{\value{bocou}}
		\item \label{asob} If a filtration $\bar G'' = \langle G''_\alpha: \ \alpha < \omega_1 \rangle$ of $G$ satisfies $(*)_{S}(G, \bar G'')$ for some $S \cap \{ \alpha + \omega: \ \alpha < \omega_1 \}= \emptyset$, then for each ordinal $\alpha$, the groups
		$G''_\alpha\subseteq G''_{\alpha + \omega}\subseteq G$ satisfy:
		if $K\subseteq G$ is countable, then
		$K/(G''_{\alpha+\omega} \cap K) \simeq (K+G''_{\alpha+\omega}) / G''_{\alpha+\omega}$ is free, in particular, if $G''_{\alpha+\omega} \cap K = G''_{\alpha}$, then $K/G''_{\alpha}$ is free.
		\stepcounter{bocou}
	\end{enumerate}
	
	\mn
	 We now prove the following, slightly more general claim than what is needed for Theorem $\ref{2.1}$ (but essential for Proposition $\ref{ekp}$/\ref{eki}).

	\begin{claim}\label{hulyetekkl}
%		If $G$ is an abelian group, $S \subseteq \Omega \setminus \{ \alpha + \omega: \ \alpha < \omega_1 \}$ is stationary, and the filtration $\bar G = \langle  G_\alpha: \ \alpha < \omega_1 \rangle $ satisfies filtra,
	If $G$ is an abelian group, $S \subseteq \Omega \setminus \{\alpha + \omega: \ \alpha < \omega_1 \}$ is stationary,
	and the filtration $\bar G = \langle G_\alpha: \ \alpha < \omega_1 \rangle$ satisfies  $(**)_{S}(G, \bar G)$ (from Claim $\ref{fil1}$),	\then \ there are 
		\begin{enumerate}[label = $\arabic*)$, ref = $\arabic*)$]
			\item \label{h1} a maximal independent sequence $\langle x_\alpha: \ \alpha < \omega_1 \rangle \in \ ^{\omega_1}G$,
			\item for each $\delta \in S$ a strictly increasing sequence $\langle \gamma^\delta_n: \ n \in \omega \rangle$ with limit $\delta$,
			\item \label{h3}	and a ladder system $\bar \eta = \langle \eta_\delta: \ \delta \in S \rangle$
		\end{enumerate} 
		
		satisfying
		\begin{enumerate}[label = $(\bullet)_{\roman*}$, ref = $(\bullet)_{\roman*}$]
			\item \label{e0} for each $\alpha < \omega_1$: $G_\alpha = \clpr(\{ x_i: \ i <\alpha \})$,
			\item \label{e} whenever $G / G_\delta$ is $\aleph_1$-free (in particular, if $\delta\notin S$ is limit) then $G_{\delta+\omega}
			=G_\delta \oplus  (\oplus_{n \in \omega} {\mathbb Z} x_{\delta+n})$,
			%(i.e. each element $x \in G_{\delta+\omega}$ can be written uniquely in the form $g_\delta + \sum_{n < \omega} m_n x_{\delta+n}$ with all but finitely many $m_n$'s are $0$),
			and
			\item \label{e2} if $\delta \in S$, $n \in \omega$, then for the set
			\[ v_n^\delta = \left\{ \begin{array}{ll} \alpha<\delta: & \clpr\left(G_{\alpha+1} \cup 
				\{x_{\delta+\ell}:\ell<n\} \right) \ne \\ &
				\clpr\left( G_\alpha \cup
				\{ x_{\delta+\ell}: \ \ell < n \} \right) + G_{\alpha+1} \end{array} \right\}.\]
			we have
			\[ v_n^\delta \subseteq \ran(\eta_\delta).\]
			\item \label{e3} for each $\delta \in S$, $n \leq m \in \omega$ and $\gamma_* \leq \gamma_n$, we have
			\[  \clpr(G_{\gamma_*} \cup \{ x_{\delta+i}: \ i < m \}) = \clpr(G_{\gamma_*} \cup \{ x_{\delta+i}: \ i < n \}) \oplus (\oplus_{j=n}^{m-1} \bbZ x_{\delta+j}).  \]
		\end{enumerate}
		Moreover, if $\bar \eta^0 = \langle \eta^0_\delta: \ \delta \in S \rangle$ is a ladder system on $S$, and $(**)^{+}_{S, \bar \eta^0}(G, \bar G)$ holds, which is  the conjunction of clauses \ref{alp}-\ref{delp} from $(**)_{S}(G, \bar G)$ (from Claim $\ref{fil1}$) and the additional clause 
		\begin{enumerate}[label = $(\varp)^{**}$, ref = $(\varp)^{**}$]
			\item \label{varpp}for each $\delta \in S$ for some sequence $\langle y_{\delta+\ell}: \ \ell < \omega \rangle \in \ ^\omega G_{\delta+\omega}$ which is 
	   a maximal  independent sequence over $G_\delta$,  for each $n \in \omega$ 
				\[ \left\{ \begin{array}{ll} \alpha<\delta: & \clpr\left( G_{\alpha+1} \cup
					\{y_{\delta+\ell}:\ell<n\} \right) \ne \\ &
					\clpr\left( G_\alpha \cup
					\{ y_{\delta+\ell}: \ \ell < n \} \right) +G_{\alpha+1} \end{array} \right\} \subseteq^* \ran(\eta^0_\delta),\]
		\end{enumerate}  holds, then we can assume that
	\begin{enumerate}[label = $(\bullet)_{\roman*}$, ref = $(\bullet)_{\roman*}$]
		\setcounter{enumi}{4}
		\item \label{e4}  $ \ran(\eta_\delta) \subseteq \ran(\eta^0_\delta).$
	\end{enumerate}
		
	\end{claim}
	\begin{PROOF}{Claim \ref{hulyetekkl}}
		First we construct the $x_\beta$'s. We proceed by induction on $\delta \in \Omega$ (i.e.\ $\delta$ is limit), and choose the $x_{\delta+\ell}$'s ($\ell \in \omega$). Fixing such a $\delta$, if $\delta \notin S$ then using $(**)_S(G, \bar G)$ it can be easily seen from clause \ref{gamp} that $G_{\delta+\omega}/ G_\delta$ is a free abelian group of rank $\aleph_0$ so we can pick a system $x_{\delta}$, $x_{\delta+1}$, $x_{\delta+2}$, $\dots$ such that $\{ x_{\delta+i}+G_\delta: \ i \in \omega \}$ is a basis of $G_{\delta+\omega}/G_\delta$ (e.g. apply induction and in each step invoke Fact \ref{fa1} with $H_0 = G_{\delta+n}/G_{\delta+n}$, $H_1 = G_{\delta+n+1} / G_{\delta+n}$ to obtain $x_{\delta_n}$).
		
		% assume that the limit ordinal $\delta \notin C \cap S$. Now if $\delta \in C \setminus S$, then use \ref{assc}, and we are done by \ref{base}.  On the other hand, if $\delta \notin C$ (where $\delta \in [\alpha+\omega, \suc_C(\alpha))$, so $\delta + \omega < \suc_C(\alpha)$ as well by \ref{assC}), then \ref{ass--} implies that $$G_\delta = G_{\alpha+\omega} \oplus (\oplus_{\gamma \in [\alpha+\omega, \delta)} \bbZ y_\gamma),$$ and
		%$$G_{\delta+\omega} = G_{\alpha+\omega} \oplus (\oplus_{\gamma \in [\alpha+\omega, \delta+\omega)} \bbZ y_\gamma),$$
		%which clearly implies \ref{e}.
		So we can assume that $\delta \in S$.	%Note that by clause (c) of \ref{asse}:
		%\mn
		%\begin{enumerate}[label = $\boxplus_{\arabic*}$, ref = $\boxplus_{\arabic*}$]
		%	\setcounter{enumi}{\value{bocou}}
		%	\item if $\delta\in S$ then $G_\delta \cup
		%	\{x_{\delta+n}:n<\omega\}$ generates a subgroup of 
		%	$G_{\delta+\omega}$ such that $G_{\delta+\omega}/ \cll( 
		%	G_{\delta+\omega}\cup \{x_{\delta+n}:n<\omega\} )$ is
		%	a torsion group,
		%	\stepcounter{bocou}
		%\end{enumerate}
		%\mn
		%	(which is possible trivial). However,
		 Without loss of generality:
		 \newcounter{boocou}
		\begin{enumerate}[label = $\blacksquare_{\arabic*}$, ref = $\blacksquare_{\arabic*}$]
			\setcounter{enumi}{\value{boocou}}
			\item \label{nfree} $\delta \in S \Rightarrow
			G_{\delta+\omega}/G_\delta$ is not free,
			\stepcounter{boocou}
		\end{enumerate}  as we can decrease 
		$S$ (and $(**)_S(G, \bar G)$ will still hold). 
		\begin{definition} \label{zdf}
		 We fix a system $\{z_{\delta+i}: \ i \in \omega \} \subseteq G_{\delta+\omega}$ being a maximal independent family over $G_\delta$ so that
		 if moreover \ref{varpp} holds, then the system $y_{\delta+\ell} = z_{\delta+\ell}$ ($\ell < \omega$) witnesses \ref{varpp}.
		\end{definition}
		So for every  $n < \omega$ we can consider the set
		\begin{equation} \label{vdeltan}
			w_n^\delta = \left\{ \begin{array}{ll} \alpha<\delta: & \clpr\left( G_{\alpha+1} \cup
			\{z_{\delta+\ell}:\ell<n\} \right) \ne \\ &
			\clpr\left( G_\alpha \cup
			\{ z_{\delta+\ell}: \ \ell < n \} \right) +G_{\alpha+1} \end{array} \right\}.
		\end{equation}
		
		\mn
		\begin{sclaim} \label{kiszam}
			If $\{z_{\delta+i}: \ i \in \omega \} \subseteq G_{\delta+\omega}$ is a maximal independent family over $G_\delta$, and $w^\delta_n$'s are defined as in \eqref{vdeltan}, then for each $n < \omega$ and ordinal $\alpha_*<\delta$ the set
			$w^\delta_n\cap \alpha_*$ is finite.
		\end{sclaim} 
		\begin{PROOF}{Subclaim \ref{kiszam}}
			Fixing $\alpha_* < \delta$, $n \in \omega$ we define for each $\alpha \leq \alpha_*$ the groups
			\begin{equation} \label{HKdf} \begin{array}{ll} H_{\alpha} =  & \clpr\left( G_\alpha \cup 
					\{z_{\delta+\ell}:\ell<n\} \} \right), \text{ and} \\
					K_{\alpha} = & H_{\alpha} / G_\alpha. \end{array}
			\end{equation}
			%From now on in the proof of Claim \ref{kiszam} we will refer to $H^\delta_{\alpha,n}$ only as $H_\alpha$, and $K^\delta_{\alpha,n}$ as $K_\alpha$; furthermore, $\alpha, \beta$ are $\leq \alpha_*$.
			Observe that by Lemma $\ref{isose}$ (the Second Isomorphism Theorem) %and recalling $G_\alpha = \clpr(\{x_i: \ i < \alpha\}) = H_{\alpha} \cap G_\beta$ 
			if $\alpha \leq \beta$ we have that the canonical mapping 
			$$\varphi_{\beta, \alpha}: K_\alpha = H_\alpha / G_\alpha \to (H_\alpha+ G_\beta) / G_\beta \leq H_\beta/ G_\beta = K_\beta $$
			(i.e. $\varphi_{\beta,\alpha}(h + G_\alpha ) = h + G_\beta$) 
			is an embedding of $K_\alpha$ into $K_\beta$, moreover $\varphi_{\gamma,\beta} \circ \varphi_{\beta,\alpha} = \varphi_{\gamma,\alpha}$ for each $\alpha \leq \beta \leq \gamma \leq \alpha_*$, so $\{K_\alpha, \varphi_{\beta,\alpha}: \ \alpha \leq \beta \leq \alpha_*\}$ forms a direct system, so we can assume that $K_\alpha \leq K_\beta$ if $\alpha \leq \beta$. It is easy to check that 
			\begin{enumerate}[label = $(\star)_{\arabic*}$, ref = $(\star)_{\arabic*}$]
				\item $H_\alpha + G_\beta \neq H_\beta$ iff $\varphi_{\beta,\alpha}$ is not surjective, and
				\item $\bigcup_{\alpha < \beta} K_\alpha = K_\beta$ if $\beta \leq \alpha_*$ is limit, i.e. the sequence is continuous.
			\end{enumerate} 
			
			Observe that $\alpha_* +\omega < \delta$ as $ \alpha_* < \delta \in S$ (no $\gamma \in S$ is of the form $\xi+\omega$) %(in fact elements of $S$ are divisible by $\omega^2$), 
			but note that $\alpha_*$ is not necessarily limit.
			Recalling that $\{z_{\delta+\ell}: \ \ell < \omega\}$ is independent over $G_\delta$, thus over $G_{\alpha_* + \omega}$ and $G_\beta$ for any $\beta \leq \alpha_*$, too, so by the definition of  $H_{\beta}$ we have $H_{\beta} \cap G_{\alpha_*+\omega} = G_{\beta}$. 
			Therefore by (the main clause of) \ref{asob} 
			\begin{enumerate}[label = $(\star)_{\arabic*}$, ref = $(\star)_{\arabic*}$]
				\setcounter{enumi}{2}
				\item \label{Kre} the quotient $K_{\beta} = H_{\beta} / G_{\beta} = H_\beta/ (H_\beta \cap G_{\alpha_*+\omega}) \simeq  (H_\beta + G_{\alpha_*+\omega}) / G_{\alpha_*+\omega}$ is free whenever $\beta \leq \alpha_*$ (so free for each $\beta < \delta$).
			\end{enumerate}

			In order to finish the proof assume on the contrary that for infinitely many $\alpha < \alpha_*$ we have $K_{\alpha} \subsetneq K_{\alpha+1}$, let $\beta \leq \alpha_* < \delta$ be a limit point of this set, i.e.\ $\beta$ is a limit ordinal, and for cofinally many $\alpha < \beta$, $K_{\alpha} \subsetneq K_{\alpha+1}$ holds. But $K_\beta$ is a free group of finite rank (of rank $n$), so finitely generated, and so for some $\alpha < \beta$ $K_\alpha = K_\beta$ must hold, which is a contradiction.
		\end{PROOF}
		For future reference we remark the following fact that follows from the previous proof.
		\begin{lemma} \label{tee}
			If $H \leq H' \leq G$ are pure subgroups in $G$, $g_0,g_1, \ldots, g_{j-1} \notin H'$ are independent over $H'$, and $G/H'$ is $\aleph_1$-free, then $\clpr(H \cup \{g_0,g_1, \ldots, g_{j-1}\}) / H$ is a free group (of finite rank).
			
			In particular, if $G$ is an abelian group, $S \subseteq \Omega \setminus \{\alpha + \omega: \ \alpha < \omega_1 \}$ is stationary,
			and the filtration $\bar G = \langle G_\alpha: \ \alpha < \omega_1 \rangle$ satisfies  $(**)_{S}(G, \bar G)$ (thus $G / G_{\alpha+\omega}$ is $\aleph_1$-free), then for any $\delta \in S$, $\beta < \delta$, $g_0,g_1, \ldots, g_{j-1} \notin G_\delta$ which are independent over $G_\delta$ we have
			$$ \clpr(G_\beta \cup \{g_0,g_1, \ldots, g_{j-1}\}) / G_\beta \textrm{ is free (of finite rank).}$$
		\end{lemma}
		\begin{PROOF}{Lemma \ref{tee}}
			Note that $\clpr(H \cup \{g_0,g_1, \ldots, g_{j-1}\}) \cap H' = H$ (as otherwise for some  $h'-h \in H' \setminus H$ we would have $h'-h = k_0 g_0 + \cdots + k_{j-1} g_{j-1}$ contradicting that the $g_i$'s form an independent system over $H'$) and proceed as in the proof above.
		\end{PROOF}
					
		In addition to the claim observe that:
		\begin{enumerate}[label = $\blacksquare_{\arabic*}$, ref = $\blacksquare_{\arabic*}$]
			\setcounter{enumi}{\value{boocou}}
			\item  if $\delta \in S$ and $\alpha_* < \delta$, then
			\[
			\clpr(G_{\alpha_*} \cup \{z_{\delta+\ell}:\ell<n\}) /
			\cll (G_{\alpha_*}\cup\{z_{\delta+\ell}:\ell<n\}),
			\]
			is a finite group (and a torsion group, of course),
			similarly, if $\alpha < \beta < \delta$, and $w_n^\alpha \subsetneq w_n^\beta$, then
			\[ \clpr(G_{\beta} \cup \{z_{\delta+\ell}:\ell<n\}) / \left( \clpr(G_{\alpha} \cup \{z_{\delta+\ell}:\ell<n\}) + G_\beta \right) \]
			is a finite group.
			\stepcounter{boocou}
		\end{enumerate}
		\mn
		At this point recall that if we assume \ref{varpp}, then
		\[ \left\{ \begin{array}{ll} \alpha<\delta: & \clpr\left( G_{\alpha+1} \cup
			\{y_{\delta+\ell}:\ell<n\} \right) \ne \\ &
			\clpr\left( G_\alpha \cup
			\{ y_{\delta+\ell}: \ \ell < n \} \right) +G_{\alpha+1} \end{array} \right\} \subseteq^* \ran(\eta^0_\delta),\]
			and since in that case $z_{\delta+n} = y_{\delta+n}$ ($n \in \omega$) by Definition \ref{zdf}, we have
			\[ \left\{ \begin{array}{ll} \alpha<\delta: & \clpr\left( G_{\alpha+1} \cup
			\{z_{\delta+\ell}:\ell<n\} \right) \ne \\ &
			\clpr\left( G_\alpha \cup
			\{ z_{\delta+\ell}: \ \ell < n \} \right) +G_{\alpha+1} \end{array} \right\} \subseteq^* \ran(\eta^0_\delta).\]
				
		Note that if we choose $x_{\delta+\ell}$ to be any 
		$x'\in G_{\delta+\ell+1} \setminus 
		G_{\delta+\ell}$, then the sequence will satisfy \ref{e0}.
		This justifies the following.
		\begin{definition}
			\item \label{repl}  by induction on $n \in \omega$
			we define $\langle x_{\delta +n}:n < \omega \rangle$ so that for each $\ell$, $x_{\delta+\ell} \in G_{\delta+\ell+1} \setminus 
			G_{\delta+\ell}$ holds as follows.
			\begin{enumerate}[label = $\arabic*)$, ref = $\arabic*)$]
				\item  	Choose first
				\[
				\gamma^\delta_0 < \ldots \gamma^\delta_n < \gamma^\delta_{n+1} < 
				\dots < \delta = \bigcup\limits_{n<\omega}\gamma^\delta_n,
				\]
				such that if \ref{varpp} holds, then  moreover
				\begin{equation} \label{replegy} \forall n: \ \ w_{n+1}^\delta \setminus \gamma^\delta_n \subseteq \ran(\eta^0_\delta). \end{equation}
				(recalling how we defined $w^\delta_n$'s in \eqref{vdeltan}, i.e. 
				\[
					w_n^\delta = \left\{ \begin{array}{ll} \alpha<\delta: & \clpr\left( G_{\alpha+1} \cup
						\{z_{\delta+\ell}:\ell<n\} \right) \ne \\ &
						\clpr\left( G_\alpha \cup
						\{ z_{\delta+\ell}: \ \ell < n \} \right) +G_{\alpha+1} \end{array} \right\}).
				\]
							\item \label{repl2} Second, by induction on $n < \omega$ choose $x_{\delta+n} \in \clpr(G_{\gamma^\delta_n} \cup \{ x_{\delta+\ell}: \ \ell < n\} \cup \{ z_{\delta+n}\}) \setminus G_{\delta+n}$ as above such
				that:
				\[
				\clpr(G_{\gamma^\delta_n}\cup\{x_{\delta+\ell}:\ell<n\} \cup \{z_{\delta+ n} \})=
				\clpr(G_{\gamma^\delta_n}\cup\{ x_{\delta+\ell}:\ell<n\}) \oplus \bbZ
				x_{\delta+n}.
				\]
			\end{enumerate}
			%\stepcounter{boocou}
		%\end{enumerate}
		\end{definition}
		Choosing such $x_{\delta+n}$'s is possible by the following known fact (applying to the finite rank group $H_1 = \clpr(G_{\gamma^\delta_n}\cup\{x_{\delta+\ell}:\ell<n\} \cup \{z_{\delta+ n} \})/ G_{\gamma^\delta_n}$, which is free, since $\clpr(G_{\gamma^\delta_n}\cup\{x_{\delta+\ell}:\ell<n\} \cup \{z_{\delta+ n} \}) \cap G_{\gamma^\delta_n+\omega} =  G_{\gamma^\delta_n}$ (recall \ref{asob}), and letting $H_0 = \clpr(G_{\gamma^\delta_n}\cup\{x_{\delta+\ell}:\ell<n\})/ G_{\gamma^\delta_n}$).
		\begin{fact} \label{fa1} Let $H_0 \leq H_1$ be free abelian groups of finite rank, where $\clpr_{H_1}(H_0) = H_0$, i.e. $H_0$ is a pure subgroup, and for some $g$ we have $\clpr(H_0 \cup \{g\}) = H_1$.
			Then for some $g^* \in H_1 \setminus H_0$ we have $H_1 = H_0 \oplus \bbZ g^*$.  
		\end{fact}
		\begin{PROOF}{Fact  \ref{fa1}}
			Since $H_0$ is a pure subgroup $H_1/H_0$ is torsion-free. As $H_1$ is a finite rank free group, it is finitely generated, so is $H_1/H_0$, therefore w.l.o.g.\ $H_1/H_0$ is a finitely generated group of rank one. This means that $H_1 / H_0$ is isomorphic to a subgroup $H^*$ of the additive group $\bbQ$, with the generating set $\{n_1q^*, \dots, n_kq^* \}$ with $q^*$ not necessarily belonging to $H^*$, but the greatest common divisor of $\{n_1,n_2, \dots n_k\}$ is $1$. Now by a standard elementary argument for some $\ell_1,\ell_2, \dots, \ell_k \in \bbZ$ 
			$$\ell_1n_1 + \ell_2 n_2+ \dots + \ell_k n_k =1,$$
			so $H^* = \bbZ q^*$, i.e. a free group.
			Choosing $g_*$ so that $g_*+H_0$ generates $H_1/H_0$ clearly works.
		\end{PROOF}
		\begin{fact} \label{fa2} Let $H_0 \leq H_1$ be free abelian groups of finite rank, where $\clpr_{H_1}(H_0) = H_0$, i.e. $H_0$ is a pure subgroup, and for some $H_0$-independent set $\{g_0,g_1, \dots, g_n\}$ we have $\clpr(H_0 \cup \{g_i: \ i <n\}) = H_1$.
			Then for some $g^*_i \in \clpr(H_0 \cup \{ g_j: \ j \leq i\})$ ($i < n$) we have $H_1 = H_0 \oplus (\oplus_{i <n} \bbZ g^*_i)$.  
		\end{fact}
		\begin{PROOF}{Fact \ref{fa2}} By induction, apply Fact $\ref{fa1}$ at each step for $\clpr(H_0 \cup \{ g_j: \ j \leq i\})$ and $\clpr(H_0 \cup \{ g_j: \ j \leq i+1\})$.
		\end{PROOF}
		%\begin{claim} \label{kov} Replacing the $x_{\delta+n}$'s by $x'_{\delta+n}$ for each $\delta \in S$, $n \in \omega$ where requiring that \ref{repl}\ref{repl2} holds implies
		%		for each $n$ and $\alpha \geq \gamma^\delta_n$, the group $H^\delta_{\alpha,n}$ defined in.
		%	\end{claim}
		%\begin{PROOF}{Claim \ref{kov}}
		%	Fix $\delta \in S$. First observe that  $x'_{\delta+n} \in \clpr(G_{\gamma^\delta_n} \cup \{ x_{\delta+\ell}: \ \ell < n+1\} \setminus G_{\delta+n}$ implies that 
		%	$$x'_{\delta+n} = q_{i_0}x_{i_0} + \dots + q_{i_{k-1}} x_{i_{k-1}} + q_{\delta} x_\delta+ \dots + q_{\delta+n} x_{\delta+n}$$
		%	for some $i_0,i_1, \dots, i_{k-1} \in \gamma^\delta_n$, $q_j \in \bbQ$ for each $j$.
		%	Therefore after redefining the $x_{\xi+\ell}$'s ($\xi \in S$) not only the $G_\alpha$'s were not changed, but	
		%	$$x'_{\delta+n} = q_{i_0}x_{i_0} + \dots + q_{i_{k-1}} x_{i_{k-1}} + q_{\delta} x_\delta+ \dots + q_{\delta+n} x_{\delta+n}$$
		%\end{PROOF}
		Note the following simple observation which is easy to check:
			\begin{enumerate}[label = $\blacksquare_{\arabic*}$, ref = $\blacksquare_{\arabic*}$]
			\setcounter{enumi}{\value{boocou}}
			\item \label{repli} By induction on $\ell$ one can prove that $x_{\delta+\ell} \in \clpr(G_{\gamma_\ell^\delta} \cup \{ z_{\delta+j}: \ j \leq \ell\})$, and so
			$$ \forall \gamma_* \geq \gamma^\delta_\ell \ \ \clpr(G_{\gamma_*} \cup \{ z_{\delta+j}: \ j \leq \ell\}) = \clpr(G_{\gamma_*} \cup \{ x_{\delta+j}: \ j \leq \ell\}).$$
			\stepcounter{boocou}
		\end{enumerate}

		\begin{definition} \label{HHdef}
			For $\delta \in S$, $\gamma \leq \delta$, $n \in \omega$ we define the group
				$$H^\delta_{\gamma,n} = \clpr(G_{\gamma}\cup\{x_{\delta+\ell}:\ell<n \}).$$
		\end{definition}
		\noindent So using this, the choice of the $x_{\delta+n}$'s in \ref{repl} implies that
		\begin{enumerate}[label = $\blacksquare_{\arabic*}$, ref = $\blacksquare_{\arabic*}$]
			\setcounter{enumi}{\value{boocou}}
			\item \label{replkesz}  for each $\delta \in S$, $n \in \omega$: 
			$$H^\delta_{\gamma^\delta_n,n+1} = \clpr(G_{\gamma^\delta_n}\cup\{x_{\delta+\ell}:\ell<n+1 \})=
			\clpr(G_{\gamma^\delta_n}\cup\{ x_{\delta+\ell}:\ell<n\}) \oplus \bbZ
			x_{\delta+n}.$$
			\stepcounter{boocou}
		\end{enumerate}
		Now we are ready to argue \ref{e3}. After proving Subclaim \ref{direo} we are going to finish the proof of Claim \ref{hulyetekkl} by constructing $\eta_\delta: \omega \to \delta$, verifying \ref{e2}, and possibly \ref{e4}.
	
		We have to remark the following (easy) corollary of	\ref{replkesz}:
		\begin{sclaim} \label{direo0}
			%If $H^\delta_{\gamma_*,n+1} = H^\delta_{\gamma_*,n} +  \bbZ x_{\delta+n}$ for $n \in \omega$, then
			For each $\delta \in S$, $n \in \omega$, $\gamma_* \leq \gamma^\delta_n$, it is the case that
			$$(H^\delta_{\gamma_*,n+1} = ) \ \clpr(G_{\gamma_*}\cup\{x_{\delta+\ell}:\ell<n+1 \})=
			\clpr(G_{\gamma_*}\cup\{ x_{\delta+\ell}:\ell<n\}) \oplus \bbZ
			x_{\delta+n}.$$
		\end{sclaim}
		\begin{PROOF}{Subclaim \ref{direo0}}
			Since $H^\delta_{\gamma_*,n}$ is the pure subgroup generated by $\{x_\alpha: \ \alpha < \gamma_*\} \cup \{x_{\delta+k}: \ k <n \}$ and the set $\{x_\alpha: \ \alpha < \omega_1\}$ is independent we  only have to verify that 
			$$H^\delta_{\gamma_*,n+1} = H^\delta_{\gamma_*,n} +  \bbZ x_{\delta+n}.$$
			
			Now suppose that
			$$g \in H^\delta_{\gamma_*,n+1} \setminus (H^\delta_{\gamma_*,n} \oplus \bbZ x_{\delta+n}).$$
			
			On the one hand $g \in H^\delta_{\gamma_*,n+1} = \clpr(G_{\gamma_*}\cup\{x_{\delta+\ell}:\ell<n+1)$ implies that 
			\begin{equation} \label{g} g = \sum_{k =0}^{s-1} q_{k} x_{i_k} + \sum_{k=0}^{n} p_{k}x_{\delta+k}, \end{equation}
			$\{i_0,i_1, \dots, i_{s-1} \} \subseteq \gamma_*$,	
			$\{q_0,q_1, \dots, p_0,p_1, \dots, p_n\} \subseteq \bbQ$.
			On the other hand $g \in H^\delta_{\gamma_*,n+1} \subseteq H^\delta_{\gamma^\delta_n,n+1} = H^\delta_{\gamma^\delta_n,n} \oplus  \bbZ x_{\delta+n}$, so $g = g_0 + z x_{\delta+n}$, where $z \in \bbZ$, and $g_0$ is a $\bbQ$-linear combination of elements of the independent set $\{x_\alpha: \ \alpha < \gamma^\delta_n\} \cup \{x_{\delta_k} : \ k < n \}$. This means that $p_n = z \in \bbZ$, and $g_0 \in \clpr(G_{\gamma_*}\cup\{x_{\delta+\ell}:\ell<n))$, as desired.
			
		\end{PROOF}
		
		\begin{sclaim} \label{direo}% If $H^\delta_{\gamma_*,n+1} = H^\delta_{\gamma_*,n} +  \bbZ x_{\delta+n}$ for $n \in \omega$, then
			For each $\delta \in S$, $n \leq m \in \omega$, $\gamma_* \leq \gamma^\delta_n$, we have
			\[H^\delta_{\gamma_*,m} = H^\delta_{\gamma_*,n} \oplus (\oplus_{j=n}^{m-1}  \bbZ x_{\delta+j}). \]
		\end{sclaim}
		
		\begin{PROOF}{Subclaim \ref{direo}}
			Fix $\delta \in S$, $n \in \omega$, we proceed by induction on $m$.
			For $m=n$ the statement is void, for $m=n+1$ the statement is exactly Subclaim $\ref{direo0}$. Assume that we already know that $H^\delta_{\gamma_*,m} = H^\delta_{\gamma_*,n} \oplus (\oplus_{j=n}^{m-1}  \bbZ x_{\delta+j})$.
			Now by Subclaim $\ref{direo0}$ (as  $\gamma^\delta_n \leq \gamma^\delta_m$, so $\gamma_* \leq \gamma^\delta_m$)
			$$H^\delta_{\gamma_*,m+1} = H^\delta_{\gamma_*,m} \oplus  \bbZ x_{\delta+m} = H^\delta_{\gamma_*,n} \oplus (\oplus_{j=n}^{m-1}  \bbZ x_{\delta+j}) \oplus \bbZ x_{\delta+m},$$

		\end{PROOF}
		\noindent Recall that the $v_n^\delta$'s were defined in \ref{e2} as
		\[ v_n^\delta = \left\{ \begin{array}{ll} \alpha<\delta: & \clpr\left(G_{\alpha+1} \cup 
			\{x_{\delta+\ell}:\ell<n\} \right) \ne \\ &
			\clpr\left( G_\alpha \cup
			\{ x_{\delta+\ell}: \ \ell < n \} \right) + G_{\alpha+1} \end{array} \right\}.\]
		 and observe that Subclaim $\ref{direo}$ clearly implies that
		\begin{enumerate}[label = $\blacksquare_{\arabic*}$, ref = $\blacksquare_{\arabic*}$]
			\setcounter{enumi}{\value{boocou}}
			\item \label{stab} for each $\delta \in S$, $n \leq m < \omega$ we have
			\[
			(\forall \delta \in S, \ \forall n\leq m \in \omega) \quad	v^\delta_{m} \cap \gamma^\delta_n = v^\delta_n \cap \gamma^\delta_n,
			\] 
			\stepcounter{boocou}
		\end{enumerate}
		since an easy calculation yields that  for any fixed $\gamma_* < \gamma^\delta_n$ (so $\gamma_*+1 \leq \gamma_n^\delta$) the conditions 
		$$H^{\delta}_{\gamma_*,n} + G_{\gamma_*+1} \neq H^{\delta}_{\gamma_*+1,n},$$ and the assertion
		$$H^{\delta}_{\gamma_*,m} + G_{\gamma_*+1} = H^{\delta}_{\gamma_*,n} + (\bbZ x_{\delta+n} + \bbZ x_{\delta+n+1} + \dots + \bbZ x_{\delta{m-1}}) + G_{\gamma_*+1}$$
		$$ \neq $$
		$$ H^{\delta}_{\gamma_*+1,m} = H^{\delta}_{\gamma_*+1,n} + (\bbZ x_{\delta+n} + \bbZ x_{\delta+n+1} + \dots + \bbZ x_{\delta{m-1}}) $$
		are equivalent. (In fact it can be shown that $v^\delta_n \subseteq v^\delta_{n+1}$ holds for each $n$.)
		\begin{enumerate}[label = $\blacksquare_{\arabic*}$, ref = $\blacksquare_{\arabic*}$]
			\setcounter{enumi}{\value{boocou}}
			\item Define $v^\delta := \bigcup\limits_{n<\omega} v^\delta_n$.
			\stepcounter{boocou}
		\end{enumerate}
			\begin{sclaim} \label{d}
			For each $\delta \in S$,  the set $v^\delta = \bigcup\limits_{n<\omega} v^\delta_n$ is a subset of 
			$\delta$ such that $\alpha<\delta\Rightarrow v^\delta\cap \alpha$ is finite.
		\end{sclaim}
		\begin{PROOF}{Subclaim \ref{d}}
			Fix $\alpha<\delta$, set $n$ to be so that $\alpha< \gamma^\delta_n$, we  need to verify that 
			$\bigcup\limits_{m<\omega} v^\delta_m \cap \alpha$ is finite. By \ref{stab} clearly
			$$ \bigcup\limits_{m \geq n} v^\delta_m \cap \alpha = v^\delta_n \cap \alpha,$$
			whereas for each $v^\delta_i \cap \alpha$ ($i<n$) (and for $i = n$) is finite by applying Subclaim \ref{kiszam} with the roles $z_{\delta+i} = x_{\delta+i}$ ($i < \omega$).
			
			Therefore,  we obtained that $v^\delta\cap \alpha = \bigcup_{i\leq n} (v^\delta_i \cap \alpha)$, which is finite, as desired.
	\end{PROOF}
		
		Moreover, for future reference we have to remark that
		\begin{enumerate}[label = $\blacksquare_{\arabic*}$, ref = $\blacksquare_{\arabic*}$]
			\setcounter{enumi}{\value{boocou}}
			\item  \label{furef} if $\beta < \beta'$ are two consecutive elements of $v^\delta$, then for each $n$ and $\theta \in (\beta, \beta']$
			$$ H^\delta_{\beta',n}= \clpr(\{x_\alpha: \ \alpha < \beta'\} \cup \{ x_\delta+j: \ j < n\}) = H^\delta_{\theta,n} + G_{\beta'}$$
			\stepcounter{boocou}
		\end{enumerate}
		(by an easy induction argument, using the definition of $v^\delta_n$ \ref{vdeltan}, since $v^\delta_n \cap (\beta, \beta') = \emptyset$).
		
		If $v^\delta$ is bounded in $\delta$ (e.g.\ by some $\gamma^\delta_n$) we get a contradiction to
		``$G_{\delta+\omega}/G_\delta$ is not free''. (Why would this lead to a contradiction? If $v_\delta \subseteq \gamma^\delta_m$ for some $m< \omega$, then using Subclaim $\ref{direo}$
		$$ H^\delta_{\gamma^\delta_m, \infty} = \clpr(\{g_\alpha: \ \alpha < \gamma^\delta_m\} \cup \{x_{\delta+\ell}: \ \ell < \omega \}) = H^\delta_{\gamma^\delta_m} \oplus ( \oplus_{\ell =m}^{\infty} \bbZ x_{\delta+\ell}),$$
		and so the group $H^\delta_{\gamma^\delta_m, \infty} / G_{\gamma^\delta_m}$ is free (\ref{Kre}), but from our indirect assumption and \ref{furef} $H^\delta_{\gamma^\delta_m, \infty} + G_\delta = G_{\delta+\omega}$, so by the isomorphism theorem $H^\delta_{\gamma^\delta_m, \infty} / G_{\gamma^\delta_m} \simeq G_{\delta+\omega} / G_\delta$ is free, contradicting \ref{nfree}.)
		
		Therefore, by Subclaim \ref{d}, and $|v^\delta|=\aleph_0$: 
		\mn
		\begin{enumerate}[label = $\blacksquare_{\arabic*}$, ref = $\blacksquare_{\arabic*}$]
			\setcounter{enumi}{\value{boocou}}
			\item if $\delta \in S$ then 
			$v^\delta$ is an unbounded subset of $\delta$ of order type $\omega$.
			\stepcounter{boocou}
		\end{enumerate}
		
		\begin{definition}\label{uhdef0}{\ }
			\begin{enumerate}[label = $\boxminus_{\arabic*}$, ref = $\boxminus_{\arabic*}$]
				\item 	We let $\eta_\delta: \omega \to \delta$ enumerate $v^\delta$ in increasing order.
			\end{enumerate}
		\end{definition}
		Finally we have to argue that \ref{varpp} implies \ref{e4}, i.e.\ $v^\delta = \bigcup_{n \in \omega} v_n^\delta \subseteq \ran(\eta^0_\delta)$.
		So suppose that $n \in \omega$ is the least natural number such that $v_{n+1}^\delta \nsubseteq \ran(\eta^0_\delta)$, pick $\theta \in v_{n+1}^\delta \setminus (v_{n}^\delta \cup \ran(\eta^0_\delta))$ (note that by \ref{e2} $v_0^\delta = \emptyset$, thus $n \geq0$ necessarily). Now \ref{stab} and $\theta \in v_{n+1}^\delta \setminus v_{n}^\delta$ imply that $\gamma^\delta_{n}\leq \theta$, so we obtain
		\begin{equation} \label{hegy} 
			\theta \in v_{n+1}^\delta \setminus (\ran(\eta^0_\delta) \cup \gamma^\delta_{n}).
		\end{equation}
		But by \ref{repli} 
		$$\forall \gamma_* \geq \gamma^\delta_n: \ \ \clpr(G_{\gamma_*} \cup \{ y_{\delta+j}: \ j \leq n \}) =  \clpr(G_{\gamma_*} \cup \{ x_{\delta+j}: \ j \leq n\}), $$
		hence recalling the definition of $v^\delta_{n+1}$ and $w^\delta_{n+1}$ in \ref{e2}, \ref{varpp} clearly $v^\delta_{n+1} \setminus \gamma^\delta_n = w^\delta_{n+1} \setminus \gamma^\delta_{n}$. Now by the definition of the $\gamma^\delta_m$'s
		$\eqref{replegy}$ from \ref{repl} $w^\delta_{n+1} \setminus \gamma^\delta_{n} \subseteq \ran(\eta_\delta^0)$, so 
		$$\theta \in v^\delta_{n+1} \setminus \gamma^\delta_n = w^\delta_{n+1} \setminus \gamma^\delta_{n} \subseteq \ran(\eta_\delta^0)$$
		contradicting $\eqref{hegy}$, we are done.

	\end{PROOF}
	
	\begin{claim} \label{bef}
		Suppose that $G$ is an abelian group, $S \subseteq \omega_1 \setminus \{\alpha+\omega: \ \alpha < \omega_1 \}$ is stationary, the filtration $\bar G = \langle G_i: \ i \in \omega_1 \rangle$ satisfies  $(**)_{S}(G, \bar G)$ (from Claim $\ref{fil1}$), and $\langle x_\alpha: \ \alpha < \omega_1 \rangle$, $\bar \eta^1$, $\langle \gamma^\delta_n: \ \delta \in S, n \in \omega \rangle$
	 satisfy \ref{h1}-\ref{h3}, \ref{e0}-\ref{e3} of Claim \ref{hulyetekkl}, i.e. 
			\begin{enumerate}
			\item $\langle x_\alpha: \ \alpha < \omega_1 \rangle \in \ ^{\omega_1}G$ is a maximal independent sequence,
			\item for each $\delta \in S$ the sequence $\langle \gamma^\delta_n: \ n \in \omega \rangle$ is strictly increasing with limit $\delta$,
			\item 	and $\bar \eta = \langle \eta_\delta: \ \delta \in S \rangle$ is a ladder system,
		\end{enumerate} 
		for which
		\begin{enumerate}[label = $(\bullet)_{\roman*}$, ref = $(\bullet)_{\roman*}$]
			\item  for each $\alpha < \omega_1$: $G_\alpha = \clpr(\{ x_i: \ i <\alpha \})$,
			\item  whenever $G / G_\delta$ is $\aleph_1$-free (in particular, if $\delta\notin S$ is limit) then $G_{\delta+\omega}
			=G_\delta \oplus  (\oplus_{n \in \omega} {\mathbb Z} x_{\delta+n})$,
			%(i.e. each element $x \in G_{\delta+\omega}$ can be written uniquely in the form $g_\delta + \sum_{n < \omega} m_n x_{\delta+n}$ with all but finitely many $m_n$'s are $0$),
			and
			\item \label{ee2} if $\delta \in S$, $n \in \omega$, then for the set
			\[ v_n^\delta = \left\{ \begin{array}{ll} \alpha<\delta: & \clpr\left(G_{\alpha+1} \cup 
				\{x_{\delta+\ell}:\ell<n\} \right) \ne \\ &
				\clpr\left( G_\alpha \cup
				\{ x_{\delta+\ell}: \ \ell < n \} \right) + G_{\alpha+1} \end{array} \right\}.\]
			we have
			\[ v_n^\delta \subseteq \ran(\eta_\delta).\]
			\item \label{ee3} for each $\delta \in S$, $n \leq m \in \omega$ and $\gamma_* \leq \gamma_n$, we have
			\[  \clpr(G_{\gamma_*} \cup \{ x_{\delta+i}: \ i < m \}) = \clpr(G_{\gamma_*} \cup \{ x_{\delta+i}: \ i < n \}) \oplus (\oplus_{j=n}^{m-1} \bbZ x_{\delta+j}).  \]
		\end{enumerate}
	
		\Then 
		$$ (\text{every ladder system }\bar \eta' \text{ very similar to }\bar \eta^1 \text{ has } \aleph_0\text{-uniformization}) \Rightarrow $$
		$$ \Rightarrow  (G \text{ is a }W_\omega \text{ group}). $$
	\end{claim}
	\begin{PROOF}{Claim \ref{bef}}
	\begin{definition}\label{uhdef}
		Fix $\delta \in S$. Using $\{x_\alpha: \ \alpha < \omega_1\}$, $\gamma^\delta_n$ ($n \in \omega$) and $\bar \eta^1$ 		
	\begin{enumerate}[label = $(\intercal)_{\arabic*}$, ref = $(\intercal)_{\arabic*}$]
				%		\item let $h_\delta: \omega \to \omega$ be defined as $h_\delta(m) = \max\{n: \ \eta^1_\delta(n) < \gamma^\delta_m\}$,
			\item \label{uhdef3}	for each $n$ we can find a $Y_{\delta,n}$ such that:
			
				\begin{enumerate}[label = (\greek*), ref = (\greek*)]
				\item  $Y_{\delta, n} \subseteq \{ x_i: \ i \leq \eta^1_\delta(n)\}$ is finite, and
			
				\item   letting $m$ be such that $\eta^1_\delta(n) \in [\gamma^\delta_{m-1}, \gamma^\delta_m)$ we demand
				$$\clpr(G_{\eta^1_\delta(n)+1} \cup
				\{x_{\delta+\ell}:\ell<m\}) = G_{\eta^1_\delta(n)+1} +
				\clpr(\{x_{\delta+\ell}:\ell<m\} \cup Y_{\delta,n}),$$
			\end{enumerate}
			(in other words, one can have a system 
			$$g^{n,m}_0,g^{n,m}_1, \dots, g^{n,m}_{m-1} \in \clpr(G_{\eta^1_\delta(n)+1} \cup
			\{x_{\delta+\ell}:\ell<m\}),$$ each $g^{n,m}_{j}$ is a $\bbQ$-linear combination of elements of $\{x_{\delta+\ell}:\ell<m\} \cup Y_{\delta,n}$ so that $g^{n,m}_0 + G_{\eta^1_\delta(n)+1}$,  $g^{n,m}_1 + G_{\eta^1_\delta(n)+1}$, $\dots$, $g^{n,m}_{m-1} + G_{\eta^1_\delta(n)+1}$ freely generates $H^\delta_{\eta^1_\delta(n)+1,m} / G_{\eta^1_\delta(n)+1}$),
			possible by e.g.\ using Lemma \ref{tee},
			and we can assume that $x_{\eta^1_\delta(n)} \in Y_{\delta,n}$ holds,
		\mn
		\item \label{uhdef4} define $u_{\delta,n} \subseteq \eta^1_\delta(n)$ by the equality 
		   $Y_{\delta, n} =  \{ x_\alpha: \ \alpha \in u_{\delta,n}\} \cup \{ x_{\eta^1_\delta(n)}\}$.
		\end{enumerate}			
	\end{definition}
 \noindent	We claim that 
 \newcounter{booocou}
		\begin{enumerate}[label = $\blacktriangle_{\arabic*}$, ref = $\blacktriangle_{\arabic*}$]
		\setcounter{enumi}{\value{booocou}}
		\item \label{kisebbk} if $\delta \in S$, $n \in \omega$ is fixed, and  $m$ is chosen such that $\eta^1_\delta(n) \in [\gamma^\delta_{m-1}, \gamma^\delta_m)$, then it follows from the demand above on $Y_{\delta,n}$ that for any $k \leq m$
		$$\clpr(G_{\eta^1_\delta(n)+1} \cup
		\{x_{\delta+\ell}:\ell<k\}) = G_{\eta^1_\delta(n)+1} +
		\clpr(\{x_{\delta+\ell}:\ell<k\} \cup Y_{\delta,n}),$$
		and so we can write $g^{n,k}_0,g^{n,k}_1, \dots, g^{n,k}_{k-1}$ as a $\bbQ$-linear combination of elements of $\{x_{\delta+\ell}:\ell<k\} \cup Y_{\delta,n}$ so that $g^{n,k}_0 + G_{\eta^1_\delta(n)+1}$,  $g^{n,k}_1 + G_{\eta^1_\delta(n)+1}$, $\dots$, $g^{n,k}_{k-1} + G_{\eta^1_\delta(n)+1}$ freely generates $H^\delta_{\eta^1_\delta(n)+1,k} / G_{\eta^1_\delta(n)+1}$.	
		
		\stepcounter{booocou}
	\end{enumerate}
	(Why is \ref{kisebbk} true? If $g \in \clpr(G_{\eta^1_\delta(n)+1} \cup
	\{x_{\delta+\ell}:\ell<k\}$, then $g \in \clpr(G_{\eta^1_\delta(n)+1} \cup
	\{x_{\delta+\ell}:\ell<m\}$ obviously, and so 
	 $$g = g' + \sum_{j<m}  k_j \cdot g^{n,m}_j$$
	  with $k_j$'s being integers, $g^{n,m}_j$'s are from \ref{uhdef}, $g' \in G_{\eta^1_\delta(n)+1}$. But since each $g^{n,m}_j \in \clpr(G_{\eta^1_\delta(n)+1} \cup
	  \{x_{\delta+\ell}:\ell<m\})$, and the set $\{x_{\delta+n}: \ n <m \}$ is independent over $G_\delta$, hence over $G_{\eta^1_\delta(n)+1}$ too, so 
	  $$\sum_{j<m}  k_j \cdot g^{n,m}_j \in \clpr(G_{\eta^1_\delta(n)+1} \cup
	  	\{x_{\delta+\ell}:\ell<k\}) = \clpr(\{x_\alpha: \ \alpha < \eta^1_\delta(n)+1 \} \cup
	  	\{x_{\delta+\ell}:\ell<k\}).$$
	  	But we know that 
	  		$$\sum_{j<m}  k_j \cdot g^{n,m}_j \in \clpr(Y_n \cup 
	  	\{x_{\delta+\ell}:\ell<m\}),$$
	  	  where $Y_n \subseteq \{x_\alpha: \ \alpha < \eta^1_\delta(n)+1 \}$. So recalling that $\{x_\beta: \ \beta < \omega_1\}$ is an independent system, we get that 
	  $$\sum_{j<m}  k_j \cdot g^{n,m}_j \in \clpr(Y_n \cup
	  \{x_{\delta+\ell}:\ell<m\}) \cap \clpr(\{x_\alpha: \ \alpha < \eta^1_\delta(n)+1 \} \cup
	  \{x_{\delta+\ell}:\ell<k\}) =$$
	  $$= \clpr(Y_n \cup
	  \{x_{\delta+\ell}:\ell<k\},$$
	  so $g \in \clpr(G_{\eta^1_\delta(n)+1}) + \clpr(Y_n \cup
	  \{x_{\delta+\ell}:\ell<k\},$ indeed, as desired.)

	Moreover, (using that $\eta^1_\delta(n) \in [\gamma^\delta_{m-1}, \gamma^\delta_m)$ implies $\eta^1_\delta(n)+1 \leq \gamma^\delta_m$),  clause \ref{ee3} in the premises says that if $k \geq m$ then $H^\delta_{\eta^1_\delta(n)+1,k} = H^\delta_{\eta^1_\delta(n)+1,m} \oplus ( \oplus_{j=m}^{k-1} \bbZ x_{\delta+j})$, so  
		\begin{enumerate}[label = $\blacktriangle_{\arabic*}$, ref = $\blacktriangle_{\arabic*}$]
		\setcounter{enumi}{\value{booocou}}
		\item \label{foszamolas} if $\delta \in S$, $n \in \omega$ are fixed, then for  every $k \geq m$ (hence for every $k \in \omega$)
		$$\clpr(G_{\eta^1_\delta(n)+1} \cup
		\{x_{\delta+\ell}:\ell<k\}) = G_{\eta^1_\delta(n)+1} +
		\clpr(\{x_{\delta+\ell}:\ell<k\} \cup Y_{\delta,n}).$$
		\stepcounter{booocou}
	\end{enumerate}
	Now we recall that if $k$ is fixed, then $(\eta^1_\delta(n)$, $\eta^1_\delta(n+1)) \cap v^\delta_k = \emptyset$ by condition \ref{ee2} we obtain:
	$$	\clpr(G_{\eta^1_\delta(n+1)} \cup
	\{x_{\delta+\ell}:\ell<k\}) = \clpr(G_{\eta^1_\delta(n)+1} \cup
	\{x_{\delta+\ell}:\ell<k\}) + G_{\eta^1_\delta(n+1)}.$$
	this means that in \ref{foszamolas} we could as well write $G_{\eta^1_\delta(n+1)}$ in place of $G_{\eta^1_\delta(n)+1}$:
	\begin{enumerate}[label = $\blacktriangle_{\arabic*}$, ref = $\blacktriangle_{\arabic*}$]
	\setcounter{enumi}{\value{booocou}}
	\item \label{foszamolas2} 
		 ($\forall \delta \in S$, $ \forall n \in \omega$) for  every $k \in \omega$
		$$\clpr(G_{\eta^1_\delta(n+1)} \cup
		\{x_{\delta+\ell}:\ell<k\}) = G_{\eta^1_\delta(n+1)} +
		\clpr(\{x_{\delta+\ell}:\ell<k\} \cup Y_{\delta,n}).$$
				
		\stepcounter{booocou}
	\end{enumerate}
		Similarly,
	\begin{enumerate}[label = $\blacktriangle_{\arabic*}$, ref = $\blacktriangle_{\arabic*}$]
		\setcounter{enumi}{\value{booocou}}	
	\item \label{foszamV}
		 for each $V \subseteq \eta^1_\delta(n+1)$ with $u_{\delta,n} \subseteq V$, letting $G_V := \clpr(\{x_\alpha: \ \alpha \in V \})$, and $G^{+\delta,k}_V := \clpr(\{x_\alpha: \ \alpha \in V \cup [\delta,\delta+k]\})$ an easy calculation yields that
			$$G^{+\delta,k}_V= G_V + \clpr(\{x_{\delta+\ell}:\ell<k\} \cup Y_{\delta,n}),$$
			equivalently, there exists some independent set $\{g^{V,k}_0,g^{V,k}_1, \dots, g^{V,k}_{k-1}\} \in \clpr(\{x_{\delta+\ell}:\ell<k\} \cup Y_{\delta,n})$, such that 
			$$G^{+\delta,k}_V / G_V \simeq \oplus_{j <k} \bbZ (g^{V,k}_j + G_{V}),$$
			which equivalence is justified by the following: first let $H_1 = \clpr(\{x_{\delta+\ell}:\ell<k\} \cup Y_{\delta,n})$, $H_0 = \clpr(Y_{\delta,n})$, and apply Fact $\ref{fa2}$, so $H_1 = H_0 \oplus (\oplus_{i<k} \bbZ g_i)$ (where for each $i <k$  $g_i \in \clpr(Y_{\delta,n} \cup \{x_{\delta+j} \ j \leq i\})$ holds). Now observe, that $H_1 \cap G_V = H_0$, so by the second isomorphism theorem the mapping $g + H_0 \mapsto g + G_V$ is an isomorphism from $H_1 / (H_1 \cap G_V)$ onto $H_1 + G_V / G_V$, as desired. But note that by \ref{foszamolas} $H_1 + G_V$ must be $G^{+\delta,k}_V$, otherwise the isomorphism between 
			$$H_1 + G_V / G_V$$ and 
			$$H_1+G_V + G_{\eta^1_\delta(n+1)} / G_{\eta^1_\delta(n+1)} =\clpr(G_{\eta^1_\delta(n+1)} \cup
			\{x_{\delta+\ell}:\ell<k\}) / G_{\eta^1_\delta(n+1)}$$ 
			(viewing both groups as subgroups of $\oplus_{j<k} \bbQ x_{\delta_j})$
			cannot be surjective.
	
		\stepcounter{booocou}
		\end{enumerate}

	\mn
	
	Now we can turn to  proving  that $G$ is W$_\omega$-group.
	In order to do this we fix
	\begin{enumerate}[label = $\blacktriangle_{\arabic*}$, ref = $\blacktriangle_{\arabic*}$]
		\setcounter{enumi}{\value{booocou}}
		\item \label{20.1} 
	 the groups $K \leq H$, and a homomorphism $\varphi$  from $H$ onto $G$,
	where $K$ is the kernel of $\varphi$, and $K$ is a countable free
	group, so w.l.o.g.\ $K = \bbZ_\omega = \oplus_{j < \omega} \bbZ$.
		\stepcounter{booocou}
\end{enumerate}

	We reduce the task of finding a suitable homomorphism to a more general combinatorial problem. This will need some preparations: from Definition $\ref{a4}$ to $\ref{a6b}$, in this section we prepare the ground for stating Theorem $\ref{1.1}$, as Section $\ref{s2}$ is devoted exclusively to the proof of that theorem. At the end of the section Definition $\ref{PSdf}$ and Lemma $\ref{befej}$ put the present problem (i.e.\ that of Proposition \ref{ekp} \ref{eki}) in the frame of Theorem $\ref{1.1}$, and justify that.
	
	\begin{definition}
		\label{a4}
		 We say that ${\gp}^0$ is a $S$-uniformization frame \when \,
		${\gp}^0 = (S,\bar\eta, \bar u)  = (S^{\gp^0},\bar \eta^{\gp^0}, \bar u^{\gp^0})$ satisfies
		\mn
		\begin{enumerate}[label = $(\alph*)$, ref = $(\alph*)$]
			\item   $\langle \eta_\delta:\delta \in S \rangle$ is an 
			$S$-ladder system, i.e., $S \subseteq \omega_1$ is a stationary set of
			limit ordinals and for each $\delta \in S,\eta_\delta$ is an 
			$\omega$-sequence of ordinals $<\delta$, 
			strictly increasing, with limit $\delta$; 
			\sn
			\item   $\bar{u} =  \langle u_{\delta,n}: \delta \in S, \ n \in \omega \rangle $, with $u_{\delta,n}$ being a 
			finite subset of $\eta_\delta(n)$, 
		%	\item $\bar{u} = \langle u_\alpha: \ \alpha < \omega_1\rangle$ is a system of sets with each $u_\alpha \in [\alpha]^{<\aleph_0}$, and for $\delta \in S$ we prescribe $u_{\delta+n} = u^*_{\delta,n}$,
		%	\item 	   $\bar h = \langle h_\delta:\delta \in S \rangle$ and for each
		%	$\delta\in S$ $h_\delta$ is a function from $\omega$ to $\omega$.
		%	\sn
		\end{enumerate}
		\mn
	\end{definition}
	\begin{definition}
		\label{1.2}
		Let ${\gp}^0$ be a $S$-uniformization frame. 
		\begin{enumerate}[label = $\arabic*)$, ref = $\arabic*)$]
		\item We define the sets 
		$$b^\delta_{n,k} = (b^\delta_{n,k})^{\gp^0} = \bigcup_{j \leq n} (u_{\delta,j} \cup \{ \eta_\delta(j) \}) \cup [\delta, \delta+k] \ \ (\delta \in S, \ n,k \in \omega),$$
			we call $b^\delta_{n,k}$ a $\gp^0$-basic set.
		\item  We let $$\cB = \cB^{\gp^0} = \{b^\delta_{n,k}: \ \delta \in S, n,k \in \omega \},$$
		and call $\cB$ the base of $\gp^0$.
		
		\item \label{1.2)2} A subset $x$ of $\omega_1$ is called ${\gp}^0$-closed 
		(or closed if ${\gp}^0$ is clear from the context) \when \,:
		\mn
			\begin{enumerate}[label = (\greek*), ref = (\greek*)]
			\item \label{1.2)2a}   $\epsilon+1\in x\Rightarrow \epsilon\in x$,
			\sn
			\item \label{1.2)2b}  if $\delta\in S\cap x$, \then \  for some $0 < N = N(\delta) \leq \omega$
			for every $n<\omega$:
			\begin{enumerate}
				\item[$(i)$]   $\eta_\delta(n) \in x \text{ iff } n < N$,
				\sn
				\item[$(ii)$] $[\eta_\delta(n), \eta_\delta(n+1)) \cap x \ne \emptyset$ $\to$ $\eta_\delta(n) \in x$,
				\sn
				\item[$(iii)$]   if $\eta_\delta(n) \in x$ then $u_{\delta,n} \subseteq x$.
			\end{enumerate}
		\end{enumerate}

	\end{enumerate}

	\end{definition}

	\begin{definition} \label{pprobd}
	We say $\gp = (S,\bar \eta,\bar u, \Psi) = (S^{\gp},\bar \eta^{\gp},
		\bar u^{\gp}, \Psi^{\gp})$ is a special $S$-uniformization
		problem when $\gp^0 = (S^{\gp},\bar \eta^{\gp},\bar u^{\gp})$ is an $S$-uniformization frame, and there exists a countable set $C$ such that
		\mn
		\begin{enumerate}[label = $\boxdot_{\alph*}$, ref = $\boxdot_{\alph*}$]
			\item \label{pda} $\Psi$ is a function with domain $\cB^{\gp^0}$, and for each $b \in \cB^{\gp^0}$ $\Psi(b)$ is a countable non-empty family of functions with domain $b$, range $\subseteq C$ satisfying that for any $n$ there is a
			countable $\Upsilon_n = \Upsilon_n^\gp$ such that if
			$b \in \cB^{\gp_0}$, $|b|=n$ then
			$\{f\circ \text{ OP}_{b, n}: f\in \Psi(b)\}\in\Upsilon_n$.
			\item   \label{pdb} If we define the function $\Psi^+$  
			 with domain $[\omega_1]^{<\aleph_0}=\{u:u\subseteq
			\omega_1$ finite$\}$ by $\Psi^+(u) = \{f \in \ ^uC: \ \forall b \subseteq u \ (b \in \cB^{\gp_0} \Rightarrow f \upharpoonright b \in  \Psi(b))\}$, then for each $u \in [\omega_1]^{<\aleph_0}$  $\Psi^+(u)$ is a  non-empty family (which must be countable, of course). 
			\item   \label{pdc}  If $b, b' \in \cB^{\gp^0}$, $b \subseteq b'$ and
			$f \in \Psi(b')$ \then \, $f \restriction b\in \Psi(b)$.
			\sn
			
			\item  \label{pdd}  If $u\subseteq v\in [\omega_1]^{<\aleph_0}$,
			$f \in \Psi^+(u)$ and $u$ is finite and $\gp^0$-closed \then \, $(\exists f') [f\subseteq f' \in \Psi^+(v)]$.
			\sn
		
			%\item (FOLLOWS)   Assume that $u\subseteq \delta+m+1 \subseteq
			%\omega_1$ and $\{\delta,\delta+1,\ldots,\delta + m\} \subseteq
			%u,h_\delta(m) \le n < \omega,\{\eta_\delta(\ell):\ell < n\} \subseteq u
			%\cap \delta \subseteq \eta_\delta(n)$ and $u$ 
			%is finite ${\gp}^0$-closed (see Definition \ref{1.2} below) 
			%and $u \cap \delta \subseteq v \subseteq \eta_\delta(n) < \delta$ and 
			%$v$ finite.
			
			%If $f_1\in \Psi^+(u),f_2\in \Psi^+(v)$, and $f_1 \restriction 
			%(u\cap \delta)=f_2\restriction(u\cap \delta)$ \then \,
			%$f_1\cup f_2\in \Psi^+(u\cup v)$.
			%\sn
			%\item   (FOLLOWS) Assume $\alpha \notin S^{\gp}_* :=
			%\{\delta +n:\delta \in S$ and
			%$n < \omega\}$ and $u_\alpha \subseteq u \subseteq \alpha$, $u$ is
			%finite. 
			
			%If $f_2 \in \Psi^+(\{\alpha\} \cup u_\alpha)$ and 
			%$f_1 \in \Psi^+(u)$ and $f_2 \rest u_2 = f_1 \rest u_\alpha$ then $f_1 \cup f_2
			%\in \Psi^+(u \cup \{\alpha\})$ (we can assume $u_\alpha = \emptyset$ for
			%such $\alpha$'s; this causes no real change but is not used).
		\end{enumerate}
	\end{definition}

	\begin{claim} \label{Upsn}
		If $\gp = (S,\bar \eta,\bar u, \Psi) = (S^{\gp},\bar \eta^{\gp},
		\bar u^{\gp}, \Psi^{\gp})$ is a special $S$-uniformization problem, then
		for each $n < \omega$ there exists a countable set $\Upsilon^+_n$ such that:
		\[ \forall x \in [\omega_1]^{n}: \ \{f\circ \text{ OP}_{x, n}: f\in \Psi^+(x)\} \in\Upsilon^+_n.   \]

	\end{claim}
	\begin{PROOF}{Claim \ref{Upsn}}
		If $x \in [\omega_1]^n$, then clearly there are only finitely many $b \in \cB^{\gp^0}$ with $b \subseteq x$. Now by the definition of $\Psi^+$ in \ref{pdb}
		the set $$\{f\circ \text{ OP}_{x, n}: f\in \Psi^+(x)\}$$
		can be coded by  $\langle \text{ OP}_{x, n}^{-1}(b): \ b \in \cB^{\gp^0}, b \subseteq x  \rangle$, i.e.\ the relative position of the finitely many basic sets that are included, and
		$$ \left\langle \{f\circ \text{ OP}_{b, n}: f\in \Psi(b)\}: \ b\subseteq x, \ b \in \cB^{\gp^0} \right\rangle. $$
		But note that this latter sequence has its entries from $\bigcup_{k \in \omega} \Upsilon_k$ by \ref{pda}, so  there are only countably many ways to choose the parameters (which will define the set in question), we are done.
		
	\end{PROOF}

		\begin{definition}\label{a6b}{ \ }
		\begin{enumerate}[label = $\arabic*)$, ref = $\arabic*)$]
		\item We say $f$ is a solution of ${\gp}$, a special 
		$S$-uniformization problem \when \, $f$ 
		is a function with domain $\omega_1$ such that for
		every $u\in [\omega_1]^{<\aleph_0}$ we have $f\restriction u\in
		\Psi^+(u)$ (equivalently, for each $b \in \cB$, $b \subseteq u$ $f \restriction b \in \Psi(b)$).
		% so being a solution depends on $\bar \eta^{\gp}, \Psi^{\gp}$, $\bar u^{\gp}$,
		\mn
		\item We say that $f$ is a partial solution of ${\gp}$ \If  \, $f$ is
		a function with domain $\subseteq \omega_1$ such that for every finite
		$u \subseteq \text{ Dom}(f)$ we have $f \restriction u \in
		\Psi^+(u)$. 
		\mn
	\item We say ${\gp}$ is simple if $\alpha < \omega_1 \Rightarrow
		\alpha + \omega \notin S^{\gp}$. 
	\mn	
	\end{enumerate}
	\end{definition}
	\begin{convention} 	From now on if $\gp$ is a special $S$-uniformization problem, then  $\gp$-closed will mean $\gp^0$-closed.
	\end{convention}

	Here we state the main theorem of Section $\ref{s2}$.	
	\begin{namedthm*}{\textbf{Theorem \ref{1.4}}}
		Let $S$ be a stationary set of limit
		ordinals $< \omega_1$ [which is simple, i.e., $(\forall \alpha <
		\omega_1)(\alpha + \omega \notin S)]$ and 
		$\bar \eta^*$ be an $S$-ladder system.  \Then \,
		$(\mathbf{D})_{S,\bar \eta^*} \Leftrightarrow (\mathbf{E})_{S,\bar \eta^*}$ where
		\mn
		\begin{enumerate}
			\item[$(\mathbf{D})_{S,\bar \eta^*}$]  for any ladder system $\bar \eta$ on $S$ if $\eta$ is
			very similar to $\bar \eta^*$, \then \, it has $\aleph_0$-uniformization
			\sn
			\item[$(\mathbf{E})_{S,\bar \eta^*}$]   for every ladder system  $\bar \eta$ very
			similar to $\bar \eta^*$, for every  simple special
			$S$-uniformization problem  ${\gp}$ with $\bar \eta^{\gp} = \bar \eta$,
			 ${\gp}$ has a solution.
		\end{enumerate}

	\end{namedthm*}

%	\begin{notation}
%		\label{a6} 
%		As we use $u_\ell,u_\alpha$ also for other purposes, even when ${\gp}$
%		is constant we do not write $u_\alpha$ for $u^{\gp}_\alpha$ (as we may
%		do for $S^{\gp},\Phi,u_{\delta,n}$, etc.) but write 
%		$u^{\gp}_\alpha$ or $u^*_\alpha$.
%	\end{notation}
	
	Now we only have to translate our lifting problem to a special $S$-uniformization problem to finish the proof of Theorem $\ref{2.1}$.
	
	Recalling \ref{20.1} fix	a map $F: \{ x_\alpha: \ \alpha < \omega_1\}\to H$  such that $\varphi \circ  F(x_\alpha)=x_\alpha$,
	(note that $F$ does not necessarily induce a homomorphism from $G = \clpr(\{x_\alpha: \ \alpha < \omega_1\}$, only from the generated free group $\cll(\{x_\alpha: \ \alpha < \omega_1\})$). Note that in the case that $F$ induces a homomorphism $\tilde{F}:\clpr(\{x_\alpha: \ \alpha < \omega_1\} \to H$, then clearly $\varphi \circ \tilde{F} = \text{id}_G$.
	Now observe that 
		\begin{enumerate}[label = $\boxplus_{\arabic*}$, ref = $\boxplus_{\arabic*}$]
		\setcounter{enumi}{\value{bocou}}
		\item the stationary set $S$ together with $\bar \eta^1$, $\bar u$ from Definition $\ref{uhdef}$ form an $S$-uniformization frame $\gp^0$.
		\stepcounter{bocou}
	\end{enumerate}
	So if we can form an $S$-uniformization problem, each solution of which yields a suitable  homomorphism from the entire $G$, then we will be done.
	\begin{definition} \label{PSdf}
		Fix $\delta \in S$, and $b \in \cB^\gp$ (from Definition $\ref{1.2}$, so $b = \bigcup_{j \leq n} (u_{\delta,j} \cup \{ \eta^1_\delta(j)\} \cup [\delta, \delta+k]$ for some $n$, $k$).
		\begin{enumerate}[label = $(\odot)$, ref = $(\odot)$]
			\item Define $f \in \Psi(b)$, iff $f: b \to \bbZ_\omega = \oplus_{k \in \omega} \bbZ$, and the function $F^{+f}: \{x_\alpha: \ \alpha \in b\} \to H$, $F^{+f}(x_\alpha)=  F(x_\alpha)+ f(\alpha)$ extends to a homomorphism $\widetilde{F^{+f}}: \clpr_G(\{x_\alpha: \ \alpha \in b\}) \to H$.
		\end{enumerate}
	\end{definition}
	
	We encapsulate the missing claims in the following assertion.
	\begin{lemma}\label{befej}{\ }
		\begin{enumerate}[label = $(\odot)_{\arabic*}$, ref = $(\odot)_{\arabic*}$]
			\item $\Psi$ (from Definition $\ref{PSdf}$) together with the $S$-uniformization frame $\gp^0$ (from Definition $\ref{uhdef}$) satisfies the demands in Definition $\ref{pprobd}$,
		\end{enumerate}
		moreover,
	\begin{enumerate}[label = $(\odot)_{\arabic*}$, ref = $(\odot)_{\arabic*}$]
		\stepcounter{enumi}
		\item for every $\gp$-closed set $U$, and partial solution $f$ of $\gp$ with $\dom(f) = U$ the defined map
			 $$ \begin{array}{ll} F^{+f}: & \{x_\alpha: \ \alpha \in U\} \to H, \\ 
			 x_\alpha  \mapsto &  F(x_\alpha)+ f(\alpha) 
			 \end{array} $$
		induces a homomorphism $\widetilde{F^{+f}}: \clpr_G(\{ x_\alpha: \ \alpha \in U \}) \to H$.
	\end{enumerate}
		
	\end{lemma}
	\begin{PROOF}{Lemma \ref{befej}}
		First recall that every countable subgroup of $G$ is free, and every free group is W$_\omega$-group (just pick a set freely generating the group and pick an arbitrary preimage for each generator), in particular:
			\newcounter{bibocou} \setcounter{bibocou}{0}
			\begin{enumerate}[label = $\boxminus_{\arabic*}$, ref = $\boxminus_{\arabic*}$]
			\setcounter{enumi}{\value{bibocou}}
			\item \label{cl0a} for every pure  subgroup $G_* = \clpr_G(\{g_i: \ i < N \leq \omega \}) \leq G$ of either finite or countable infinite rank there exists a homomorphism $\psi : G_* \to H$ with $\varphi \circ \psi = \text{id}_{G_*}$.
			\stepcounter{bibocou}
		\end{enumerate}
		Using Fact $\ref{fa2}$ for every pure  subgroup $G_* = \clpr_G(\{g_i: \ i < N \leq \omega \}) \leq G$ of either finite or countable infinite rank, and pure subgroup $G_{**}$ of $G_*$ of finite rank $G_{**}$ is a direct summand of $G_*$ (i.e.\ $G_* = G_{**} \oplus G'$), (and of course  $G'$ is free) so
		\begin{enumerate}[label = $\boxminus_{\arabic*}$, ref = $\boxminus_{\arabic*}$]
			\setcounter{enumi}{\value{bibocou}}
			\item \label{cl0a2} for every pure  subgroup $G_* = \clpr_G(\{g_i: \ i < N \leq \omega \}) \leq G$ of either finite or countable infinite rank, pure subgroup $G_{**}$ of $G_*$ of finite rank, and homomorphism $\psi^0: G_{**} \to H$ with $\varphi \circ \psi^0 = \text{id}_{G_{**}}$ there exists a homomorphism $\psi : G_* \to H$ extending $\psi^0$ with $\varphi \circ \psi = \text{id}_{G_*}$.
			\stepcounter{bibocou}
		\end{enumerate}
	
		For the requirement \ref{pda} it is clear that for each finite $b \in \cB$ the collection  $\Psi(b)$ of mappings is non-empty. We claim that for any fixed $b \in \cB$, $|b| = n$, 
		 	\begin{enumerate}[label = $\boxminus_{\arabic*}$, ref = $\boxminus_{\arabic*}$]
			\setcounter{enumi}{\value{bibocou}}
			\item \label{claa} there exists $\bar k = \langle k_j^\alpha: \ j < n, \ \alpha \in b \rangle$ with ($\forall j <n$, $\forall \alpha \in b$) $k_j^\alpha \in \bbZ$, such that for each $f_0 \in \Psi(b)$ we have
			 $$\Psi(b) = \left\{ f_{\bar k, \bar d} : \ \bar d  \in \ ^n \bbZ_\omega \right\}$$
			 (where $f_{\bar k, \bar d}(\alpha) = f_0(\alpha) +   \sum_{\ell =0}^{n-1} k^\alpha_\ell \cdot d_\ell)$),
			 			\stepcounter{bibocou}
		\end{enumerate}	
		which would clearly complete the proof of \ref{pda}. So	pick a basis $\{e_0,e_1, \dots e_{n-1} \} \subseteq \clpr_G(\{x_i: \ i \in b\})$, write $x_\alpha = \sum_{j = 0}^{n-1} k_j^\alpha e_j$,  and fix $f_0 \in \Psi(b)$. As for each $f: b \to \bbZ_\omega$ there is at most only one way $F^{+f}: \{x_\alpha: \ \alpha \in b\} \to H$ can extend to a homomorphism $\widetilde{F^{+f}}$ from $\clpr(\{x_\alpha: \ \alpha \in b\})$ there is obviously a bijection between the sets $\{\widetilde{F^{+f}} - \widetilde{F^{+f_0}}: \ f \in \Psi(b)\}$ and  $^n\bbZ_\omega$ (i.e.\ the attained value on the base $\{e_0,e_1, \dots e_{n-1} \}$). Changing back from the base $\{e_0, e_1, \dots, e_{n-1}\}$ to $\{x_\alpha: \ \alpha \in b\}$ it is easy to see \ref{claa}.
		For \ref{pdb} note that
			\begin{enumerate}[label = $\boxminus_{\arabic*}$, ref = $\boxminus_{\arabic*}$]
			\setcounter{enumi}{\value{bibocou}}
			\item whenever $f: u \to \bbZ_\omega$ is such that the mapping $F^{+f}: \{x_\alpha: \ \alpha \in u \}$ extends to a homomorphism from $\clpr_G(\{x_\alpha: \ \alpha \in u\}$ it is obviously a homomorphism from $\clpr_G(\{x_\alpha: \ \alpha \in b\}$ for any $\gp$-closed $b \subseteq u$,
			\stepcounter{bibocou}
		\end{enumerate}	
		and similarly \ref{pdc} is as well obvious.
		For \ref{pdd} recall \ref{cl0a} and \ref{cl0a2} so it is enough to prove that if $U$ is $\gp$-closed, $f:U \to \bbZ_\omega$ satisfies that for each $b \in \cB$ $f \restriction b \in \Psi(b)$, then $F^{+f}$ extends to a homomorphism  $\widetilde{F^{+f}}: \clpr_G(\{x_\alpha: \ \alpha \in U\}) \to H$. (And so if $U$ is finite, $U \subseteq V$, then \ref{cl0a2} means that we can further extend the homomorphism, so recalling Definition $\ref{PSdf}$ and \ref{pdb} the corresponding mapping from $V$ will belong to $\Psi^+(V)$.) 
		
		So fix $U$ and $f$ with $[(b \subseteq U)$ $\to$ $(f \restriction b \in \Psi(b))]$. Let $\delta+k < \omega_1$ (with $\delta$ limit) be minimal such that $F^{+f}$ does not extend to $\clpr_G(\{x_\alpha: \ \alpha \in U \cap (\delta+k)\})$.
		Since this extension (if exists) is unique it is easy to see that only successor steps may be problematic, so $\delta+k-1 \in U$. Moreover, $\delta \in S$ necessarily, since the $x_\alpha$'s were given by Claim $\ref{hulyetekkl}$, and so otherwise $G_{\delta+k} = G_\delta \oplus (\oplus_{i<k} \bbZ x_{\delta+i})$, so if 
		$F^{+f} \restriction \{ x_\alpha: \ \alpha \in U \cap \delta\}$ extends to $G_{U \cap \delta}$ (and $ F^{+f} \restriction \{ x_\alpha: \ \alpha \in [\delta, \delta+k)\}$ clearly extends to the $k$-rank abelian free group), then it extends to their direct sums as well.
		
		Now there exists $u = \{\alpha_0, \alpha_1, \dots, \alpha_{m-1}\} \subseteq U \cap (\delta+\omega)$ finite, and $k_0,k_1, \dots, k_{m-1} \in \bbZ$, $d \in  \bbZ \setminus \{0\}$ such that
	$$g^* = \frac{k_0 x_{\alpha_0} + k_1  x_{\alpha_1} + \dots + k_{m-1} x_{\alpha_{m-1}}}{d} \in \clpr(\{ x_\alpha: \ \alpha \in U \cap \delta+\omega\}) \leq G,$$
		but
	$$\sum_{j=0}^{m-1} k_j\left(F(x_{\alpha_j})+ f(\alpha_j)\right) \in H \text{ is not divisible by } d.$$
		Now w.l.o.g.\ (since the $\gp$-closedness of $U$ implies that $(\forall \beta < \omega_1):$ ($\beta+1 \in U \ \to \ \ \beta \in U$):
		 \begin{enumerate}[label = $\boxminus_{\arabic*}$, ref = $\boxminus_{\arabic*}$]
			\setcounter{enumi}{\value{bibocou}}
			\item for some $k \leq m$ we have $\alpha_{m-k+ j} = \delta + j $ for $j =0,1, \dots, k-1$,
			whereas $\{ \alpha_0, \alpha_1, \dots, \alpha_{m-k-1} \} \subseteq \delta$.
			\stepcounter{bibocou}
		\end{enumerate}
		 By the minimality of $\delta$ the function $F^{+f}$ extends to a homomorphism from $\clpr_G(\{x_\alpha: \ \alpha \in U \cap \delta\})$. As $U$ is $\gp$-closed, (recalling Definition $\ref{1.2}$/ \ref{1.2)2}) either $\ran(\eta^1_\delta) \subseteq U$, and then each $u_{\delta,n} \subseteq U$, or there exists a minimal $n$ such that $U \cap \delta \subseteq \eta^1_\delta(n)$. No matter which case holds, we have
		 \begin{enumerate}[label = $\boxminus_{\arabic*}$, ref = $\boxminus_{\arabic*}$]
		 	\setcounter{enumi}{\value{bibocou}}
		 	\item $\forall j \in \omega$: $(\eta^1_\delta(j) \in U) \to (u_{\delta,j} \subseteq U)$,
		 	\stepcounter{bibocou}
		 \end{enumerate}	
	 	therefore as $u_{\delta,j} \subseteq \eta^1_\delta(j)$'s are finite (\ref{uhdef3}, \ref{uhdef4} in Definition $\ref{uhdef}$) w.l.o.g.\
	 	 \begin{enumerate}[label = $\boxminus_{\arabic*}$, ref = $\boxminus_{\arabic*}$]
	 		\setcounter{enumi}{\value{bibocou}}
	 		\item if $\eta^1_\delta(j) \in \{\alpha_0, \alpha_1, \dots, \alpha_{m -k-1}\}$, then $\bigcup_{\ell \leq j} (u_{\delta,\ell} \cup \{\eta^1_\delta(\ell)\}) \subseteq \{\alpha_0, \alpha_1, \dots, \alpha_{m-k-1}\}$.
	 		\stepcounter{bibocou}
	 	\end{enumerate}
 		Let $j_0 \in \omega$ be such that
 		\begin{enumerate}[label = $\boxminus_{\arabic*}$, ref = $\boxminus_{\arabic*}$]
 			\setcounter{enumi}{\value{bibocou}}
 			\item  $\bigcup_{\ell \leq j_0} (u_{\delta,\ell} \cup \{ \eta^1_\delta(\ell)\}) \subseteq \{\alpha_0, \alpha_1, \dots, \alpha_{m-k-1}\} \subseteq \eta^1_\delta(j_0+1)$.
 			\stepcounter{bibocou}
 		\end{enumerate}
 			Therefore by \ref{uhdef4} $Y_{\delta,j_0} \subseteq \{ x_{\alpha_i}: \ i \in m \}$, and   by \ref{foszamolas2} we obtain $g^* = g_0 + g_1$, where $g_1 \in \clpr(Y_{\delta,j_0} \cup \{ x_\alpha: \ \alpha \in [\delta, \delta+k)\}$, and $g_0 \in G_{\eta^1_\delta(j_0+1)}$. But as $g^* \in \clpr(\{x_\alpha: \ \alpha \in (U \cap \eta^1_\delta(j_0+1)) \cup [\delta, \delta+k)\})$, $Y_{\delta,j_0} \subseteq \{x_\alpha: \ \alpha \in U \cap \eta^1_\delta(j_0+1)\}$, obviously $g_1 \in \clpr(\{ x_\alpha: \ \alpha \in U \cap \eta^1_\delta(j_0+1)\})$. But $F^{+f} \restriction \{x_\alpha: \ \alpha \in U \cap \delta\}$ and $F^{+f} \restriction (Y_{\delta,j_0} \cup \{x_{\delta+i}: \ i <k\}) = F^{+f} \restriction \{x_\alpha: \ \alpha \in b^\delta_{j_0,k}\}$ both extend to the corresponding generated pure subgroups, so to their sums, so to $g^*$ as well, which is a contradiction.

	\end{PROOF}

 \end{PROOF}

	Having finished the proof of the last ingredient, we can put all the necessary claims together to argue \ref{eki} of Proposition \ref{ekp}.
	\begin{PROOF}{Clause \ (ii) \ of \  Proposition \ref{ekp}}(Clause  $\ref{eki}$  of   Proposition \ref{ekp})
		Suppose that $S \subseteq \Omega$ is stationary, $\bar \eta$ is a ladder system for which ran($\eta_\delta) \subseteq \Omega$ for each $\delta \in S$ (so necessarily $S$ is simple, $\alpha+  \omega \notin S$ for any $\alpha$).
		Let $G$ be a group of size $\aleph_1$, $\bar G' = \langle G'_\alpha: \ \alpha < \omega_1 \rangle$ a filtration of $G$, and the ladder system $\bar \eta^*$ are such that $\bar \eta^*$ is very similar to $\bar \eta$, $(*)^+_{S, \bar \eta^*}(G, \bar G')$ holds from the proposition, i.e. clauses \ref{alp0}-\ref{varpp0} are satisfied. This implies that $(**)_S(G, \bar G')$ from Lemma \ref{fil1} holds, and we can apply Claim \ref{hulyetekkl}. Let $\langle x_\alpha: \ \alpha < \omega_1 \rangle$, and the ladder system $\bar \eta^1$ (and $\gamma^\delta_n$'s) be given by the claim. Moreover, since we are in the proof of Clause  $\ref{eki}$  of   Proposition \ref{ekp}, we assume \ref{varpp0}:
			\begin{enumerate}[label = $(\varepsilon)^+$, ref = $(\varepsilon)^+$]
				\item 
		for each $\delta \in S$ for some sequence $\langle y_{\delta+\ell}: \ \ell < \omega \rangle \in \ ^\omega G'_{\delta+\omega}$ which is  a maximal independent sequence over $G'_\delta$, we have for each $n \in \omega$:
		\[  \left\{ \begin{array}{ll} \alpha<\delta: & \clpr\left( G'_{\alpha+1} \cup
			\{y_{\delta+\ell}:\ell<n\} \right) \ne \\ &
			\clpr\left( G'_\alpha \cup
			\{ y_{\delta+\ell}: \ \ell < n \} \right) +G'_{\alpha+1} \end{array} \right\} \subseteq^* \ran(\eta^*_\delta).\]
		\end{enumerate}
		 This means that clause \ref{varpp} stated in Claim \ref{hulyetekkl} with the role $\bar \eta^0 := \bar \eta^*$ i.e.
		 	\begin{enumerate}[label = $(\varp)^{**}$, ref = $(\varp)^{**}$]
		 	\item for each $\delta \in S$ for some sequence $\langle y_{\delta+\ell}: \ \ell < \omega \rangle \in \ ^\omega G_{\delta+\omega}$ which is 
		 	a maximal  independent sequence over $G_\delta$,  for each $n \in \omega$ 
		 	\[ \left\{ \begin{array}{ll} \alpha<\delta: & \clpr\left( G_{\alpha+1} \cup
		 		\{y_{\delta+\ell}:\ell<n\} \right) \ne \\ &
		 		\clpr\left( G_\alpha \cup
		 		\{ y_{\delta+\ell}: \ \ell < n \} \right) +G_{\alpha+1} \end{array} \right\} \subseteq^* \ran(\eta^0_\delta),\]
		 \end{enumerate}
		 is also satisfied,
		  we also know that for each $\delta \in S$, for the output $\eta_\delta^1$ of the claim: ran($\eta_\delta^1$) $\subseteq^*$ ran($\eta_\delta^*$).
		
		At this point we claim that every ladder system $\bar \eta^2$ that is very similar to $\bar \eta^1$ has $\aleph_0$-uniformization, using that every ladder system very similar to $\bar \eta$ has $\aleph_0$-uniformization:
		we define $\eta^3_\delta$ so that $\eta^3_\delta(n) \in \Omega$ for each $n$ and $\bar \eta^3$ is a ladder system very similar to $\bar \eta^2$.
		Then, since for $\delta \in S$ we have ran($\eta_\delta) \subseteq \Omega$, and 
		$$\{\eta_\delta(n)+\omega: \ n \in \omega\} = \{\eta^*_\delta(n) + \omega: \ n \in \omega\},$$
		we obtain $$\{\eta_\delta(n)+\omega: \ n \in \omega\} \supseteq^* \{\eta^1_\delta(n) + \omega: \ n \in \omega\} = \{\eta^3_\delta(n) + \omega: \ n \in \omega\}. $$
		Since both $\eta_\delta(n)$'s and $\eta^3_\delta(m)$'s are limit ordinals, clearly
		 $$\{\eta_\delta(n): \ n \in \omega\} \supseteq^* \{\eta^3_\delta(n): \ n \in \omega\}. $$
		 Now  the fact that $\bar \eta$ has $\aleph_0$-uniformization implies that  $\bar \eta^3$ has $\aleph_0$-uniformization, and by Lemma \ref{verisi}, $\bar \eta^2$ has $\aleph_0$-uniformization, too. 
		 This means that we can appeal to Lemma \ref{befej} (applying it to $\langle x_\alpha: \ \alpha < \omega_1 \rangle$, and the ladder system $\bar \eta^1$, and the $\gamma^n_\delta$'s), which completes the proof of clause \ref{eki} of Proposition \ref{ekp}.
		 
	\end{PROOF}
		Having proved clause \ref{eki} of Proposition \ref{ekp}, with all the necessary ingredients we can verify \ref{AS} $ \Rightarrow $ \ref{CS} from Theorem \ref{2.1}. So suppose that $S \subseteq \omega_1$ is a stationary set of limit ordinals, and every $S$-ladder system has $\aleph_0$-uniformization. Fix a strongly $\aleph_1$-free abelian group $G$ and filtration $\bar G'$ which satisfy $(*)_S(G, \bar G')$ from Definition \ref{ABCdf}. Now applying Claim \ref{fil1} we can replace $S$ with $S \setminus \{\alpha+ \omega: \ \alpha < \omega_1\}$ (with the fact that every $S$-ladder system has $\aleph_0$-uniformization remaining true), get the filtration $\bar G = \langle G_\alpha : \ \alpha < \omega_1 \rangle$ and $(**)_S(G, \bar G)$ (from Claim \ref{fil1}). Moreover, we can apply Claim \ref{hulyetekkl}, which gives us $\bar x = \langle x_\alpha: \ \alpha < \omega_1 \rangle$ and the ladder system $\bar \eta$ on our new $S$, and that $\bar G$, $\bar x$ and $\bar \eta$ satisfy \ref{e0}-\ref{e3}. It remains to invoke \ref{eki} from Proposition \ref{ekp}: $(A)_{S,\bar \eta}$ holds (with $\bar \eta$ given by the claim), so we obtain $(C)_{S,\bar \eta}$. Thus it suffices to check that  $(*)^+_{S, \bar \eta^*}(G, \bar G')$ holds for some filtration $\bar G'$: set $\bar G' = \bar G$, $\bar \eta^* = \bar \eta$, $y_{\delta+n} = x_{\delta+n}$ ($\delta \in S$, $n \in \omega$) (recall \ref{e0}-\ref{e3}). This finishes the proof of \ref{AS} $ \Rightarrow $ \ref{CS} (assuming Theorem \ref{1.4}).
	\vspace{0.5cm}
	
	It is only left to prove clause \ref{ekii} from Proposition $\ref{ekp}$, so we fix the ladder system $\bar \eta^*$ on $S$.
	\newcounter{befcou} \setcounter{befcou}{0}
		\begin{enumerate}[label = $\bigodot_{\arabic*}$, ref = $\bigodot_{\arabic*}$]
		\setcounter{enumi}{\value{befcou}}
		\item We define the free abelian group $F$ to be generated by 
		$$\{ \dot{x}_\alpha: \ \alpha < \omega_1 \} \cup \{ \dot{y}_{\delta,i,j}: \ \delta \in S, \ i \leq j \in \omega \}.$$
		\stepcounter{befcou}
	\end{enumerate} 
	First we will need some definitions and claims
	\begin{definition}{\ }  \label{GpKp}
		\begin{enumerate}
			\item For a sequence $\bar p = \langle p_{\delta,i,j}: \ \delta \in S, \ i,j \in \omega\rangle \in \ ^{S \times \Delta^+} \bbZ$  (where by $\Delta^+$ we mean $\langle i,j: \ i \leq j \in \omega \rangle$) we define the abelian group $G_{\bar p}$
			by the generators
			$$\{ x_\alpha: \ \alpha < \omega_1 \} \cup \{ y_{\delta,i,j}: \ \delta \in S, \ i \leq j \in \omega \}$$
			and by the relations
			\begin{equation} \label{rel0} p_{\delta,i,j} y_{\delta,i,j} = y_{\delta,0,0} - (x_{\eta^*_\delta(j)}-x_{\eta^*_\delta(i)}) \ \ (\delta \in S, \ i<j), \end{equation}
			and
			 	\begin{equation} \label{rel3} p_{\delta,i,i} y_{\delta,i+1,i+1} = y_{\delta,0,0} - x_{\eta^*_\delta(i)} \ \ (\delta \in S, \ i \in \omega), \end{equation}
			i.e. a group isomorphic to $F / K_{\bar p}$, where  $K_{\bar p}$ is the subgroup of $F$ generated by $$\{  p_{\delta,i,j} \dot{y}_{\delta,i,j} - (\dot{y}_{\delta,0,0} - (\dot{x}_{\eta^*_\delta(j)}-\dot{x}_{\eta^*_\delta(i)})): \ \delta \in S, \ i< j \in \omega \},$$
			$$\{  p_{\delta,i,i} \dot{y}_{\delta,i+1,i+1} - (\dot{y}_{\delta,0,0} - \dot{x}_{\eta^*_\delta(i)}): \ \delta \in S, \ i \in \omega \},$$
			
			\item Define $K_{S,\bar \eta^*} = \{G_{\bar p}: \ \bar p  \in ^{S \times \Delta^+} \bbZ,$  $(\forall \delta \in S) \  \langle p_{\delta,i,j}: \ i,j \in \omega \rangle$ is a sequence of pairwise relatively prime positive integers $\}$.
			\item With a slight abuse of notation we define the group $G_{\bar p}$ when $\bar p = \langle p_{\delta,i}: \ \delta \in S, \ i \in \omega\rangle \in \ ^{S \times \Delta^+} \bbZ$ to be 
			$G_{\bar p'}$, where $\bar p' = \langle p'_{\delta,i,j}: \ \delta \in S, \ i,j \in \omega \rangle$ is defined by the equalities
			$$p'_{\delta,i,i} = p_{\delta,i} \ (i \in \omega),$$
			$$p'_{\delta,i,j} = 1 \ (i < j).$$
		%	\item with some abuse of notation, 
		\end{enumerate}

	\end{definition}

	The following is standard, but for the sake of completeness we include the proof.
	\begin{lemma} \label{kotszam}
		Suppose that we are given $\bar p = \langle p_{\delta,i,j}: \ \delta \in S, \ i \in \omega \rangle \in \ ^{S \times \omega}\bbZ$. If for each $\delta$ the sequence $\langle p_{\delta,i,j}: \ i \leq j \in \omega \rangle$ consists of pairwise relatively prime positive integers (i.e. there are no common prime factors of $p_{\delta,i,j}$ and $p_{\delta,k,l}$ if $(i,j) \neq (k,l)$, possibly $p_{\delta,i,j} =1$ for some $(i,j)$), then for the filtration 
		$\bar G^{\bar p} = \langle G^{\bar p}_\beta: \ \beta \in \omega_1 \rangle$ defined as
		$$ G^{\bar p}_\beta = \clpr(\{x_\alpha, y_{\alpha,0,0} : \ \alpha < \beta \})$$
		the assertion $(*)^+_{S, \bar \eta^*}(G_{\bar p}, \bar G^{\bar p})$ (from Proposition  $\ref{ekp}$) holds,	in particular  the groups in $K_{S,\bar \eta^*}$ are strongly $\aleph_0$-free.
	\end{lemma}  
	\begin{PROOF}{Lemma \ref{kotszam}}
		First we prove that the collection 
		$$I := \{x_\alpha: \ \alpha < \omega_1 \} \cup \{ y_{\delta,0,0}: \ \delta \in S\}$$ forms a maximal independent family. If this were true, then it would imply that $\{x_\alpha + G^{\bar p}_\beta: \ \beta \leq \alpha < \omega_1 \} \cup \{ y_{\delta,0,0} + G^{\bar p}_\beta: \ \beta \leq \delta \in S\}$ is independent in $G_{\bar p} / G^{\bar p}_\beta$.
		
		As $G_{\bar p} = \clpr( \{x_\alpha: \ \alpha < \omega_1 \} \cup \{ y_{\delta,0,0}: \ \delta \in S\})$ we only have to argue that $I$ is independent.
		Working with the representation of $G_{\bar p}$ as $F / K_{\bar p}$, it is enough to show that the set 
		$$ \dot{I} : = \{\dot{x}_\alpha: \ \alpha < \omega_1 \} \cup \{ \dot{y}_{\delta,0}: \ \delta \in S\} \subseteq F$$
		is independent over $K_{\bar p}$. But the generator set 
		$$\{ p_{\delta,i,j} \dot{y}_{\delta,i,j} - (\dot{y}_{\delta,0,0} - (\dot{x}_{\eta^*_\delta(j)}- \dot{x}_{\eta^*_\delta(i)}): \ \delta \in S, \ i< j \in \omega \}  $$
		$$ \cup $$
		$$\{ p_{\delta,i,i} \dot{y}_{\delta,i+1,i+1} - (\dot{y}_{\delta,0} - \dot{x}_{\eta^*_\delta(i)}): \ \delta \in S, \ i \in \omega \}$$
		
		freely generates $K_{\bar p}$, so comparing the coefficients clearly
		$$ \cll(\dot{I}) \cap K_{\bar p} = \{0\},$$
		so 
			\newcounter{befclcou} \setcounter{befclcou}{0}
				\begin{enumerate}[label = $\circleddash_{\arabic*}$, ref = $\circleddash_{\arabic*}$]
				\setcounter{enumi}{\value{befclcou}}
				\item \label{Iind} the system $$I = \{x_\alpha: \ \alpha < \omega_1 \} \cup \{ y_{\delta,0,0}: \ \delta \in S\} \text{ is independent in } G_{\bar p},$$ 
				indeed. In particular,
				\item $G_{\bar p}$ is torsion-free.
				\stepcounter{befclcou} 	\stepcounter{befclcou}
				
			\end{enumerate} 
		Second, we can argue that for any fixed $\beta \notin S$,
		the quotient $G^{\bar p} / G^{\bar p}_\beta$ is $\aleph_1$-free. By \cite[Thm. 19.1.]{Fu} it suffices to prove that every subgroup of finite rank is free. Since a subgroup of a free group is always free it is enough to show that a cofinal system of the finite rank subgroups (w.r.t. $\leq$) consists of free group.
		So fix $J = \{x_{\alpha_0}, x_{\alpha_1}, \dots, x_{\alpha_{j-1}}\}$, $L = \{y_{\delta_0,0,0}, y_{\delta_1,0,0}, \dots, y_{\delta_{\ell-1},0,0} \}$ for some $j,\ell<\omega$, where $\beta\leq \alpha_0 < \alpha_1 < \dots < \alpha_{j-1} <\gamma$ and $\beta \leq \delta_0 < \delta_1 < \dots < \delta_{\ell-1} < \gamma$.
		We would like to prove that $\clpr((L  \cup J) + G^{\bar p}_\beta) \leq G_{\bar p}/G^{\bar p}_\beta$ is	free, and by the above remark we can increase $J$ and $L$, and assume that
			\begin{enumerate}[label = $\circleddash_{\arabic*}$, ref = $\circleddash_{\arabic*}$]
			\setcounter{enumi}{\value{befclcou}}
			\item \label{asabove} if $y_{\delta_{k_0},0,0} \neq y_{\delta_{k_1},0,0} \in L$, $\eta^*_{\delta_{k_0}}(n_0) = \eta^*_{\delta_{k_1}}(n_1)$ for some $0 < n_0,n_1$, then $x_{\eta^*_{\delta_{k_0}}(n_0)} = x_{\eta^*_{\delta_{k_1}}(n_1)} \in J$.
			\stepcounter{befclcou}
		\end{enumerate} 
		Now we claim the following. By the characterization theorem  of finitely generated abelian groups \cite[Thm 15.5.]{Fu} the claim finishes the proof of Lemma $\ref{kotszam}$.
		\begin{claim}\label{oegy} 
			If $L$ and $J$ are as above, then $\clpr((L \cup J)+ G^{\bar p}_\beta) \leq G_{\bar p} / G_\beta^{\bar p}$ is generated by 
			$$ \{ z + G_\beta^{\bar p}: \ z \in L \cup K \} \cup$$
			$$ \cup \left\{ \frac{y_{\delta,0,0} - x_{\eta^*_{\delta}(n)})}{p_{\delta,n,n}} + G_\beta^{\bar p}: \ (y_{\delta,0,0} \in L,   n \in \omega) \wedge (x_{\eta^*_{\delta}(n)} \in J \ \vee \ \eta^*_{\delta}(n)< \beta) \right\},$$
				$$ \begin{array}{rl} \cup \left\{ \right. \frac{y_{\delta,0,0} - (x_{\eta^*_{\delta}(m)}- x_{\eta^*_{\delta}(n)}) }{p_{\delta,n,m}} + G_\beta^{\bar p}: & \ (y_{\delta,0,0} \in L,  \  n<m \in \omega) \wedge \\ & (x_{\eta^*_{\delta}(n)} \in J \ \vee \ \eta^*_{\delta}(n)< \beta) \wedge (x_{\eta^*_{\delta}(m)} \in J \ \vee \ \eta^*_{\delta}(m)< \beta) \left. \right\} \end{array}$$

			in particular it is finitely generated.
		\end{claim}
		\begin{PROOF}{Claim \ref{oegy}}(Claim \ref{oegy})
			Recalling the definition of the group $G_{\bar p}$ we have that the set $\sG_0 \cup \sG_1 \cup \sG_2 \cup \sG_3$ generates $G_{\bar p} / G_\beta^{\bar p}$, where
			$$ \sG_0 = \{ x_\alpha + G_\beta^{\bar p}: \ \beta \leq \alpha < \omega_1\}, $$
			$$ \sG_1 = \{ y_{\delta,0,0} +  G_\beta^{\bar p}: \ \delta \in S \setminus \beta\},$$ 
			 $$ \sG_2 \{ (y_{\delta,0,0}-x_{\eta^*_\delta(n)}) /p_{\delta,n,n} +  G_\beta^{\bar p}: \ \delta \in S \setminus \beta, \ n \in \omega \}, $$
			$$ \sG_3 \{ (y_{\delta,0,0}-(x_{\eta^*_\delta(m)} - x_{\eta^*_\delta(n)}) /p_{\delta,n,m} +  G_\beta^{\bar p}: \ \delta \in S \setminus \beta, \ n<m \in \omega \}, $$
			 so it is enough to verify the following subclaim.
			 \begin{sclaim}\label{scl} 
			 	Let $h \in G_{\bar p} / B^{\bar p}_\beta$ be given, suppose that
			 	$$ \begin{array}{rl}  h= &\sum_{\alpha < \omega_1} k_\alpha (x_\alpha + G_\beta^{\bar p}) + \\ 
			 		 & + \sum_{\delta \in S \setminus \beta} k_{\delta,0,0} (y_{\delta,0,0} +  G_\beta^{\bar p})+ \\
			 		 & +\sum_{\delta \in S, n \in \omega} k_{\delta,n+1,n+1} ((y_{\delta,0,0}-x_{\eta^*_\delta(n)}) /p_{\delta,n,n} +  G_\beta^{\bar p}) + \\
			 		 & + \sum_{\delta \in S, n<m \in \omega} k_{\delta,n,m} ((y_{\delta,0,0}-(x_{\eta^*_\delta(m)} - x_{\eta^*_\delta(n)}) /p_{\delta,n,m} +  G_\beta^{\bar p}).  \end{array}			$$
			 		
			Then the following hold:
			\begin{enumerate}[label = $\roman*)$, ref = $\roman*)$]
				\item \label{i} for every $\delta \in S \setminus \beta$, $n<m \in \omega$, if $y_{\delta,0,0} \notin L$, then 
				$$p_{\delta,n,m} \nmid k_{\delta,n,m}  \ \to \ h \notin \clpr(L \cup J),$$
				and
					$$p_{\delta,n,n} \nmid k_{\delta,n+1,n+1}  \ \to \ h \notin \clpr(L \cup J),$$
				\item \label{iii} for every $\delta \in S \setminus \beta$, $n<m \in \omega$, if $(y_{\delta,0,0} \in L$), and either ($x_{\eta^*_\delta(n)} \notin J$ $\&$ $\eta_\delta^*(n) \geq \beta$), or ($x_{\eta^*_\delta(m)} \notin J$ $\&$ $\eta_\delta^*(m) \geq \beta$), then 
				$$p_{\delta,n,m} \nmid k_{\delta,n,m} \ \to \ h \notin \clpr(L \cup J),$$
			%	\item \label{iv} for every $\alpha \geq \beta$ if $x_\alpha \notin J$, then 
			%		$$k_{\alpha}\neq 0 \ \to \ h \notin \clpr(L \cup J).$$
		%		\item \label{iv+} for every $\delta \in S \setminus \beta$, if $y_{\delta,0,0} \notin L$, then 
		%		$$m'_{\delta,0,0} \neq 0 \ \to \ h \notin \clpr(L \cup J).$$
			\end{enumerate}
			 \end{sclaim}
			(This would finish the proof of Claim \ref{oegy}, i.e. $h \in \clpr(L \cup J)$ implies that $h$ is a $\mathbb{Z}$-linear combinations of the finite set in the claim and possibly adding ($x_\alpha + G_\beta^{\bar p}$)'s ($\alpha \geq \beta$, $x_\alpha \notin J$), or  $(y_{\delta,0,0} +  G_\beta^{\bar p}$)'s ($\delta \in S \setminus \beta$, $y_{\delta,0,0} \notin L$). But since $h \in \clpr(L \cup J)$, we have that for some $0 \neq \ell \in \mathbb{Z}$, $\ell h$ is a $\mathbb{Z}$-linear combination of elements from $L \cup K$.)
			\begin{PROOF}{Subclaim \ref{scl}}(Subclaim \ref{scl})
				Fix $\delta \geq \beta$. We prove first \ref{i}, so we fix $\delta \in S$, and we need to show that
			  if $y_{\delta,0,0} \notin L$, $n<m \in \omega$, then 
			$$p_{\delta,n,m} \nmid k_{\delta,n,m} \ \to \ h \notin \clpr(L \cup J),$$
			and
			$$ p_{\delta,n,n} \nmid  k_{\delta,n+1,n+1}\ \to \ h \notin \clpr(L \cup J).$$
			
		 Suppose on the contrary, and consider a counterexample $h \in \clpr(L \cup J)$, and the values 
		$$ k_{\delta,n,m} \ \ (n,m \in \omega), $$
			which are not all $0$'s.
			Clearly (recalling \ref{Iind}):
			$$k_{\delta,0,0} + \sum_{n<m} \frac{k_{\delta, n,m}}{p_{\delta,n,m}} +  \sum_{n =0}^{\infty} \frac{k_{\delta,n+1,n+1}}{p_{\delta,n,n}}  = 0.$$
			Now since the sequence $\langle p_{\delta,n,m}: \ n,m \in \omega \rangle$ consists of pairwise relatively prime integers, even the fact that this sum is an integer implies 
			$$ \forall n<m \in \omega: \  p_{\delta,n,m} | k_{\delta,n,m},$$
			and
			$$ \forall n \in \omega: \  p_{\delta,n,n} | k_{\delta,n+1,n+1}.$$
		
					Now we can argue \ref{iii}.	Now for this fixed $\delta$ we can replace the relevant elements of $\sG_2 \cup \sG_3$ (i.e. those that have $y_{\delta,0,0}$ occurring in them) by the sum of elements from $\sG_0$ and $\sG_1$, we arrive at another $\bbZ$-linear combination with the same sum $h$, which derived decomposition satisfies 
					$$ [(\forall n<m) \ k_{\delta,n,m}=0 \ \& \ k_{\delta,n+1,n+1}=0]$$ 
						(only $k_{\delta,0,0}$ may be nonzero).
				%	But now as $h \in \clpr(L \cup J)$, so it corresponds to a $\mathbb{Q}$-linear combination of elements of $L$ and $J$, necessarily $k_{\delta,0,0} = 0$, which completes the proof of \ref{i}.

					By the above argument we can of course assume that
					\begin{equation} \label{assa} (\forall \delta \in S, \forall n \in \omega): \  [y_{\delta,0,0} \notin L \to \ k_{\delta,n+1,n+1} =0]
					\end{equation}
					and
					\begin{equation} %\label{assb}
						 (\forall \delta \in S, \forall n<m \in \omega): \  [y_{\delta,0,0} \notin L \to k_{\delta,n,m} = 0].
					\end{equation}
									
					So let $\alpha \geq \beta$ such that 
					\begin{equation} \label{notinj} x_\alpha \notin J. \end{equation}
					 Again, $h \in \clpr(L \cup J)$ implies
					$$\begin{array}{rl}
						0 = &k_{\alpha} -\\ 
						& -  \sum_{\delta\in S, y_{\delta,0,0} \in L, \ n \in \omega, \ \eta^*_\delta(n) = \alpha} \frac{k_{\delta,n+1,n+1}}{p_{\delta,n,n}} \\
						& - \sum_{\delta\in S, y_{\delta,0,0} \in L, \ n<m \in \omega, \ \eta^*_\delta(n) = \alpha} \frac{k_{\delta,n,m}}{p_{\delta,n,m}} \\
							& + \sum_{\delta\in S, y_{\delta,0,0} \in L, \ l<n \in \omega, \ \eta^*_\delta(n) = \alpha} \frac{k_{\delta,l,n}}{p_{\delta,l,n}}
				 \end{array}.$$
				 But by \ref{asabove} it follows from $x_\alpha \notin J$ that there is at most one $\delta_0 \in S$ for which $y_{\delta,0,0} \in L$, and $\alpha \in \eta^*_\delta(n)$ for some $n$. W.l.o.g. we can assume that there is such a $\delta_0$, (and $\eta^*_{\delta_0}(n_0) = \alpha$) otherwise we are done. So
						$$\begin{array}{rl}
						0 = &k_{\alpha} -\\ 
						& -   \frac{k_{\delta_0,n_0+1,n_0+1}}{p_{\delta_0,n_0,n_0}} \\
						& - \sum_{m>n_0} \frac{k_{\delta_0,n_0,m}}{p_{\delta_0,n_0,m}} \\
						& + \sum_{l<n_0} \frac{k_{\delta:0,l,n_0}}{p_{\delta_0,l,n_0}}
					\end{array}.$$
				Using that $\langle p_{\delta_0,n,m}: \ n \leq m \rangle$ consists of pairwise relatively prime numbers, we are done.
					
		%			Now for any $q$ for which the sum $\sum_{\delta\in S, \ n \in \omega, \ \eta^*_\delta(n) = \alpha, \  p_{\delta,n} = q} m'_{\delta,n+1}$ has non-zero components, by $\eqref{assa}$ no such components can be divisible by $q$, so there are at least two different pairs (of the form $(\delta,n)$) satisfying
		%			$\delta\in S, \ n \in \omega, \ \eta^*_\delta(n) = \alpha, \  p_{\delta,n} = q$ and $q \nmid m'_{\delta,n+1}$. Now by our assumptions on $\bar p$, and as $\eta^*_\delta$ is strictly increasing there are $\delta_0 \neq \delta_1$,  with $\eta^*_{\delta_0}(n_0) =\eta^*_{\delta_1}(n_1)= \alpha$, and $m'_{\delta_0,n_0+1} \neq 0$, $m'_{\delta_1,n_1+1} \neq 0$. But we assume \ref{i}, so $y_{\delta_0}, y_{\delta_1} \in L$, and by our assumptions $L$ and $J$ together satisfy \ref{asabove}, hence $x_\alpha \in J$, contradicting $\eqref{notinj}$.
						\end{PROOF}
		\end{PROOF}
	\end{PROOF}
	In order to complete the proof of clause \ref{ekii} we have to verify that if for each $\bar p$ satisfying the demands from Lemma $\ref{kotszam}$ the group $G_{\bar p}$ is Whitehead, then $\bar \eta^*$ has $\aleph_0$-uniformization.
	\begin{claim}\label{megolc}
		Suppose that the sequence $\bar p \in \ ^{S \times \Delta^+} \bbZ$ satisfies that for each $\delta \in S$ $\langle p_{\delta,n,m}: \ n\leq m \in \omega \rangle$ is a sequence of pairwise relative prime integers, and the group $G_{\bar p} \in K_{S,\bar \eta^*}$ (from Definition $\ref{GpKp}$) is Whitehead.
		
		Then for every sequence $\bar a = \langle a_{\delta,n,m} : \ \delta \in S, \ n\leq m \in \omega \rangle \in \ ^{S \times \omega} \bbZ$ there is a solution $\langle b_\delta: \ \delta \in S \rangle \in \ ^S \bbZ$, $\langle c_\alpha: \ \alpha < \omega_1 \rangle \in \ ^{\omega_1}\bbZ$ of $\bar a$,
		which means that 
		\begin{equation} \label{egyenlet} \forall \delta \in S, \ n<m \in \omega: \ b_\delta +a_{\delta,n,m} \equiv c_{\eta^*_\delta(m)}- c_{\eta^*_\delta(n)}  \ (\textrm{mod } p_{\delta,n,m}), \end{equation} 
		and
		\begin{equation} \label{2enlet} \forall \delta \in S, \ n \in \omega: \ b_\delta + a_{\delta,n,n} \equiv c_{\eta^*_\delta(n)}  \ (\textrm{mod } p_{\delta,n,n}). \end{equation} 
	\end{claim}
	\begin{PROOF}{Claim \ref{megolc}}
		Define the abelian group $H$ generated by 
		$$\{ x^*_{\alpha}, y^*_{\delta,n,m}: \ \alpha < \omega_1, \delta \in S, n\leq m \in \omega\} \cup \{ \zeta\}$$
		freely, except for the relations
		\begin{equation} \label{rel}  p_{\delta,n,m} y^*_{\delta,n,m} = y^*_{\delta,0,0} - (x^*_{\eta^*_\delta(m)}  -x^*_{\eta^*_\delta(n)}) - a_{\delta,n,m} \zeta \ \ (\delta \in S, \ n<m \in \omega), \end{equation}
		and
		\begin{equation} \label{rel2}  p_{\delta,n,n} y^*_{\delta,n+1,n+1} = y^*_{\delta,0,0} - (x^*_{\eta^*_\delta(m)}  -x^*_{\eta^*_\delta(n)}) - a_{\delta,n,n} \zeta \ \ (\delta \in S, \ n \in \omega). \end{equation}
		Now similarly to calculations in the proof of Lemma $\ref{kotszam}$ one can check that
			\newcounter{mecou} \setcounter{mecou}{0}
		\begin{enumerate}[label = $\circleddash_{\arabic*}$, ref = $\circleddash_{\arabic*}$]
			\setcounter{enumi}{\value{mecou}}
			\item \label{Izind} the system $$I^* = \{x^*_\alpha: \ \alpha < \omega_1 \} \cup \{ y^*_{\delta,0,0}: \ \delta \in S\} \cup \{ \zeta\}$$
		is a maximal independent family  in $H$, and
			\item \label{a} $\clpr_H(\{\zeta\}) = \bbZ \zeta$ (using the condition on $\bar p$, by the same argument as in Subclaim $\ref{oegy}$).
			\stepcounter{mecou} 
				\stepcounter{mecou} 	
		\end{enumerate} 
		It is also straightforward to check that there is a homomorphism $\varphi: H \to G_{\bar p}$ sending $x_\alpha$ to $x^*_\alpha$, $y^*_{\delta,0,0}$ to $y_{\delta,0,0}$, and $\zeta$ to $0$, since it respects the relations defining the groups in $\eqref{rel0}$ and $\eqref{rel3}$.
		Now since each element in $H$ is a $\bbQ$-linear combination of $I^*$, \ref{a} implies that 
		\begin{enumerate}[label = $\circleddash_{\arabic*}$, ref = $\circleddash_{\arabic*}$]
			\setcounter{enumi}{\value{mecou}}
			\item \label{kerf} $\textrm{Ker}(\varphi) = \bbZ \zeta \simeq \bbZ$,
			\stepcounter{mecou} 	
		\end{enumerate} 
		so there is a homomorphism $f: G_{\bar p} \to H$ with $\varphi \circ f =  \textrm{id}_{G_{\bar p}}$.
		Now clearly
			\begin{enumerate}[label = $\circleddash_{\arabic*}$, ref = $\circleddash_{\arabic*}$]
			\setcounter{enumi}{\value{mecou}}
			\item $ f(y_{\delta,n,m}) - y^*_{\delta,n,m} \in \bbZ \zeta$ ($\delta \in S$, $n\leq m$),
			\item $f(x_\alpha) - x^*_{\alpha} \in \bbZ \zeta$ ($\alpha < \omega_1$),
			\stepcounter{mecou} 	
			\stepcounter{mecou} 	
		\end{enumerate} 
		so we can define the vectors $\bar b$, $\bar c$ by the following equations:
			\begin{enumerate}[label = $\circleddash_{\arabic*}$, ref = $\circleddash_{\arabic*}$]
			\setcounter{enumi}{\value{mecou}}
			\item $b_\delta \zeta = f(y_{\delta,0,0}) - y^*_{\delta,0,0}$ ($\delta \in S$),
			\item $c_{\alpha} \zeta = f(x_\alpha) - x^*_{\alpha}$ ($\alpha < \omega_1$).
			\stepcounter{mecou} 	
			\stepcounter{mecou} 	
		\end{enumerate} 
		Then recalling $\eqref{rel}$, fixing $\delta \in S$, $n <m \in \omega$ and using $f(y_{\delta,n,m}) = y^*_{\delta,n,m} + e \zeta$ for some $e \in \bbZ$ (which holds by \ref{kerf}) on the one hand 
		$$y^*_{\delta,0,0} - (x^*_{\eta^*_\delta(m)}- x^*_{\eta^*_\delta(n)}) - a_{\delta,n,m} \zeta = p_{\delta,n,m}  y^*_{\delta,n,m},$$
		and on the other hand,
		$$ p_{\delta,n,m} ( y^*_{\delta,n,m} + e \zeta) = p_{\delta,n,m} f(y_{\delta,n,m}) = f(p_{\delta,n,m} y_{\delta,n,m}) = f ( y_{\delta,0,0}- (x_{\eta^*_\delta(m)} - x_{\eta^*_\delta(n)})) =$$
		$$ f ( y_{\delta,0,0})- (f(x_{\eta^*_\delta(m)} - f(x_{\eta^*_\delta(n)}) = 
		 b_\delta \zeta  + y^*_{\delta,0,0}  - (c_{\eta^*_\delta(m)} \zeta + x^*_{\eta^*_\delta(m)} - c_{\eta^*_\delta(n)} \zeta - x^*_{\eta^*_\delta(n)}),$$
   		so $p_{\delta,n,m} \cdot e \cdot \zeta = p_{\delta,n,m} ( y^*_{\delta,n,m} + e \zeta - y^*_{\delta,n,m} )$ can be written as:
   		 $$\left( b_\delta \zeta  + y^*_{\delta,0,0}  - (c_{\eta^*_\delta(m)} \zeta + x^*_{\eta^*_\delta(m)} - c_{\eta^*_\delta(n)} \zeta - x^*_{\eta^*_\delta(n)}) \right) - \left(y^*_{\delta,0,0} - (x^*_{\eta^*_\delta(m)}- x^*_{\eta^*_\delta(n)}) - a_{\delta,n,m} \zeta \right),$$
   		 hence
   		$$p_{\delta,n,m} \cdot e \cdot \zeta = b_\delta \zeta - (c_{\eta^*_\delta(m)} \zeta - c_{\eta^*_\delta(n)} \zeta) - (- a_{\delta,n,m} \zeta)$$
   		 so by \ref{Izind} the coefficients of $\zeta$ must be equal:
   		 $$p_{\delta,n,m} \cdot e = b_\delta - (c_{\eta^*_\delta(m)} - c_{\eta^*_\delta(n)}) +a_{\delta,n,m}.$$
   		 Finally, 
		 $$b_\delta + a_{\delta,n,m}   \equiv c_{\eta^*_\delta(m)} - c_{\eta^*_\delta(n)} \ (\textrm{mod }p_{\delta,n,m}),$$
		 as desired.
		 (Checking that $b_\delta + a_{\delta,n,n}   \equiv c_{\eta^*_\delta(n)} \ (\textrm{mod }p_{\delta,n,n}$) is completely analogous.)
	\end{PROOF}
		\begin{claim}\label{megold}
		Suppose that
		\begin{itemize}
			\item the groups in $K_{S,\bar \eta^*}$ (from Definition $\ref{GpKp}$ are all Whitehead,
		\end{itemize} moreover, assume that we are given the sequence $\langle d_{\delta,n}: \ \delta \in S, \ n \in \omega \rangle \in \ ^{S \times \omega}\omega$.
		Then there exists a sequence $\langle d^*_\alpha: \ \alpha < \omega_1 \rangle \in \ ^{\omega_1} \omega$ with the property
		$$ (\forall \delta \in S) \ \ \forall^\infty n:\  d^*_{\eta^*_\delta(n)} \geq d_{\delta,n}, $$
		and
			$$ (\forall \delta \in S) \ \ \forall^\infty n \ \forall k>0:\  d^*_{\eta^*_\delta(n+k)} \neq d^*_{\eta^*_\delta(n)}, $$
			i.e. for some $N$ the sequence $\langle d^*_{\eta^*_\delta(l)}: \ l \geq N \rangle$ is injective.
	\end{claim}
	\begin{PROOF}{Claim \ref{megold}}
		For each $\delta \in S$ define the sequence $p_{\delta,n,m}$ ($n \leq m \in \omega$) so that each $p_{\delta,n,m}$ is a prime, and if $(n,m) \neq (k,l)$, then $p_{\delta,n,m} \neq p_{\delta,k,l}$, moreover,
		
		$$ p_{\delta,n,n} \geq 4 d_{\delta,n},$$
		Note that this necessarily implies that
		$$(\forall M) (\exists N) ( \forall m \geq n \geq N): \  p_{\delta,n,m} \geq M.$$
		Now by our assumptions the group $G_{\bar p}$ is Whitehead, and defining $\langle a_{\delta,n,m}: \ \delta \in S, \  n \leq m \in \omega \rangle$ as
		\begin{equation} \label{ea} a_{\delta,n,m} = \lfloor p_{\delta,n,m}/2 \rfloor, \end{equation}
		then for some $\langle b_\delta: \ \delta \in S \rangle$, $\langle c_\alpha: \ \alpha < \omega_1 \rangle$ $\eqref{egyenlet}$ holds.
		This means that for each $\delta \in S$, for every large enough $n$ $|b_\delta| \leq p_{\delta,n,n}/4$, so for such $n$'s $\eqref{ea}$ and \eqref{2enlet} clearly imply
		$$ |c_{\eta^*_\delta(n)}| \geq p_{\delta,n,n} /4 \geq d_{\delta,n}.$$
		Moreover (for the same fixed $\delta$) if $N$ is so large that $N \leq n < m$ implies the inequality $|b_\delta| \leq p_{\delta,n,m}/4$, then one can easily get from \eqref{egyenlet}  that
		$$|c_{\eta^*_\delta(m)} - c_{\eta^*_\delta(n)}|\geq p_{\delta,n,n} /4 >0,$$
		so choosing $d^*_\alpha = |c_{\alpha}|$ works.
		
	\end{PROOF}

		\begin{claim}\label{megold2}
		Suppose that
		\begin{itemize}
			\item for every $G \in K_{S,\bar \eta^*}$ (from Definition $\ref{GpKp}$ $G$ is Whitehead.
		\end{itemize} 
		Then $\bar \eta^*$ has $\aleph_0$-uniformization.
	\end{claim}
	\begin{PROOF}{Claim \ref{megold2}}
		Fix a system  $\langle f_\delta: \ \delta \in S \rangle$, where for each $\delta \in S$ 
		$$f_\delta: \ran(\eta^*_\delta) \to \omega \text{ is a coloring with countably many colors,}$$
		and using Claim $\ref{megold}$ fix $\langle d^*_\alpha: \ \alpha < \omega_1 \rangle$ with
		\begin{equation}\label{vegt} \forall \delta \in S: \ \lim_{n \to \infty} d^*_{\eta^*_\delta(n)} = \infty, \end{equation}
		w.l.o.g. $\forall \alpha$ $d^*_\alpha \in 2\bbZ+1$.
		Now apply Claim $\ref{megold}$ again in order to obtain a system $\langle e_\alpha: \ \alpha < \omega_1 \rangle \in \ ^{\omega_1} \omega$ satisfying
		\begin{equation} \label{eq0} (\forall \delta \in S) \ \ \forall^\infty n:\  e_{\eta^*_\delta(n)} \geq d^*_{\eta^*_{\delta}(n)} \cdot (f_\delta(\eta^*_\delta(n))+1), \end{equation}
		and
		\begin{equation} \label{eq1} (\forall \delta \in S) \  (\exists N = N(\delta)) \ (\forall n,m \geq N):\  [(n \neq m) \ \to \ (e_{\eta^*_\delta(n)} \neq e_{\eta^*_\delta(m)})] \end{equation}
			\newcounter{meecou} \setcounter{meecou}{0}
		\begin{enumerate}[label = $\boxtimes_{\arabic*}$, ref = $\boxtimes_{\arabic*}$]
			\setcounter{enumi}{\value{meecou}}
			\item Choose $q_\alpha$ to be the $e_\alpha$'th prime (so in particular
			$$(\forall \delta \in S) \ \ \forall^\infty n:\  q_{\eta^*_\delta(n)} \geq d^*_{\eta^*_{\delta}(n)} \cdot (f_\delta(\eta^*_\delta(n))+1))),$$
			\item \label{qalpha}  let for each $\delta \in S$, $n \in \omega$ 
			$$ \begin{array}{rll} p_{\delta,n} & = q_{\eta^*_\delta(n)}, & \text{ if } n \geq N(\delta), \\
			 & =1 & \text{otherwise.} \end{array}$$
				\item \label{qa} and note that by \eqref{eq1} $\langle p_{\delta,n}: \ n \in \omega \rangle$ is a sequence of pairwise relatively prime integers, with
				$$ (\forall^\infty n)\  p_{\delta,n} = q_{\eta^*_\delta(n)}.$$
			\stepcounter{meecou}
			
			\stepcounter{meecou} 	
			\stepcounter{meecou} 
		\end{enumerate} 
		
		Now 	
		\begin{enumerate}[label = $\boxtimes_{\arabic*}$, ref = $\boxtimes_{\arabic*}$]
			\setcounter{enumi}{\value{meecou}}
			\item 	define the system $\langle a_{\delta,n} : \delta \in S, n \in \omega \rangle$ as
			$$ a_{\delta,n} = d^*_{\eta_{\delta}(n)} \cdot f_\delta(\eta^*_\delta(n)),$$
			\stepcounter{meecou} 
		\end{enumerate}
		and apply Claim $\ref{megold}$ again, i.e.\ we obtain the solution
		 $\langle b_\delta: \ \delta \in S \rangle \in \ ^S \bbZ$, $\langle c_\alpha: \ \alpha < \omega_1 \rangle \in \ ^{\omega_1}\bbZ$:
		\begin{equation} \label{egyenlet2} \forall \delta \in S, \ n \in \omega: \ b_\delta \equiv c_{\eta^*_\delta(n)} - a_{\delta,n} \ (\textrm{mod } p_{\delta,n}). \end{equation} 
		Now
		\begin{enumerate}[label = $\boxtimes_{\arabic*}$, ref = $\boxtimes_{\arabic*}$]
			\setcounter{enumi}{\value{meecou}}
			\item 	for each $\alpha < \omega_1$ we define $f^*(\alpha) \in \omega$ as follows. Let 
			$$c^*_\alpha \in \{kq_\alpha + \ell d^*_\alpha: \ k \in \bbZ, \ \ell \in \{0,1, \dots, \lfloor q_\alpha/d^*_\alpha \rfloor-1 \}\}$$ be the element of that set of minimal distance from $c_\alpha$ (if that set is not empty), let $f^*(\alpha)= \ell_0  \in \{0,1, \dots, \lfloor q_\alpha/d^*_\alpha \rfloor-1 \}$, where $c^*_\alpha = k_0 q_\alpha + \ell_0 d^*_\alpha$.		
			\stepcounter{meecou} 
		\end{enumerate}
		Fixing $\delta \in S$, using \ref{qalpha} (in order to prove that $f^*$ uniformizes the $f_\delta$'s) it clearly suffices to prove that
		\begin{equation} \label{sufeq} (\forall \delta \in S) \ \forall^\infty n \exists k\in \bbZ \  |k q_{\eta^*_\delta(n)} + f_\delta(\eta^*_\delta(n)) \cdot d^*_{\eta^*_\delta(n)} - c_{\eta^*_\delta(n)}| <  d^*_{\eta^*_\delta(n)} / 2 \end{equation}
		(recall that $0 \leq f_\delta(\eta^*_\delta(n)) \cdot d^*_{\eta^*_\delta(n)} < q_{\eta^*_\delta(n)}$ holds for each $\delta$ for large enough $n$).
		By $\eqref{egyenlet2}$,
		$$ b_\delta + a_{\delta,n} = b_\delta + d^*_{\eta_{\delta}(n)} \cdot f_\delta(\eta^*_\delta(n)) \equiv c_{\eta^*_\delta(n)} \ (\textrm{mod }p_{\delta,n}). $$
		First recall that $p_{\delta,n} = q_{\eta^*_\delta(n)}$ for large enough $n$ (by \ref{qa}), so 
	 for any large enough $n$, for some $k = k(n) \in \bbZ$
		$$b_\delta + d^*_{\eta_{\delta}(n)} \cdot f_\delta(\eta^*_\delta(n)) +  k q_{\eta^*_\delta(n)}= c_{\eta^*_\delta(n)}.$$
				Rearranging this equation one obtains
		$$d^*_{\eta_{\delta}(n)} \cdot f_\delta(\eta^*_\delta(n)) +  k q_{\eta^*_\delta(n)}- c_{\eta^*_\delta(n)} = - b_\delta. $$
				
		Therefore for large enough $n = n(\delta)$ we have $|b_\delta| < d^*_{\eta_\delta^*(n)} / 2$ (by $\eqref{vegt}$), obtaining $\eqref{sufeq}$, we are done.

	\end{PROOF}

\end{PROOF}

%\begin{remark}  We can replace $(B),(C)$ in Theorem \ref{2.1} by
%	$(B)^-,(C)^-$ which means that we add ``$G$ is strongly
%	$\aleph_1$-free".  For this in \ref{1.1}$(B)_S$ we can restrict
%	ourselves to simple ${\gp}$.
%\end{remark}
\newpage

\section {Special uniformization problems: a combinatorial equivalence} \label{s2}

The purpose of this section is to prove our main theorem:
	\begin{theorem}
	\label{1.1}
	Let $S$ be a stationary set of limit ordinals $<\omega_1$, which is simple (i.e. $\forall \alpha$: $\alpha+\omega \notin S$).
	\Then \, $(A)_S \Rightarrow (B)_S$, where:
	\mn
	\begin{enumerate}
		\item[$(A)_S$]   every $S$-ladder system has $\aleph_0$-uniformization,
		\sn
		\item[$(B)_S$]   every special $S$-uniformization ${\gp}$ has a
		solution, see below.
	\end{enumerate}
	\mn

\end{theorem}
	We shall prove more in \ref{1.4} below.

\begin{definition}
\label{1.3}
Assume ${\mathfrak p}$ is a special
$S$-uniformization problem, $\Psi = \Psi^{\gp}$. 

\noindent
1) We say $x_1$, $x_2$ are $\Psi$-isomorphic \when :
$|x_1|=|x_2|$ and for any $b_\ell\subseteq x_\ell$, $b_\ell \in \cB^\gp$, $f_\ell$ a function
with domain $b_\ell$ (for $\ell=1,2$) satisfying

\[
\text{OP}_{u_1, u_2}(b_2)=b_1,
\]

\[
f_1\circ \text{OP}_{b_1, b_2}=f_2
\]

\mn
we have $f_1\in \Psi(b_1) \text{ iff } f_2\in \Psi(b_2)$. 

\end{definition}

\noindent

\begin{observation}{\ } \\
\label{a14}
1) Let $\gp$ be a special $S$-uniformization problem.
The number of $\Psi$-isomorphism types is $\le \aleph_0$.

\noindent
2) For every finite $u \subseteq \omega_1$ there is a finite
   $\gp$-closed $v$ extending $u$.

%\noindent
%3) If $|u| = |v|$ and for every function with claim $f,f \in \Psi(u)
 %  \Leftrightarrow f \circ \text{ OP}_{u_2,u_1} \in \Psi/v$ \then \,
 %  $u,v$ are isomorphic.

\noindent
3) The relation ``$\bar\eta',\bar\eta$" are very similar $S$-ladder
   systems" is an equivalence relation on the set of $S$-ladder
   systems (see Definition \ref{a6b}(4)).
\end{observation}

\begin{PROOF}
This follows by clause (c) of Definition \ref{a4}, and see Claim $\ref{lezar}$ below.
\end{PROOF}

\noindent
As a warm-up we show (but we rely on Claim $\ref{lezar}$ below).
\begin{claim}
\label{1.3A}
If ${\gp}$ is a special $S$-uniformization problem, \then \, 
the forcing notion $\bbQ_{\gp} = ({\cF}_{\gp},\subseteq)$ satisfies 
the c.c.c. and $\Vdash_{{\bbQ}_{\gp}} ``\cup\{f:f \in \name
G_{\bbQ_{\gp}}\}$ is a solution for ${\gp}"$ where ${\cF}_{\gp} 
= \cup\{\Psi^{+}(u):u \subseteq \omega_1$ is finite, $\gp$-closed$\}$.
\end{claim}

\begin{PROOF}{Claim \ref{1.3A}}
For the second phrase just note that $\cF_{\gp}$ is non-empty (by
clause \ref{pdb} of Definition \ref{pprobd} and for every $\alpha <
\omega_1$ the set ${\cI}_\alpha = \{f \in {\cF}_p:\alpha \in
\text{ Dom}(f)\}$ is dense open in $\bbQ_{\gp}$ by clause \ref{pdd} in
Definition \ref{pprobd}.

As for the first phrase, ``$\bbQ_{\gp}$ satisfies the c.c.c.",
let $f_\alpha \in \Psi(x_\alpha)$ for $\alpha < \omega_1$, where each $x_\alpha$ is ${\gp}$-closed.  
As each $x_\alpha$ is finite, for
some stationary $S_0 \subseteq S$ the 
sequence $\langle x_\alpha:\alpha \in S_0 \rangle$ is a 
$\Delta$-system with heart $u^*$, and for $\alpha \ne
\beta$ in $S_0,x_\alpha \cap \alpha = x_\beta \cap \beta =
x_\alpha \cap x_\beta = u^*$.  Without loss of generality
$\alpha \in S_0 \Rightarrow f_\alpha \restriction u^* = f^*$ and $\beta
< \alpha \in S_0 \Rightarrow x_\beta \subseteq \alpha$. 
Observe that this together with the closedness of the $x_\beta$'s imply that 
\begin{equation} \label{ssza} \alpha <
\beta \in S_0 \ \Rightarrow \sup(x_\alpha)+\omega \leq \min(x_\beta \setminus u_*). \end{equation}
 Let
$\beta(\alpha) = \text{ min}(S_0 \backslash (\alpha +1))$ for $\alpha \in
S_0$. We define $v_\alpha = u^*$ if the ordinal $\delta: = $ min($x_{\beta(\alpha)} \setminus \beta(\alpha)) \notin S$, and  otherwise if this $\delta \in S$ let 
$$v_\alpha = \cll^\gp\left(\bigcup_{n <
	\omega,  \ \eta_\delta(n) < \alpha} (\{\eta_\delta(n)\} \cup u_{\delta,n}) \cup u^*\right),$$ so clearly $v_\alpha \subseteq \alpha$, and it
is finite by Claim $\ref{lezar}$ below. 

\begin{sclaim} \label{sc}
For each $\alpha \in S_0$ the finite set $y_\alpha = x_{\beta(\alpha)} \cup v_\alpha$ is $\mathfrak{p}$-closed.
\end{sclaim}  
\begin{PROOF}{Subclaim \ref{sc}}
	First we check clause  \ref{1.2)2a} of Definition \ref{1.2} \ref{1.2)2}. Pick $\epsilon+1 \in y_\alpha$. Since both $x_{\beta(\alpha)}$ and $v_\alpha$ satisfy \ref{1.2)2a}, the ordinal $\epsilon$ must be in either $x_{\beta(\alpha)}$, or $v_\alpha$.
	Now let $\delta \in S \cap y_\alpha$. If $\delta < \alpha$, then necessarily $\delta \in v_\alpha$ and since $v_\alpha$ is $\mathfrak{p}$-closed we are done. At this point we also note that
	\begin{equation}\label{eqq1} (\forall \delta \in S \cap y_\alpha \cap \alpha) \ (\forall n,k):  \ [b^\delta_{n,k} \subseteq y_\alpha  \ \rightarrow \ b^\delta_{n,k} \subseteq v_\alpha] \ \ \text{ (since  }v_\alpha \supseteq u^*). \end{equation} 
	
	Second, suppose that $\delta \geq \alpha$. Then $\delta \in x_{\beta(\alpha)} \setminus \alpha = x_{\beta(\alpha)} \setminus \beta(\alpha)$ (since $x_{\beta(\alpha)} \cap \beta(\alpha) = u^*$ as $\beta(\alpha) \in S_0$). We will distinguish two cases depending on whether $\delta = $ min($x_{\beta(\alpha)} \setminus \beta(\alpha))$, or $\delta > $ min($x_{\beta(\alpha)} \setminus \beta(\alpha))$.
	
	If $\delta > $  min($x_{\beta(\alpha)} \setminus \beta(\alpha))$, then for each $n$ we have 
	$$ \eta_\delta(n) \in x_{\beta(\alpha)} \ \iff \ \eta_\delta(n) \leq  \textrm{ max}(\delta \cap x_{\beta(\alpha)} \setminus \beta(\alpha)) \ \iff \{\eta_\delta(n) \} \cup u_{\delta,n} \subseteq x_{\beta(\alpha)},$$ (since $x_{\beta(\alpha)}$ is $\mathfrak{p}$-closed), so 
	\begin{equation} \label{eq2} \eta_\delta(n) \in y_\alpha \ \iff \ \{\eta_\delta(n) \} \cup u_{\delta,n} \subseteq y_{\alpha} \ \iff \ \eta_\delta(n) \leq \textrm{ max}(\delta \cap x_{\beta(\alpha)} \setminus \beta(\alpha)). \end{equation}
	Now if $\delta = $  min($x_{\beta(\alpha)} \setminus \beta(\alpha))$, then for each $n$
	$$ \eta_\delta(n) \in v_{\alpha} \ \iff \ \eta_\delta(n) < \alpha \ \iff \{\eta_\delta(n) \} \cup u_{\delta,n} \subseteq v_\alpha,$$
	so 
	\begin{equation} \label{eq3} \eta_\delta(n) \in y_{\alpha} \ \iff \ \eta_\delta(n) < \alpha \ \iff \{\eta_\delta(n) \} \cup u_{\delta,n} \subseteq y_\alpha, \end{equation} and we are done.
	 For future reference we note that \eqref{eqq1}, \eqref{eq2}, \eqref{eq3} together imply
	 \begin{equation} (\forall \delta \in S \cap y_\alpha \setminus \alpha) \ (\forall n,k):  \ [b^\delta_{n,k} \subseteq y_\alpha  \ \rightarrow \ (b^\delta_{n,k} \subseteq v_\alpha \ \vee \ b^\delta_{n,k} \subseteq x_{\beta(\alpha)})]. \end{equation} 
\end{PROOF}

Now for some $v^*  \in [\omega_1]^{<\aleph_0}$ and
stationary subset $S_1$ of $S_0$ we have 
 $$\alpha \in S_1
\Rightarrow (v_\alpha = v^*).$$  For each $\alpha \in S_1$ 
we can find $g_\alpha \in
\Psi^+(y_\alpha) = \Psi^+(x_{\beta(\alpha)})$ extending $f_{\beta(\alpha)}$, which
exists by clause \ref{pdd} in Definition \ref{pprobd}.  Now w.l.o.g. we can assume that 
for every $\alpha_1 < \alpha_2$ in $S_1$ we have $\beta(\alpha_1) < \alpha_2$ (intersecting $S_1$ with a club if necessary), and $g_{\alpha_1} \rest v^* = g_{\alpha_2} \rest v^*$ (by passing down to a smaller stationary set).  
Observe that if $\alpha_1 < \alpha_2$ in $S_1$, then  $y_{\alpha_1} \cup y_{\alpha_2} =  v^*\cup x_{\beta(\alpha_1)} \cup x_{\beta(\alpha_2)}$ is $\gp$-closed, since by the proof of Subclaim \ref{sc} for any $\delta \in (S \cap y_{\alpha_2} \setminus \alpha_2)$ either for all $n \in \omega$:
$$\eta_\delta(n) \in y_{\alpha_2} \ \iff \ \{\eta_\delta(n) \} \cup u_{\delta,n} \subseteq y_{\alpha_2} \ \iff \ \eta_\delta(n) \leq \textrm{ max}(\delta \cap x_{\beta(\alpha_2)} \setminus \beta(\alpha)),$$
(which is the case if $\delta > $ min($x_{\beta(\alpha_2)} \setminus \beta(\alpha)) = $ min$(y_{\alpha_2} \setminus \beta(\alpha))$), or
$$\eta_\delta(n) \in y_{\alpha_2} \ \iff \ \eta_\delta(n) < \alpha_2 \ \iff \{\eta_\delta(n) \} \cup u_{\delta,n} \subseteq y_{\alpha_2}$$
(when $\delta = $ min($x_{\beta(\alpha_2)} \setminus \beta(\alpha)) = $ min$(y_{\alpha_2} \setminus \beta(\alpha))$).
This means that
	\begin{equation} \label{eee1}  \textrm{for }\delta \in (S \cap y_{\alpha_2} \setminus \alpha_2): \  (\forall n,k):  \ [b^\delta_{n,k} \subseteq y_{\alpha_1} \cup y_{\alpha_2}]  \  \rightarrow \ [b^\delta_{n,k} \subseteq y_{\alpha_2}]. \end{equation} 
	Whereas if $\delta \in (S \cap (y_{\alpha_2} \cup y_{\delta_1}) \cap \alpha_2)$, then $\delta \in y_{\delta_1}$ (this holds for any $\epsilon \in   (y_{\alpha_2} \cup y_{\delta_1}) \cap \alpha_2$), so
	\begin{equation} \label{eee2} \textrm{if } \delta \in (S \cap (y_{\alpha_2} \cup y_{\delta_1}) \cap \alpha_2): \  (\forall n,k):  \ [b^\delta_{n,k} \subseteq y_{\alpha_1} \cup y_{\alpha_2}]  \  \rightarrow \ [b^\delta_{n,k} \subseteq y_{\alpha_1}]. \end{equation}
  Therefore, as  $f_{\beta(\alpha_1)} \cup f_{\beta(\alpha_2)}$ is a function (since they agree on $u^*$), so (by \eqref{eee1}, \eqref{eee2}) $f_{\beta(\alpha_1)} \cup f_{\beta(\alpha_2)} \in \Psi^+(y_{\alpha_1} \cup y_{\alpha_2})$, hence are
compatible in $\bbQ_{\gp}$ as required.  
\end{PROOF}

Our main theorem (Theorem $\ref{1.1}$) will follow from the following slightly more general one:
\begin{theorem}
\label{1.4}
Let $S$ be a stationary set of limit
ordinals $< \omega_1$ which is simple, i.e., $(\forall \alpha <
\omega_1)(\alpha + \omega \notin S)$, and 
$\bar \eta^*$ be an $S$-ladder system.  \Then \,
$(\mathbf{D})_{S,\bar \eta^*} \Leftrightarrow (\mathbf{E})_{S,\bar \eta^*}$ where
\mn
\begin{enumerate}
\item[$(\mathbf{D})_{S,\bar \eta^*}$]  for any ladder system $\bar \eta$ on $S$ if $\eta$ is
very similar to $\bar \eta^*$ \then \, it has $\aleph_0$-uniformization
\sn
\item[$(\mathbf{E})_{S,\bar \eta^*}$]   for every ladder system  $\bar \eta$ very
similar to $\bar \eta^*$, for every  simple special
$S$-uniformization problem  ${\gp}$ with $\bar \eta^{\gp} = \bar \eta$,
\then \, ${\gp}$ has a solution.
\end{enumerate}
\end{theorem}

\noindent
We shall prove Claim \ref{1.4} together with proving Theorem
\ref{1.1}.
\begin{PROOF}{Theorem \ref{1.1}}
\underline{Proof of \ref{1.4}}

\noindent
\underline{$(\mathbf{E})_{S,\bar \eta^*} \Rightarrow (\mathbf{D})_{S,\bar \eta^*}$}

$(\mathbf{D})_{S,\bar \eta^*}$ is essentially a special case of 
$(\mathbf{E})_{S,\bar \eta^*}$:  let $\bar\eta$ be very similar to $\bar\eta^*$
we can replace $\eta_\delta(n)$ by
$\eta_\delta(n)+1$ and use the value of the prospective solution $f$ at $\delta$ to say from which $n$
onwards the uniformization demand for $c_\delta$ holds, that is,
$c_\delta(\eta_\delta(n)) = f(\eta_\delta(n))$ if $n \geq f(\delta)$.  

In detail, this will look as follows. Given $\langle \eta_\delta:\delta\in S \rangle$,
$\langle c_\delta:\delta\in S \rangle$ as in $(\mathbf{D})_{S,\bar \eta^*}$
recalling Definition \ref{pprobd} we now define a 
special $S$-uniformization problem ${\mathfrak p}$ with 
\mn
\begin{enumerate}
\item[$(i)$]   $\bar \eta^{\gp} = \langle \eta'_\delta:\delta
\in S \rangle$, where $\eta'_\delta(n) = \eta_\delta(n)+1$ ($ 
n < \omega)$,
\sn
\item[$(ii)$]  $u^{\gp}_{\delta,n} = \emptyset$ 
for every $\delta \in S$, $n < \omega$ (so $b^\delta_{n,k} = \{ \eta'_\delta(i): \ i \leq n\} \cup [\delta, \delta+k]$),
\sn
\item[$(iii)$]   $\Psi^{\gp}(b^\delta_{n,k})$ is the set 
of all functions $f$ from $b^\delta_{n,k}$ to $\omega$ such that: 
\begin{enumerate}
\item[$\bullet$]   $(\forall i \in [f(\delta)), n]): \ \ f(\eta'_\delta (i))=c_\delta(\eta_\delta(i))$,
%\item[$\bullet$]  for every $\delta_1, \delta_2 \in \dom(f)$ if $i_1$, $i_2$ are such that $\eta_{\delta_1}(i_1) =
%\eta_{\delta_2}(i_2)$ and $i_1 \ge f(\delta_1)$ and $i_2 \ge f(\delta_2)$, then %$c_{\delta_1}(\eta_{\delta_1}(n_1)) =
%c_{\delta_2}(\eta_{\delta_2}(n_2))$,
\end{enumerate}
%\item[$(iv)$]   $\bar h = \langle h_\delta:\delta \in S \rangle$ where
%$h_\delta:\omega\to \omega$ is constantly zero.
\end{enumerate}
\mn
It has to be checked that
\mn
\begin{enumerate}
\item[$(\alpha)$]  $(\gp)$ is indeed a special $S$-uniformization
problem (recall that $\{\ran(\eta_{\delta}): \ \delta \in S \}$ is an almost disjoint family), and
\sn
\item[$(\beta)$]  $\bar\eta',\bar\eta$ are very similar.
\end{enumerate}
\mn
The second clause is immediate, let us deal with the first clause for which we need to check that $\Psi^\gp$ satisfies clauses \ref{pda}-\ref{pdb} from Definition \ref{pprobd}. For a fixed $\delta \in S$, $n,k \in \omega$ we note that  $\Psi^{\gp}(b^\delta_{n,k})$ is determined by $\langle c_\delta(\eta_\delta(j)): \ j \leq n \rangle \in \ ^{n+1}\omega$, so \ref{pda} holds. It is easy to see that clauses \ref{pdb} and \ref{pdc} hold.
Finally, fix a finite $u \subseteq \omega$ that is $\gp$-closed, $f \in \Psi^+(u)$, $v \supseteq u$, $|v| < \aleph_0$.
Now observe that if we could define $f(\alpha)$ for all the $\alpha$'s that belong to $v \setminus u$, and $\alpha$ is of the form $\alpha = \eta'_{\delta}(\ell)$ for some $\ell \in \omega$, $\delta \in u$ (and so $\alpha = \eta_\delta(\ell)+1$) so that $f(\alpha) = c_\delta(\eta_\delta(\ell))$, then we would be done.

We claim that whenever $\delta < \delta'$ both belong to $u \cap S$, and $\eta_\delta(k) = \eta_{\delta'}(l)$ for some $l,k \in \omega$, then $\eta'_{\delta}(k) = \eta'_{\delta'}(l) \in u$. This will finish the verification of clause $(\alpha)$.
But since $u$ is $\gp$-closed, in fact for each $l$ with $\eta'_{\delta'}(l) < \delta$ we have $\eta'_{\delta'}(l) \in u$, since $\delta,\delta' \in u$ (just recall \ref{1.2)2} \ref{1.2)2b} $(ii)$).

Now let $f'$ be a solution to the special uniformization problem $\mathfrak{p} = (S,\bar \eta',\bar u,\Psi)$, as guaranteed in $(B)_{S,\bar\eta^*}$ (i.e., as defined in Definition \ref{a6b}(1)).
We check the function $c:\omega_1 \rightarrow \omega$ defined by:
$c(\alpha)=: f(\alpha+1)$ is as required in $(\mathbf{D})_{S,\bar \eta^*}$ for
$\bar\eta$: 

Fix $\delta \in S$ and $n \in \omega$ for which $n \geq f(\delta)$.
Consider the set $b^\delta_{n,0}$, containing $\eta'_{\delta}(n) = \eta_\delta(n)+1$, and $\delta$.
Since $f$ is a solution, $f_0 = f \rest b^\delta_{n,0} \in \Psi^{\gp}(b^\delta_{n,k})$, thus (recalling the definition of $\Psi^{\gp}(b^\delta_{n,k})$ in clause $(iii)$):
 $$f(\eta'_\delta(n)) = f_0(\eta'_\delta(n))=  c_\delta(\eta_\delta(n)),$$ as $n \geq f_0(\delta) = f(\delta)$.
 Therefore, (for $n \geq f(\delta)$):
 $$ c(\eta_\delta(n)) = f(\eta_\delta(n)+1) = f(\eta'_\delta(n)) = c_\delta(\eta_\delta(n)),$$
 as desired.
\medskip

\noindent
\underline{$(\mathbf{D})_{S,\bar \eta^*} \Rightarrow (\mathbf{E})_{S,\bar \eta^*}$}.  

So from now on we are going to work under the assumption that for all $\bar \eta$ similar to $\bar \eta^*$ have $\aleph_0$-uniformization. 
After the following preparation Lemma $\ref{le1}$, Lemma $\ref{le2}$, Lemma $\ref{Lebe}$ together will conclude the statement.

Let ${\gp}$ be a special $S$-uniformization problem with
$\bar\eta^{\gp} = \bar\eta = \langle \eta_\delta:\delta\in S \rangle$, also
$\langle h_\delta:\delta\in S \rangle$,
$\langle u_{\delta,n}: \ \delta \in S, \ n \in \omega \rangle$ and $\Psi$ as in 
Definition \ref{pprobd} are given such that $\bar \eta$ is very similar to
$\bar \eta^*$ and $S = S^{\gp}$. 
So it is enough to prove that $\gp$ has a solution.

(We usually suppose ${\gp}$ is the index so closed means ${\gp}$-closed, etc.)
The proof is broken to stages (and definitions).
\medskip

\begin{claim}\label{lezar}
	 Let $\gp^0 = (S, \bar \eta, \bar u)$ be an  $S$-uniformization frame. For every 
finite $x\subseteq \omega_1$ there is a ${\gp}^0$-closed finite
$y$, $x\subseteq y\subseteq \omega_1$ satisfying $\max(x) = \max(y)$, moreover,
if we are given $\bar v = \langle v_\alpha: \ \alpha < \omega_1 \rangle$ with for each $\alpha$: $v_\alpha \in [\alpha]^{<\aleph_0}$, then we can prescribe $y$ to be $\bar v$-closed, by which we mean
\[ \alpha \in y \ \to \ v_\alpha \subseteq y.\]
\end{claim}

\begin{PROOF}{Claim \ref{lezar}}
%\underline{Proof of $(\alpha)$}:  
We prove this by induction on 
$\delta_x = \max \{\alpha\in \{0\}\cup\Omega:\alpha \le \sup(v)\}$
\smallskip

\noindent
\underline{Case 1, $\delta_x=0$}: trivial, let $u=[0,\sup(v)]$.
\smallskip

\noindent
\underline{Case 2, $\delta_x > 0$ and $\delta_x \notin S$}: 

Let $m$ be such that max$(x) = \delta + m$, note that sup$(x \cap
\delta) < \delta$ and so by the induction
hypothesis there is a ${\mathfrak p}$-closed finite $y' \subseteq
\delta_x$ containing $x \cap \delta_x$ and we let $y = y' \cup
[\delta_x,\delta_x +m]$.
\smallskip

\noindent
\underline{Case 3, $\delta_x>0$ and $\delta_x \in S$}: Choose $m$ 
such that $\max(x)= \delta_x +m$ and
choose $n$ such that 
$\eta_{\delta_x}(n) > \sup(x \cap \delta_x)$.
Let

\[
x^+ = (x\cap \delta_v) \cup \{\eta_{\delta_x}(\ell):\ell<n\}\cup\
\bigcup\limits_{\ell < n} u_{\delta_x,\ell}.
\]

\mn
so max$(x^+) < \eta_\delta(n) < \delta_x$ (so $\delta_{x^+} < \delta_x$).  Hence by the induction hypothesis
there is a ${\gp}-\bar v$-closed $y'$ such that for some $k<\omega$ we have
$$x^+ \subseteq y'\subseteq \max (x^+)+1 \le \eta_{\delta_x}(n) <
\delta_x$$ and let
$y=y'\cup[\delta_x, \delta_x +m]$, remembering that $\delta_x +m = \max(x)$ it
is easy to check that we are done. 

\end{PROOF}

Also note that
\mn
\newcounter{pmocou} \setcounter{pmocou}{0}
\begin{enumerate}[label = $(\intercal)_{\arabic*}$, ref = $(\intercal)_{\arabic*}$]
	\setcounter{enumi}{\value{pmocou}}
	\item the intersection of a family of $\gp-\bar v$-closed subsets of
	$\omega_1$ is closed,
	\stepcounter{pmocou}
\end{enumerate}
hence 
\begin{enumerate}[label = $(\intercal)_{\arabic*}$, ref = $(\intercal)_{\arabic*}$]
	\setcounter{enumi}{\value{pmocou}}
	\item   for any finite $x \subseteq \omega_1$ the
	closure of $x$, $\cll^{\gp-\bar v} = \cap\{y:y$ is finite $\gp-\bar v$-closed and contains
	$x\}$ is finite,  $\gp-\bar v$-closed, contains $x$ and has the same maximum. 
	\stepcounter{pmocou}
\end{enumerate}

 Note the following assertions
\mn
\begin{enumerate}[label = $(\intercal)_{\arabic*}$, ref = $(\intercal)_{\arabic*}$]
	\setcounter{enumi}{\value{pmocou}}
	\item \label{stage2}  If $c_\delta$ is a function 
	with domain $\ran(\eta_\delta)$ 
	for each $\delta \in S$ satisfying 
		\begin{enumerate}
		\item[$(*)$]    for each $\alpha<\omega_1,$ the set $\mathbf c^\alpha=:
		\{c_\delta(\alpha):\alpha\in \ran(\eta_\delta),\delta\in S\}$
		is countable,
	\end{enumerate}
	\then \, we can uniformize 
	$\langle  c_\delta:\delta\in S \rangle$, i.e.
	\begin{itemize}
		\item find a 
		function $c$ satisfying Dom$(c) = \omega_1$ and $c(\alpha)$ belongs to
		$\{c_\delta(\alpha):\alpha\in \ran(\eta_\delta)\} \cup
		\{0\}$, for every $\alpha < \omega_1$ and $\bigwedge\limits_{\delta\in
			S} (c_\delta \subseteq^* c)$;
		\item moreover, there are $\bar m = \langle
		m_\delta:\delta \in S \rangle$ and a function $\mathbf m:\omega_1
		\rightarrow \omega$ such that for every $\delta \in S$ we have: 
		$$n \in
		[m_\delta,\omega) \Rightarrow f(\eta_\delta(n)) =
		c_\delta(\eta_\delta(n)),$$
		and $(\forall^*n)(\mathbf m(\eta_\delta(n)) = m_\delta)$,
	\end{itemize} 

	\stepcounter{pmocou}
\end{enumerate}

[Why?  Let $g_\alpha:\{c_\delta(\alpha):\alpha \in \ran
(\eta_\delta)\} \rightarrow \omega$ be one to one, let
$c'_\delta:\ran(\eta_\delta) \rightarrow \omega$ be
$c'_\delta(\eta_\delta(n)) = 
g_{\eta_\delta(n)}(c_\delta(\eta_\delta(n))$.  Now we can find a function
$f'$ which uniformize $\langle c'_\delta:\delta \in S\rangle$, see
Definition \ref{0.3}(2).  It exists as we are assuming 
$(\mathbf{D})_{S,\bar \eta^*}$.  Let $m_\delta = \text{ Min}\{m:0 < m < \omega$ and
$(\forall n)(m \le n < \omega \rightarrow
f'(\eta_\delta(n)) = c'_\delta(\eta_\delta(n))$. 

Then we apply $(\mathbf{D})_{S,\bar \eta^*}$ to $\langle 
c''_\delta:\delta \in S\rangle$ where $c''_\delta$ is defined by
 $c''_\delta(\eta_\delta(n)) = m_\delta$, getting $f''$.
Let $\mathbf m = f''$ and define $c$ with domain $\omega_1$
by $c(\alpha)$ satisfies $g_\alpha(c(\alpha)) = f'(\alpha)$ when
$f'(\alpha) \in \ran(g_\alpha)$ and $c(\alpha) = 0$ otherwise.]

For a given sequence $\langle c_\delta^i:\delta\in S \rangle$,
we use $c^i,\mathbf m^i$ to denote uniformizing functions like $c$
and $\mathbf m$ as above.

\mn

\begin{definition} \label{verdf}
	Suppose that $\gq =( S, \bar \eta, \bar u, \Psi)$ is a special $S$-uniformization problem satisfying 	\begin{enumerate}[label = $\alph*)$, ref =$\alph*)$]
		\item  \label{k3a}  $u_{\delta,n} \cup \{\eta_\delta (n)\}
		\subseteq u_{\delta,n+1}$, 
		\sn
		\item  \label{k3d} $\delta = \bigcup\limits_{n<\omega} u_{\delta,n}$.
	\end{enumerate}
We call $\gq$ \emph{very special} if in addition to \ref{k3a}, \ref{k3d}  there is a system $\bar v = \langle v_\alpha: \ \alpha < \omega_1 \rangle$ with for each $\delta \in S$,  $n \in \omega$: $u_{\delta,n} = v_{\eta_\delta(n)}$.
	
\end{definition}

\begin{lemma} \label{le1}
	Let $\gp =( S, \bar \eta, \bar u, \Psi)$ be a special $S$-uniformization problem. If $\bar \eta$ has $\aleph_0$-uniformization, then there is a very special $S$-uniformization problem $\gp'$, each solution of which is a solution of $\gp$.
\end{lemma}
\begin{PROOF}{Lemma \ref{le1}}
\begin{definition} \label{gp+u'}
	If  $\gp = (S, \bar \eta,
	\bar u, \Psi)$ is a special $S$-uniformization problem, and we are given  $\bar v = \langle v_{\delta,n}: \ \delta \in S, n \in \omega \rangle$ with 
	
	$$(\forall \delta\in S, \ \forall  n\in \omega): \ \ u_{\delta,n} \subseteq v_{\delta,n} \subseteq \eta_\delta(n),$$
	(so $\gp^0_{+\bar v} = (S, \bar \eta, \bar v)$ is an $S$-uniformization frame),
	and for each $b= (b^\delta_{n,k})^{\gp^0_{+\bar v}} \in \cB^{\gp^0_{+\bar v}}$ let 
	$$f \in \Phi((b^\delta_{n,k})^{\gp^0 +\bar v}) \ \iff \ \forall b \in \mathcal{B}^\mathfrak{p}: \ b \subseteq (b^\delta_{n,k})^{\gp^0 +\bar v}) \ \to \  (f \restriction b \in \Psi(b),$$
	then we can define the tuple
	$$ \gp_{+\bar v} = (S, \bar \eta, \bar v, \Phi).$$
	
\end{definition}

\begin{claim} \label{nagyobbu}
	If  $\gp = (S, \bar \eta,
	\bar u, \Psi)$ is a special $S$-uniformization problem, and we are given  $\bar v = \langle v_{\delta,n}: \ \delta \in S, n \in \omega \rangle$ with 
	
	$$(\forall \delta\in S, \ \forall  n\in \omega): \ \ u_{\delta,n} \subseteq v_{\delta,n} \subseteq \eta_\delta(n),$$
	(so $\gp^0_{+\bar v} = (S, \bar \eta, \bar v)$ is an $S$-uniformization frame),
	then $\gp_{+ \bar v} = (S, \bar \eta, \bar v, \Phi)$ from the definition above is a special $S$-uniformization problem, moreover, each solution $g$ of $\gp_{+ \bar v}$ is a solution of $\gp$.
	
\end{claim}
\begin{PROOF}{Claim \ref{nagyobbu}}
	First we check \ref{pda} in Definition $\ref{1.2}$. Let $w \in [\omega_1]^{<\omega_1}$ be an element of $\cB^{\gp_{+\bar v}}$ (recall Definition $\ref{1.2}$). Now as $w$ is finite $w \cap S$ is also finite, and it is easy to see that $\{ b \in \cB^\gp: \ b \subseteq w \}$ is finite. Now since $(b^\delta_{n,k})^{\gp^0 + \bar v} \supseteq b^\delta_{n,k} \in \cB^{\gp^0}$, and the fact that  $$(\forall k\in \omega) \ \ \Upsilon_k = \{ \{f \circ \text{OP}_{b,k}: \ f \in \Psi(b)\}: \ b \in \cB^\gp, \ |b| = k\} \ \text{ is countable}$$ it can be easily seen that \ref{pdb} indeed holds.
	
	For \ref{pdc} notice that if $b' \subseteq b''$ are in $\mathcal{B}^\mathfrak{p+\bar v}$, then by a reformulation of the definition of $\Phi$:
	$$\Phi(b') = \bigcap_{b \in \mathcal{B}^\mathfrak{p}, b \subseteq b'} \{ f: b' \to C: f \rest b \in \Psi(b) \},$$
	$$\Phi(b'') = \bigcap_{b \in \mathcal{B}^\mathfrak{p}, b \subseteq b''} \{ f: b'' \to C: f \rest b \in \Psi(b) \}.$$
	
	Finally as every $\gp_{+\bar v}$-closed set is $\gp$-closed as well, moreover, observe that if a $\gp_{+\bar v}$-closed set $w$ contains the $\gp$-basic set $b^\delta_{n,k}$ as a subset, then  for $b'= (b^\delta_{n,k})^{\gp^0 + \bar v}$ we have $w \supseteq b' \supseteq b^\delta_{n,k}$. This implies 
	\[ w \text{ is } \gp_{+\bar v}\text{-closed} \ \to \ \Phi^+(w) = \Psi^+(w),\]
	so \ref{pdd} holds as well.
	
	As for the moreover part again recall that for each $b \in \cB^\gp$ there exists $b' \in \cB^{\gp_{+\bar v}}$ with $b' \supseteq b$, so if $g$ is a function with $\dom(g) = \omega_1$, $\forall b' \in \cB^{\gp_{+\bar v}}$ $g \restriction b' \in \Phi(b')$, then $\forall b \in \cB^\gp$ $g \restriction b \in \Psi(b)$.
\end{PROOF}

%\newcounter{pmocou} \setcounter{pmocou}{0}
\begin{enumerate}[label = $(\intercal)_{\arabic*}$, ref = $(\intercal)_{\arabic*}$]
	\setcounter{enumi}{\value{pmocou}}
	\item So using the claim, replacing $u_{\delta,n}$ by any finite $u'_{\delta,n}$ satisfying $u_{\delta,n} \subseteq u'_{\delta,n} 
	\subseteq \eta_\delta(n)$, and redefining $\Psi$ as there it suffices to solve the derived special uniformization problem.
	\stepcounter{pmocou}
\end{enumerate}

\begin{enumerate}[label = $(\intercal)_{\arabic*}$, ref = $(\intercal)_{\arabic*}$]
	\setcounter{enumi}{\value{pmocou}}
	\item \label{cle} we can define by induction on $\delta \in S$, $n \in \omega$ the finite sets $u'_{\delta,n}$ (and the $S$-uniformization frame $\gp^0_{+ \bar u'} = (S, \bar \eta, \bar u')$) satisfying $u_{\delta,n} \subseteq u'_{\delta,n}$ so that 
	\begin{enumerate}
		\item   $u'_{\delta,n} \cup \{\eta_\delta (n)\}
		\subseteq u'_{\delta,n+1}$, 
		\sn
		\item  $\delta = \bigcup\limits_{n<\omega} u'_{\delta,n}$,
	\end{enumerate}
	and the special $S$-uniformization problem $\gp_{+\bar u'} = (S, \bar \eta, \bar u', \Psi)$ from Definition \ref{gp+u'} it suffices to solve the problem $\gp_{+\bar u'}$.
	\stepcounter{pmocou}
\end{enumerate}

Observe that

\begin{enumerate}[label = $(\intercal)_{\arabic*}$, ref = $(\intercal)_{\arabic*}$]
	\setcounter{enumi}{\value{pmocou}}
	\item if  $u \subseteq \omega_1$ is $\gp$-closed and $\alpha<\omega_1$ are given,
	\then \, $u\cap \alpha$ is $\gp$-closed. 
	\stepcounter{pmocou}
\end{enumerate}

\begin{definition}\label{-m}
		Suppose that $\gp = (S, \bar \eta, \bar u, \Psi)$ is a special $S$-uniformization problem,   $\bar m = 
	\langle m_\delta:\delta\in S \rangle$ is 
	a sequence of natural numbers, define $\gp \setminus \bar m = \gp'$ as follows. Let 
	\begin{enumerate}[ label =$\bullet_{\arabic*}$, ref= $\bullet_{\arabic*}$ ]
		\item the ladder system 	$\bar \eta - \bar m = \bar \eta'$ on $S$  be defined as
		\[  \ \eta'_\delta(n) = \eta_\delta(n+m_\delta) \ \ (\delta \in S, \ n \in \omega),\]
		\item 	and the system $\bar u'$ is defined as
		\[ u'_{\delta,n} = u_{\delta,n+m_\delta} (\subseteq \eta_\delta(n+m_\delta) = \eta_\delta'(n)),\]
		\item and for each $\gp'$-basic set $(b')^\delta_{n,k}=  \bigcup_{j \leq n} (u'_{\delta,j} \cup \{ \eta_\delta'(j)\}) \cup [\delta,\delta+k)$ (so here $(b')^\delta_{n,k}$ just equals the $\gp$-basic set $b^\delta_{n+m_\delta,k}$), let $\Psi'((b')^\delta_{n,k})= \Psi(b^\delta_{n+m_\delta,k})$.
	\end{enumerate}	
\end{definition}

\begin{claim} \label{0vag} 
	Suppose that $\gp = (S, \bar \eta, \bar u, \Psi)$ is a special $S$-uniformization problem, and furthermore $\bar u$ satisfies clauses \ref{k3a} and \ref{k3d} from Definition $\ref{verdf}$.
	 Assume that  $\bar m = 
	\langle m_\delta:\delta\in S \rangle$ is 
	a sequence of natural numbers.
	Then $\gp' = \gp \setminus \bar m = (S, \bar \eta', \bar u', \Psi')$ is a special $S$-uniformization problem,
	moreover, for every solution $f$ of $\gp'$, $f$ is a solution of $\gp$ too.
\end{claim}
\begin{PROOF}{Claim \ref{0vag}}
	It is easy to see that \ref{k3a} and \ref{k3d} from Definition $\ref{verdf}$ imply $\cB^{\gp'} \subseteq \cB^{\gp}$, and $\Psi(b) = \Psi'(b)$ if $b \in \mathcal{B}^{\gp'}$, so \ref{pdb}, \ref{pdc} automatically hold.
	Also each $\gp'$-closed set $w$ is $\gp$-closed, and so
	$$\Psi^+(w) = \bigcap_{b \in \mathcal{B}^{\mathfrak{p}'}, b \subseteq w} \{ f: w \to C: f \rest b \in \Psi'(b) \}=$$
	$$= \bigcap_{b \in \mathcal{B}^{\mathfrak{p}'}, b \subseteq w} \{ f: w \to C: f \rest b \in \Psi(b) \},$$
	and if $b \in \mathcal{B}^\gp \setminus \mathcal{B}^{\mathfrak{p}'}$, then $b = b^\delta_{n,k}$ for some $n < m_\delta$, $b \subseteq (b^\delta_{0,k})^{\gp'}$, and if $b \subseteq w$, $w$ is $\gp'$-closed, then $(b^\delta_{m_\delta,k})^\gp=(b^\delta_{0,k})^{\gp'} \subseteq w$. Finally, since for any $f \in \Psi'((b^\delta_{0,k})^{\gp'})$ we have $f \in \Psi((b^\delta_{m_\delta,k})^{\gp})$, and $f \rest b \in \Psi(b)$ (by \ref{pdc} for $\gp$).
	This easily implies that for any $\gp'$-closed $w$, we have $\Psi^+(w) = (\Psi^+(w))'$, and so $\gp'$ is indeed a special $S$-uniformization problem (similarly to the proof of Claim $\ref{nagyobbu}$). (Note that $\gp$-closed sets are not necessarily $\gp'$-closed.
	
	For the moreover part observe that each old basic set $b^\delta_{n,k} = \bigcup_{j \leq n} (u_{\delta,j} \cup \{ \eta_\delta(j)\}) \cup [\delta,\delta+k)$ is contained in the new basic set
	$$(b^\delta_{\max(0,n-m_\delta),k})^{\gp'} = \bigcup_{j \leq \max(0,n-m_\delta)} (u'_{\delta,j} \cup \{ \eta'_\delta(j)\}) \cup [\delta,\delta+k),$$ 
	and for each $g$ with  $g \in \Psi'((b^\delta_{\max(0,n-m_\delta),k})^{\gp'})$
	(since we have $$\Psi'(b) = \Psi(b)^{\gp})) \ \textrm{ for } b \in \mathcal{B}^{\gp'}$$ 
	by definition)
	we have $g \restriction (b^\delta_{n,k})^\gp \in \Psi((b^\delta_{n,k})^\gp)$ as \ref{pdc} holds for $\gp$.
	
\end{PROOF}

\begin{definition}
		If $\gq =( S, \bar \eta, \bar u, \Psi)$ is a very special $S$-uniformization problem witnessed by $\bar v = \langle v_\alpha: \ \alpha < \omega_1 \rangle$, then we call a set $X \subseteq \omega_1$ $\bar v$-closed, if $X$ is $\gq$-closed, and for each $\alpha \in X$ the inclusion $v_\alpha \subseteq X$ also holds.
\end{definition}
\begin{enumerate}[label = $(\intercal)_{\arabic*}$, ref = $(\intercal)_{\arabic*}$]
	\setcounter{enumi}{\value{pmocou}}
	\item if the special $S$-uniformization problem $\gp = ( S, \bar \eta, \bar u, \Psi)$ is  very special  witnessed by $\bar v = \langle v_\alpha: \ \alpha < \omega_1 \rangle$, then with some slight abuse of notation we may 
	refer to also the tuple $(S, \bar \eta, \bar v, \Psi)$ as $\gp$.
	\stepcounter{pmocou}
\end{enumerate}

We summarize some useful facts about very special uniformization problems for future reference.
\begin{fact} \label{veryfa}
	If $\gp =( S, \bar \eta, \bar u, \Psi)$ is a very special uniformization problem, which fact is witnessed by $\bar v = \langle v_\gamma: \ \gamma < \omega_1 \rangle$, then
	\begin{enumerate}[label = $(\intercal)_{\arabic*}$, ref = $(\intercal)_{\arabic*}$]
		\setcounter{enumi}{\value{pmocou}}
		\item \label{veryfak3}  	\begin{enumerate}[label = $(\alph*)$, ref = $(\alph*)$]
			\item  \label{k3av}   $v_{\eta_\delta(n)} \cup \{\eta_\delta (n)\}
			\subseteq v_{\delta,n+1}$, 
			\sn
			\item  \label{k3dv} $\delta = \bigcup\limits_{n<\omega} v_{\eta_\delta(n)}$,
		\end{enumerate}
		therefore
		\begin{enumerate}[label =$(\alph*)$, ref = $(\alph*)$]
			\setcounter{enumii}{2}
			\item  \label{k3dv+}   $\ran( \eta_\delta \restriction n+1)
			\subseteq v_{\delta,n+1}$
		\end{enumerate}
		\stepcounter{pmocou}
	\end{enumerate}

	Also, $b^\delta_{n,k} = v_{\eta_\delta(n)} \cup \{ \eta_\delta(n)\} \cup [\delta, \delta+k]$ (so $\cB^{\gp} = \{ b^\delta_{n,k} : \ \delta \in S, \ n,k \in \omega\}$)
		\begin{enumerate}[label = $(\intercal)_{\arabic*}$, ref = $(\intercal)_{\arabic*}$]
		\setcounter{enumi}{\value{pmocou}}
		\item  \label{veryfa1} for $x \in [\omega_1]^{<\aleph_0}$
			\[ g \in \Psi^+(x) \iff \ [\forall \delta \in S \ \forall n,k: \ (b^\delta_{n,k} \subseteq x) \to (g \restriction b^\delta_{n,k} \in \Psi(b^\delta_{n,k}))].\]
		\stepcounter{pmocou}
	\end{enumerate}
	\begin{enumerate}[label = $(\intercal)_{\arabic*}$, ref = $(\intercal)_{\arabic*}$]
		\setcounter{enumi}{\value{pmocou}}
		\item  \label{veryfa1c} for $x \subseteq \omega_1$ is $\gp$-closed, if 
		\begin{enumerate}
			\item for each $\alpha \in x$, if $\alpha \in \ran(\eta_\delta)$ for some $\delta \in S \cap x$, then $v_\alpha \subseteq x$,
			\item for each $\delta \in x$ either for some $n$ $[\eta_\delta(n) \in x$ $\wedge$ $x \cap \delta \subseteq \eta_\delta(n+1)]$, or  $\ran(\eta_\delta) \subseteq x$,
			\item for each $\alpha +1 \in x$: $\alpha \in x$.
		\end{enumerate}
		\stepcounter{pmocou}
	\end{enumerate}

	\begin{enumerate}[label = $(\intercal)_{\arabic*}$, ref = $(\intercal)_{\arabic*}$]
		\setcounter{enumi}{\value{pmocou}}
		\item \label{levag} If $\bar m \in \ ^S \omega$, then the special $S$-uniformization problem $\gp' = \gp \setminus \bar m$ is also very special witnessed by the same $\bar v$, and then
			\begin{itemize}
				\item if $x \subseteq \omega_1$ is $(\gp \setminus \bar m) - \bar v$-closed, then it is $\gp- \bar v$-closed as well,
				\item for the basic sets we have 	$$\begin{array}{rl} \cB^{\gp \setminus \bar m} = & \{b^\delta_{n,k} \in \cB^\gp: \ n \geq m_\delta, \ k \in \omega\}, \\
					\Psi^{\gp \setminus \bar m} = & \Psi^{\gp} \restriction \cB^{\gp \setminus \bar m}, \end{array}$$
				\item hence 
					$$\text{if } x \in [\omega_1]^{<\aleph_0}, \text{ then } (\Psi^+)^{\gp \setminus \bar m}(x)  = (\Psi^+)^{\gp}(x).$$
			\end{itemize}
					
		\stepcounter{pmocou}
	\end{enumerate}
	
	\begin{enumerate}[label = $(\intercal)_{\arabic*}$, ref = $(\intercal)_{\arabic*}$]
		\setcounter{enumi}{\value{pmocou}}
		\item If $\bar v' = \langle v'_\alpha: \alpha < \omega_1 \rangle$ is such that for each $\alpha$ $v_\alpha \subseteq v'_\alpha \subseteq \alpha$ and $v'_\alpha$ is finite, then
		letting $u'_{\delta, n} = v'_{\eta_\delta(n)}$  the special uniformization problem $\gp_{+\bar u'}$ (as in Definition \ref{gp+u'}) is very special witnessed by $\bar v'$.
		\stepcounter{pmocou}
	\end{enumerate}
	
\end{fact}

Observe that  Claim $\ref{0vag}$ and the following finishes the proof of Lemma $\ref{le1}$.

\begin{claim} \label{verysim}
	If $\gp = (S, \bar \eta, \bar u, \Psi)$ is a special $S$-uniformization problem, $\bar u$ satisfies \ref{k3a}-\ref{k3d} from Definition $\ref{verdf}$ (where $\bar \eta$ has $\aleph_0$-uniformization), then for some $\bar m = \langle m_\delta: \ \delta \in S \rangle \in \ ^S\omega$ the special $S$-uniformization problem $\gp \setminus \bar m$ is very special.
\end{claim}
\begin{PROOF}{Claim \ref{verysim}}
	We have to find a sequence $\bar m$,  such that for any $\alpha$ if $\alpha = \eta_\delta(n_0+m_\delta) = \eta_\delta'(n_1 + m_{\delta'})$, then $u_{\delta, n_0 } = u_{\delta',n_1}$.
	So  define the functions $c_\delta$ for $\delta \in S$ by $c_\delta(\eta_\delta(n)) = u_{\delta,n}$. Again, as for any fixed $\alpha < \omega_1$ 
	$$\{c_\delta(\alpha): \ \delta \in S \ \wedge \ \alpha \in \ran(\eta_\delta)\} \subseteq [\alpha]^{< \aleph_0}$$
	there exists some function $c^*$ uniformizing $\langle c_\delta: \ \delta \in S \rangle$ and
	$\bar m$ with for each $\delta$, $n \geq m_\delta$ the equality $c^*(\eta_\delta(n)) = c_\delta(\eta_\delta(n))$ holds, which is exactly what we wanted.
\end{PROOF}

\end{PROOF}

\begin{definition}
	We call the very special $S$-uniformization problem $\gq = (S, \bar \eta, \bar v, \Psi)$    \emph{nice} if it satisfies that 
	for every $\delta \in S$ $\ran(\eta_\delta) 
	\subseteq \Omega$, and for the system  $\bar v = \langle v_\alpha: \ \alpha < \omega_1 \rangle$: $\alpha \notin \Omega$ $\to$ $v_\alpha = \emptyset$.
\end{definition}

\begin{lemma} \label{le2}
Let $\gp=  (S, \bar \eta, \bar v, \Psi)$ be a very special $S$-uniformization problem. If each ladder system $\bar \eta^* = \langle \eta^*_\delta: \ \delta \in S \rangle$ very similar to $\bar \eta$ has $\aleph_0$-uniformization then there exists a a very special $S$-uniformization problem $\gp' = (S, \bar \eta', \bar v', \Psi')$, and a sequence of natural numbers $\bar m = \langle m_\delta: \ \delta \in S \rangle$, such that $\gp'' = \gp' \setminus \bar m$ is nice, and if  $\gp''$ has a solution, then so does $\gp$.
\end{lemma}
\begin{PROOF}{Lemma \ref{le2}}

\begin{definition}
	\label{veryspecE}
	Assume that $E$  is 
	an equivalence relation on $\omega_1$, satisfying that
	each equivalence class is a finite convex set (and $\gp = (S,\bar \eta, \bar v, \Psi)$ is a very special $S$-uniformization problem).  We define ${\gp}' =
	{\gp}/E$ as follows.
	
	Let $\{e_i:i<\omega_1\}$ list the equivalence classes in increasing order
	(hence $\delta\in e_\delta$ for $\delta$ limit ordinal
	$<\omega_1$); assume further $\bar \eta' = \langle \eta'_\delta:
	\delta\in S \rangle, \bar v = \langle v_{\alpha}: 
	\ \alpha < \omega_1 \rangle$
	are defined as follows:
	\mn
	\begin{enumerate}[ label = (\greek*), ref = (\greek*)]
		\item   $\eta'_\delta$ enumerates in increasing order
		$$\{\alpha: \ (\exists n) \eta_\delta(n) \in e_\alpha\}$$ 
		
		(so $\bar \eta' = \langle \eta'_\delta:\delta\in S \rangle$ is an 
		$S$-ladder system),
		\sn
		\item \label{beta} if $\alpha < \omega_1$, \then 
		\begin{enumerate}
			\item[$\bullet_1$]  $v^*_{\alpha} = 
			\{\beta<\alpha: \  \text{ for some } \gamma \in e_\alpha \ \ e_\beta \cap v_\gamma
			 \ne \emptyset\}$,
			 \item[$\bullet_2$] then $v'_\alpha = \cll^{\bar v^*}(v^*_\alpha)$ (where $\bar v^* = \langle v^*_\alpha: \ \alpha < \omega_1 \rangle$),
		\end{enumerate}
		\item  \label{gasm}   $\Psi'$ with $\dom(\Psi') = \cB^{(\gp')^0}= \{ (b^\delta_{n,k})': \ \delta \in S, \ n,k \in \omega \}$ is defined as follows
	%	naturally: if $(b^\delta_{n,k})' \in \cB^{(\gp')^{0}}$, and $n_0$ is the maximal s.t.\ $\eta_\delta(n_0) \in e_{\eta'_\delta(n)}$, $k_0$ is maximal with $\delta+k_0 \in e_{\delta+k}$,  then 
	$$\Psi'((b^\delta_{n,k})')=\left\{ \begin{array}{ll} f_g: & \left( \dom(g) = \bigcup\limits_{\alpha\in (b^\delta_{n,k})'}e_\alpha \right) \ \& \\ & \left(\forall b \in \mathcal{B}^\mathfrak{p}: \ b \subseteq \bigcup\limits_{\alpha\in (b^\delta_{n,k})'}e_\alpha \ \to \ g\restriction b \in \Psi(b) \right) \end{array} \right\},
		$$ where for $g \in \Psi^+(\bigcup\limits_{\alpha\in b} e_\alpha)$ 
		we let dom$(f_g) = b$ and  $f_g
		(\alpha)= \langle g(\beta_0), g(\beta_1), \dots, g(\beta_{j-1}) \rangle$ where $e_\alpha = \{ \beta_0 < \beta_1 < \dots < \beta_{j-1} \}$,
		\sn
		%\item[$(\delta)$]    $\bar h'$ is 
		%defined naturally (and remember clause $(\gamma)$ of stage 2)
		%\sn
		\item${\gp}' = {\gp}/(E,\bar m) = 
		{\gp}/(\{e_\alpha:\alpha < \omega_1\})$ is 
		$(S,\bar \eta',\bar v',\Psi')$.
	\end{enumerate}
	\mn
\end{definition}

\begin{claim} \label{vekvag} 
	Suppose that $\gp = (S, \bar \eta, \bar v = \langle v_\alpha: \ \alpha < \omega_1 \rangle, \Psi)$ is a very special $S$-uniformization problem. Assume that  $E$  is 
	an equivalence relation on $\omega_1$, satisfying that
	each equivalence class is a finite convex set. \Then
	\mn
	\begin{enumerate}[label = $(\ast)_{\arabic*}$, ref =  $(\ast)_{\arabic*}$]
		\item \label{a1}   ${\gp}' = (S,\bar \eta',\bar v',\Psi')$ 
		forms a very special $S$-uniformization problem (defined above in Definition $\ref{veryspecE}$),
		\sn
		\item    if $f'$ is a solution for ${\gp}'$ 
		 \then \,
		there is a solution for ${\mathfrak p}$ (which is defined naturally satisfying the equality $\langle f(\beta): \ \beta \in e_\alpha \rangle = f'(\alpha)$),
		\sn
		\item  \label{a3}  new closed sets essentially are old closed sets, that is, if $u' \subseteq \omega_1$ is ${\gp}'$-closed (i.e.
		for $\bar \eta'$ ,$\bar v'$) \then \,
		$\bigcup \{e_\alpha:\alpha \in u'\}$ is ${\gp}$-closed (i.e.
		for $\bar \eta,\bar v$),
		\sn
		\item    $\bar \eta',\bar \eta,\bar \eta^*$ are very similar (see
		Definition \ref{a6b}(4)),
		\sn
		\item \label{a5}    $\gp' = (S, \bar \eta', \bar v', \Psi')$ satisfies 
		\begin{enumerate}
			\item for each $\delta \in S$: $\bigcup_{n \in \omega} v'_{\eta_\delta(n)} = \delta$,
			\item for each $\delta \in S$, $n \in \omega$: $v'_{\eta'_\delta(n)} \cup \{\eta'_\delta(n)\} \subseteq v'_{\eta'_\delta(n+1)}$.
		\end{enumerate}	
	\end{enumerate}
\end{claim}
\begin{PROOF}{Claim \ref{vekvag}}
	We first verify \ref{a3}. If $x$ is $\gp'$-closed, letting $x^e = \bigcup\{e_\alpha: \ \alpha \in x\}$, 
	then it is enough to show that if $\delta \in x$ (equivalently, $\delta \in x^e$),
	and $\eta_\delta(k) \notin x^e$, then  $(\eta_\delta(k), \delta) \cap x^e = \emptyset$, 
	or else, $\eta_\delta(k) \in x^e$ implies $v_{\eta_\delta(k)} \subseteq x^e$.
	
	First, if $\eta_\delta(k) \notin x^e$, then for the natural number $n$ defined by $\eta_\delta(k) \in e_{\eta'_\delta(n)}$
	clearly $\eta'_\delta(n) \notin x$, so for each $\alpha \in (\eta'_\delta(n), \delta)$, $\alpha \notin x$, which in turn implies that $[\min(e_{\eta'_\delta(n)}), \delta) \cap x^e = \emptyset$, so $(\eta_\delta(k), \delta) \cap x^e = \emptyset$.
	On the other hand, if $\eta_\delta(k) \in x^e$, then $\eta'_\delta(n) \in x$, so 
	$v'_{\eta'_\delta(n)} \subseteq x$, in particular $v^*_{\eta'_\delta(n)} \subseteq x$, and so recalling \ref{beta} again $v_{\eta_\delta(k)} \subseteq x^e$.
	
	Now we can turn to \ref{a5}, here we use that $v'_\alpha$'s are defined in two steps (instead of just letting $v'_\alpha = v^*_\alpha$).
	So fix $\delta \in S$, $n \in \omega$, it is clear from the construction that $\eta'_\delta(n) \in v^*_{\eta'_\delta(n+1)} \subseteq v'_{\eta'_\delta(n+1)}$ (since $\eta_\delta(\ell) \in v_{\eta_\delta(k)}$ always holds if $\ell < k$ by \ref{veryfak3}).
	For $v'_{\eta'_\delta(n)} = \cll^{\bar v^*}(v^*_{\eta'_\delta(n)}) \subseteq v'_{\eta'_\delta(n+1)}$ we only have to argue that $v^*_{\eta'_\delta(n)} \subseteq v'_{\eta'_\delta(n+1)}$. But as we have already seen that $\eta'_\delta(n) \in v^*_{\eta'_\delta(n+1)}$, the $\bar v^*$-closure of $v^*_{\eta'_\delta(n+1)}$ must contain $v^*_{\eta'_\delta(n)}$.
	
	We are ready to argue \ref{a1}. For \ref{pda} it can be easily seen that $\Psi'((b^\delta_{n,k})')$ can be defined by a single member of $\bigcup_{\ell < \omega} \Upsilon_\ell$, and the position of the (finitely many) $\gp$-basic sets in $\bigcup_{\alpha \in (b^\delta_{n,k})'} e_\alpha$. 
	It is straightforward to check that if $y = x^e = \bigcup_{\alpha \in x} e_\alpha$
	and $g \in (\Psi^+)'(y)$, then $f_g \in \Psi^+(x)$.
	Moreover, if $x$ is $\gp'$-closed, then $y = x^e$ is $\gp$-closed, and similarly to the proof of Claim $\ref{nagyobbu}$ one can easily check that $f_g \in (\Psi^+)(y)$ is equivalent to $g \in (\Psi')^+(x))$ (check that if some $\gp$-basic set $b^\delta_{n^*,k^*} \subseteq x^e$ contributes to the restrictions (where $n^*,k^*, n,k$ as in \ref{gasm}) then by the $\gp'$-closedness of $x$ there is a corresponding $\gp'$-basic set $(b^\delta_{n,k})' \subseteq x$).
	
	Finally, let $f' = f_g$ for some function $g$ with domain $\omega_1$. It is straightforward to check that as for each $\delta \in S$, $n,k \in \omega$ $f' = f_g \restriction (b^\delta_{n,k})' \in \Psi'((b^\delta_{n,k})')$, by the definition of $\Psi'$ we have $g\restriction b^\delta_{n^*,n^*} \in \Psi(b^\delta_{n^*,k^*})$ (where $n^*,k^*$ as in \ref{gasm}), and for each $b^\delta_{n^*,k^*}$ there exists some $(b^\delta_{n',k'})'$ with $ b^\delta_{n^*,k^*} \subseteq \bigcup_{\alpha \in (b^\delta_{n',k'})'} e_\alpha$, so  we are done. 

\end{PROOF}

\noindent
Applying Claim $\ref{vekvag}$ we have the following (thus finishing the proof of Lemma $\ref{le2}$):
\begin{enumerate}[label = $(\intercal)_{\arabic*}$, ref = $(\intercal)_{\arabic*}$]
	\setcounter{enumi}{\value{pmocou}}
	\item \label{limr} (if the very special $\gp= (S, \bar  \eta, \bar v, \Psi)$ is given) for some $E$ and $\bar m$ the very special $S$-uniformization ${\gp}'' = (\gp / E) \setminus \bar m = (S, \bar  \eta'', \bar v'', \Psi'')$ given by Definition $\ref{veryspecE}$ satisfies $\ran(\eta''_\delta) 
	\subseteq \Omega$, and $\alpha \notin \Omega$ $\to$ $v''_\alpha = \emptyset$.
	\stepcounter{pmocou}
\end{enumerate}
\begin{PROOF}{}(\ref{limr})
Define for $\delta\in S$ the function $\eta'_\delta:\omega \to \delta$ listing 
$$\{\alpha \in \Omega:\ [\alpha, \alpha+\omega)\cap \ran(\eta_\delta) \ne \emptyset\}.$$
As each $\delta \in S$ is divisible by $\omega^2$ according to our assumptions (i.e.\ $S$ is simple we are proving Theorem $\ref{1.4}$) $\bar\eta' = \langle \eta'_\delta:\delta \in S\rangle$ is an
$S$-ladder system, which is also very similar to $\bar\eta$, and hence to $\bar\eta^*$.
For $\delta \in S$, define $c_\delta$ by $c_\delta(\eta'_\delta(n))=
\eta_\delta \restriction (k_{\delta,n} + 1)$, where $k_{\delta,n} <
\omega$ is the maximal $k$  such that
$\eta_\delta(k)<\eta_\delta(n) + \omega$.
Note that the assumptions of \ref{stage2} hold (concerning $(*)$, $c_\delta(\alpha)$ is
from ${}^{\omega>}(\alpha+\omega)$ which is countable). So the conclusion
of $(\alpha)$ of \ref{stage2} holds, say, witnessed by $c$.
Let $\alpha E\beta$ iff $\alpha=\beta$ or $\alpha = \delta+k$, $\beta = \delta+\ell$ for some $\delta \in S$ and $k,\ell \in \omega$, and $c(\delta)$ is a
finite sequence with last element in $[\alpha, \alpha+\omega)\cap
[\beta, \beta+\omega)$.
Then use Claim $\ref{vekvag}$ to obtain the very special $S$-uniformization problem $\gp' = \gp /E$. So if $\gp' = (S, \bar \eta', \bar v', \Psi')$, then for each $\delta \in S$, $\ran(\eta'_\delta)$ contains only finitely many ordinals outside $\Omega$, say, the largest is at most the $m_\delta-1$'th. Now apply  \ref{levag}, letting $\gp'' = \gp' - \bar m$, and if
$\gp'' = (S, \eta'', \bar v'', \Psi'')$, where each $\ran(\eta_\delta'')$ consists of only limit ordinals, then we can clearly assume that $v''_{\alpha+1} = \emptyset$ for $\alpha <\omega_1$, this does not change the set $\mathcal{B}^{\gp''}$ of $\gp''$-basic sets or $\gp''$-closed sets (Definitions \ref{1.2}, \ref{pprobd}), since ran($\eta_\delta) \subseteq \Omega$), so it doesn't change the set of solutions.
\end{PROOF}

\end{PROOF}

We summarize the steps taken so far in the following lemma.

\begin{lemma}\label{nicele}
	Suppose that $S$ is a stationary set of limit ordinals which is simple, every ladder system $\bar \eta$ very similar to $\bar \eta^*$ has $\aleph_0$-uniformization. For every special $S$-uniformization $\mathfrak{p}$ there exists a very special $S$-uniformization problem $\mathfrak{q} = (S, \bar \eta^{\gq}, \bar v^{\gq}, \Psi^\gq)$ which is nice, and every solution to $\mathfrak{q}$ is a solution to $\gp$, too, and every ladder system very similar to $\bar \eta^\gq$ has $\aleph_0$-uniformization.
\end{lemma}
\begin{PROOF}{Lemma \ref{nicele}}
	By Claim \ref{nagyobbu} and \ref{cle} we reduce the problem to solving the special uniformization problem $\gp_{+\bar u'} = (S, \bar \eta, \bar u', \Psi)$, where $\bar u'$ satisfies \ref{k3a}, \ref{k3d} from Definition \ref{verdf}.
	Then Claim \ref{verysim} yields $\gp'' = (S, \bar \eta'', \bar  v'', \Psi'')$, a very special $S$-uniformization problem  on the ladder system $\bar \eta'' = \bar \eta \setminus \bar m$ for some $\bar m$, which is very similar to $\bar \eta$.
	
	Finally, applying Lemma \ref{le2} to $\gp''$, it reduces our problem to solving a very special $S$-uniformization problem $\gp''' = (S, \bar \eta''', \bar v''', \Psi''')$,  where $\bar \eta'''$ is very similar to $\bar \eta''$ (and satisfies that for each $\delta \in S$ ran$(\eta_\delta)$ contains finitely many ordinals outside $\Omega$, as is seen from the following).
	Moreover, for some $\bar m$, $\gp''' \setminus \bar m$ is a very special $S$-uniformization problem on $\bar \eta^{(4)} = \bar \eta''' \setminus \bar m$ which is nice, and $\bar \eta^{(4)}$ being a finite modification of $\bar \eta'''$, is very similar to $\bar \eta'''$ (so $\gq = \gp''' \setminus \bar m$ works).
\end{PROOF}

\begin{fact} \label{smotvag}
 If the nice $S$-uniformization problem $\gq = (S, \bar \eta, \bar v, \Psi)$  and $\langle m_\delta: \ \delta \in S \rangle$ are given, then $\gq \setminus \bar m$ is also nice, each solution $f$ of the latter is solution for $\gq$ as well.
\end{fact}

\begin{lemma} \label{Lebe}
Let $\gp = (S, \bar \eta, \bar v, \Psi)$ be a very special $S$-uniformization problem, which is nice, and the $S$-ladder system $\bar \eta = \langle \eta_\delta: \ \delta \in S \rangle$ has $\aleph_0$-uniformization. Then $\gp$ has a solution.
\end{lemma}
\begin{PROOF}{Lemma \ref{Lebe}}
In  this finishing lemma we will not introduce new types of uniformization problems, but we still need to remove (for each $\delta \in S$) proper initial segments of $\ran (\eta_\delta)$ (i.e. appealing to Fact $\ref{smotvag}$) finitely many times, obtaining nicer and nicer systems (Claim $\ref{elok}$, Lemma $\ref{wlem}$, Lemma $\ref{ultuni}$). The consecutive refinements (applications of Fact $\ref{smotvag}$)  may  ruin earlier steps, for example the system $\bar \nu$ in the following claim, or the reader may think of the difference between $\gp$-closedness and $\gp-\bar m $-closedness, so we will have to do it carefully.
However, finally after Lemma $\ref{ultuni}$ we will not replace $\gp$ by $\gp - \bar m$, instead, in Claim $\ref{Ff}$ we will construct a solution for $\gp$ using $\gp- \bar m$-``histories'' for $\bar m$ given by Lemma $\ref{ultuni}$ (see below in Definition $\ref{hisdef}$).

After Claim \ref{Ff} we will put the pieces together to complete the proof of the present lemma (Lemma \ref{Lebe}).

\begin{claim} \label{nucl}
	Suppose that $\bar \eta$ is a ladder system on $S$ with $\aleph_0$-uniformization. Then there exists
	$\bar \nu = \langle \nu_\alpha: \ \alpha < \omega_1 \rangle$  and ladder system $\bar \eta'$ on $S$ such that 
	\begin{enumerate}[label = $\alph*)$, ref = $\alph*)$]
		\item $\bar \eta' = \bar \eta - \bar m$ for some $\bar m \in \ ^S\omega$, i.e. \ for each $\delta \in S$ 
		$$\eta'_\delta(n) = \eta_\delta(n+m_\delta),$$
		\item  for each $\alpha < \omega_1$,  $\nu_\alpha$ is a finite strictly 
		increasing sequence of ordinals $<\alpha$,
		\sn
		\item  for $\delta\in S$, $n<\omega$ we have 
		\[
		\nu_{\eta_\delta(n)}=\eta'_\delta\restriction n, \]
		%\item   $\ell < \ell g(\nu_\alpha) \Rightarrow
		%\nu_{\nu_\alpha(\ell)} = \nu_\alpha\restriction\ell$. 
	\end{enumerate}
\end{claim}
\begin{PROOF}{Claim \ref{nucl}}
	For each $\delta \in S$ define
	\[ c_\delta(\eta_\delta(n))=\eta_\delta \restriction n, \]
	(note that for each $\alpha$ $\{c_\delta(\alpha): \ \delta \in S\} \subseteq \ ^{<\omega}\alpha$, so is countable),
	let $c: \omega_1 \to \ ^{<\omega}\omega_1$ uniformize $c_\delta$.
	Finally define $\nu_\alpha^0 = c(\alpha)$, so w.l.o.g.\ for each $\alpha < \omega_1$ $\nu_\alpha^0 = \langle \nu_\alpha^0(0), \nu_\alpha^0(1), \dots, \nu_\alpha^0(|\nu^0_\alpha|-1)\rangle \in \ ^{<\omega}\alpha$ and let 
	$$k_\alpha =  \max\left( \{ n< |\nu^0_\alpha|: \ \nu^0_{\nu^0_\alpha(n)}  \neq \nu^0_\alpha \restriction n) \} \cup \{-1\}\right) +1,$$ 
	and let 
	$$m_\delta = \min \{ m \in \omega: \ (\forall \ell \geq m) \ \eta_\delta \restriction \ell = \nu^0_{\eta_\delta(\ell)}\}.$$ 
	Now observe that for each $\delta \in S$, if $\ell \geq m_\delta$, then $k_{\eta_\delta(\ell)} = m_\delta$, so letting
	\[ \nu_\alpha = \langle \nu^0_\alpha(k_\alpha), \nu^0_\alpha(k_\alpha+1), \dots, \nu^0_\alpha(|\nu^0_\alpha|-1) \rangle \in \ ^{|\nu^0_\alpha|-k_\alpha}\alpha,\]
	\[ \eta_\delta'(n) = \eta_\delta(n+m_\delta) \]
	works.
\end{PROOF}

%Note that
%\begin{enumerate}[label = $(\intercal)_{\arabic*}$, ref = $(\intercal)_{\arabic*}$]
%	\setcounter{enumi}{\value{pmocou}}
%	\item applying Claim $\ref{nucl}$ for a very special $S$-uniformization problem $\gp = (S, \bar \eta, \bar v, \Phi)$ which satisfies \ref{k3folyv} \ref{k3av} we get that
%	\[ \\]
%	\stepcounter{pmocou}
%\end{enumerate}

 \begin{definition} \label{Arank}
 	Assume that $\bar \eta$ is a ladder system on the stationary set $S$.
 	We say that the system of sets $\langle A_n: \ n \in \omega \rangle$ is a $\bar \eta$-rank, if the following hold:
 		\begin{enumerate}[label = $(\star)_{\alph*}$, ref = $(\star)_{\alph*}$]
 		\item $\langle A_n: n \in \omega \rangle$ is pairwise disjoint, covers $S \cup (\cup\{ \ran(\eta_\delta): \ \delta \in S\})$,
 		\item \label{sta2} if $\delta \in S \cap A_n$, then $\ran(\eta_\delta) \subseteq A_{n+1}$,
 	%	\item for $\gamma \in S$ 
 	%	$$\gamma \in A_0 \ \iff \ \text{for each } \delta \in S \ \gamma \notin \ran(\eta_{\delta}).$$
 	\end{enumerate}
 \end{definition}

Observe that 
\begin{enumerate}[label = $(\intercal)_{\arabic*}$, ref = $(\intercal)_{\arabic*}$]
	\setcounter{enumi}{\value{pmocou}}
	\item \label{limranko} if $\bar A$ is an $\bar \eta$-rank for the ladder system $\bar \eta$, and
	$\langle m_\delta: \delta \in S \rangle \in \ ^S\omega$, then $\bar A$ is also an $\bar \eta - \bar m$-rank.
	\stepcounter{pmocou}
\end{enumerate}

\begin{claim}\label{elok}
	Suppose that $\bar \eta$ is a ladder system on $S$ with $\aleph_0$-uniformization. Then there exists $\bar m = \langle m_\delta: \ \delta \in S \rangle$ and $\bar A = \langle A_n: n \in \omega \rangle$ such that letting $\bar \eta' = \bar \eta - \bar m$ (from Definition $\ref{-m}$) $\bar A$ is a $\bar \eta'$-rank.

\end{claim}
\begin{PROOF}{Claim \ref{elok}}
	First apply Claim $\ref{nucl}$, so (after possibly replacing $\eta$ with $\eta - \bar m^*$ for some $\bar m^*$) we can assume that there exists $\langle \nu_\alpha: \ \alpha < \omega_1 \rangle$ such that
	\[ \forall \delta \in S, \ n \in \omega: \ \nu_{\eta_\delta(n)} = \eta_\delta \restriction n. \]
	Define for $\delta \in S$ the function $c_\delta: \ran(\eta_\delta) \to  \omega$ as
	\[ c_\delta(\eta_\delta(n)) = |\nu_\delta|,  \] 
	let $c^*: \omega_1 \to \omega$ uniformize $\langle c_\delta: \ \delta \in S \rangle$.
	Now define the function $g: \omega_1 \to \omega$ by $g(\alpha) = \min(c^*(\alpha)+1, |\nu_\alpha|)$,
	(note that 
	\begin{equation} \label{gfels} g(\alpha) \leq |\nu_\alpha|) \ \ (\alpha \in \omega_1), \end{equation}
		 and choose $d_\delta$ so that
	\begin{equation} \label{gals}m \geq d_\delta \ \to \ [g(\eta_\delta(m)) =  c^*(\eta_\delta(m))+1 = |\nu_\delta| + 1] \ \ (> | \nu_\delta| \geq g(\delta)). \end{equation}
	Let $B_k = \{ \alpha < \omega_1: \ g(\alpha)=k \}$.
	So \eqref{gals} and \eqref{gfels} together imply that for each $\delta$ 
	\begin{itemize}
		\item $\delta \in B_{g(k)}$,
		\item  while $\ran(\eta_\delta  \restriction [d_\delta, \infty)) \subseteq B_k$ for some $k > g(\delta)$,
		\item so we let $C_k = \{ \delta \in S: \ \ran(\eta_\delta  \restriction [d_\delta, \infty)) \subseteq B_k \}$, which is a partition of $S$, and
		\[(\forall k \in \omega) \ C_k \subseteq \bigcup_{j <k} B_j.\]
	\end{itemize}
	Now we construct the $A_i$'s as follows in the form of $A_i = f^{-1}(i)$ for a suitable function $f$. We define $f \restriction \bigcup_{j\leq i} B_j$, $\langle m_\delta: \ \delta \in (\bigcup_{j\leq i} C_j) \rangle$ by induction on $i$ so that for each $i \in \omega$
	\begin{enumerate}[label =  $(\curlyvee)$, ref = $(\curlyvee)$ ]
		\item for each $\delta \in \bigcup_{j \leq i} C_j$ %(so $f \restriction \bigcup_{j\leq i} B_j$  defined $\to$ $f(\delta)$ is defined) 
		we have
		$$\forall m \geq m_\delta: \ f(\eta_\delta(m)) = f(\delta)+1.$$
	\end{enumerate}
	 So
	suppose that $f \restriction \bigcup_{j \leq i } B_j$, $\langle m_\delta: \ \delta \in (\bigcup_{j\leq i} C_j) \rangle$ are already constructed. For each $\delta \in S \cap (\bigcup_{j \leq i } B_j)$ with $\ran(\eta_\delta  \restriction [d_\delta, \infty)) \subseteq B_{i+1}$ define the colouring $c'_\delta: \ran(\eta_\delta) \to \omega$ as $c'_\delta(\eta_\delta(n))= f(\delta)$. Now if $c'$ uniformizes the $c'_\delta$'s, then let $f \restriction B_{i+1} = c' \restriction B_{i+1}$, and choose for each $\delta \in S$ a suitable $m_\delta$.
\end{PROOF}

%So $A_n$'s partition. But $\nu_{\eta_\delta(n)} \triangleright \eta_\delta \restriction n$ (may be longer!)

\begin{definition} \label{wdf}
	Let $\gp = (S, \bar \eta, \bar v, \Psi)$ be a very special $S$-uniformization problem. We call
	the system $\bar w = \langle w_\alpha: \alpha < \omega_1 \rangle$ $\gp$-pressed down, if 
		\begin{enumerate}[label = $\bullet_{\arabic*}$, ref = $\bullet_{\arabic*}$]
		\item \label{w1} for each $\alpha$, $w_\alpha \subseteq \alpha$ is $\gp$-closed,
		\item for each $\delta \in S$, $n$ we have $v_{\eta_\delta(n)} \subseteq w_{\eta_\delta(n)} \subseteq w_{\eta_\delta(n+1)}$, (in particular $\ran(\eta_\delta \restriction n) \subseteq w_{\eta_\delta(n)}$)
		\item \label{w3}   for each $\delta \in S$, $n \in \omega$: $w_\delta \subseteq w_{\eta_\delta(n)}$.
	  	\end{enumerate}
\end{definition}

\begin{lemma} \label{wlem}
	Let $\gp = (S, \bar \eta, \bar v, \Psi)$ very special $S$-uniformization problem with $\bar \eta$ admitting $\aleph_0$-uniformization, the $\bar \eta$-rank $\langle A_n: \ n \in \omega \rangle$ given by Lemma $\ref{wlem}$.
	Then there exists $\bar m = \langle m_\delta: \ \delta \in S \rangle$ and a $(\gp \setminus \bar m)$-pressed down system $\bar w = \langle w_\alpha: \alpha < \omega_1 \rangle$.
	(And niceness is of course preserved if $\gp$ is nice.)

	%	\item if $\alpha$ is not of the form for 

\end{lemma}
\begin{PROOF}{Lemma \ref{wlem}}
	We can apply Claim $\ref{elok}$ and so remove an initial segment from each $\ran(\eta_\delta)$. First we construct $\bar m = \langle m_\delta: \ \delta \in S \rangle$,
	$\bar w^{(0)} = \langle w^{(0)}_\alpha: \ \alpha < \omega_1 \rangle$ with
	\begin{enumerate}[label = $\bullet_{\arabic*}$, ref = $\bullet_{\arabic*}$]
		\item for each $\delta \in S$, $n \geq m_\delta$: $v_{\eta_\delta(n)} \subseteq w^{(0)}_{\eta_\delta(n)}$,
		\item \label{w^03}   for each $\delta \in S$, $n \geq m_\delta$: $ w^{(0)}_{\eta_\delta(n)} = w^{(0)}_\delta  \cup v_{\eta_\delta(n)}$.
	\end{enumerate}

	We are going to construct $\langle w^{(0)}_\delta: \ \delta \in A_i \rangle$ by induction on $i$.
	If $\gamma \in A_0$, then $\gamma \notin \ran(\eta_\delta)$ for any $\delta \in S$, we can let $w^{(0)}_\gamma = v_\gamma \subseteq \gamma$.
	
	Assuming that $w_\gamma$ is constructed for each $\gamma \in \bigcup_{j \leq i} A_j$, and $m_\delta$ for $\delta \in \bigcup_{j<  i} A_j$ note that
	if $\gamma \in A_{i+1} \cap \ran(\eta_\delta)$, then $\delta \in A_{i}$ by our assumptions on the $A_n$'s (\ref{sta2} in Claim $\ref{elok}$).
	For $\delta \in A_{i} \cap S$  define $c_\delta: \ \ran(\eta_\delta) \to [\omega_1]^{<\omega}$,
	$$c_\delta (\eta_\delta(n))= \left\{ \begin{array}{ll} w_\delta^{(0)} \cup v_{\eta_\delta(n)} & \text{ if } w^{(0)}_\delta \subseteq \eta_\delta(n), \\
		v_{\eta_\delta(n)}, & \text{otherwise,} \end{array}
		 \right.$$
	note that $v_{\alpha} \subseteq \alpha$ by the definition of very special uniformization problems (Definition $\ref{verdf}$), so $w_\delta \subseteq \eta_\delta(n)$ implies that $c_\delta(\eta_\delta(n)) \in [\eta_\delta(n)]^{<\aleph_0}$ (since $v_{\eta_\delta(n)} \subseteq \eta_\delta(n)$), thus we can find a function $c^*$ with $\dom(c^*) = A_{i+1}$ uniformizing the $c_\delta$'s.
	W.l.o.g. we can assume that $c^*(\gamma) \supseteq v_{\gamma}$, set $w^{(0)}_\gamma = c^*(\gamma)$.
	Choose $m_\delta$ so that $n \geq m_\delta$ implies $c_\delta(\eta_\delta(n)) = c^*(\eta_\delta(n))$ and $w^{(0)}_\delta \cup v_{\eta_\delta(n)} =  c^*(\eta_\delta(n))$ (thus 
	$w^{(0)}_{\eta_\delta(n)} = c^*(\eta_\delta(n)) = w^{(0)}_\delta \cup v_{\eta_\delta(n)}$ if $n \geq m_\delta$).
	
	Once we defined the entire sequences $\bar m$, $\bar w^{(0)}$ it is straightforward to check that
	defining $w_\alpha = \cll^{\gp \setminus \bar m}(w^{(0)}_\alpha)$ works (since $v_\gamma \subseteq w^{(0)}_\gamma$ for $\gamma \in \omega_1$, and for each $\delta$, $n \geq m_\delta$, we have $w^{(0)}_{\eta_\delta(n)} = w^{(0)}_\delta \cup v_{\eta_\delta(n)}$, so
	$w^{(0)}_{\eta_\delta(n)} \supseteq w^{(0)}_\delta$).
	
\end{PROOF}
	
Observe that if we have a very special $S$-uniformization problem $\gp = (S, \bar \eta, \bar v, \Psi)$, and thus obtain the $\bar \eta$-rank $\langle A_i: \ i \in \omega\rangle$, and the $\gp$-pressed down system
$\bar w$ then
 it can be easily seen by induction on $k$ that

\begin{enumerate}[label = $(\intercal)_{\arabic*}$, ref = $(\intercal)_{\arabic*}$]
	\setcounter{enumi}{\value{pmocou}}
	\item  for any $\delta_0, \delta_1, \dots, \delta_{k}$, and $n_1, n_2, \dots, n_{k}$ with $\delta_{i+1} = \eta_{\delta_{i}}(n_{i+1})$ we have
			$$w_{\delta_0} \subseteq w_{\delta_k} \ \& \ v_{\delta_0} \subseteq w_{\delta_k}.$$ 
			
	\stepcounter{pmocou}
\end{enumerate}
Recall that if $\eta_{\delta}(n) = \delta^*$, then $\ran( \eta_{\delta} \restriction n) \subseteq v_{\delta^*}$ by \ref{k3dv+} from \ref{veryfak3}, so
\begin{enumerate}[label = $(\intercal)_{\arabic*}$, ref = $(\intercal)_{\arabic*}$]
	\setcounter{enumi}{\value{pmocou}}
	\item \label{fiok}  if we fix any $\delta_0, \delta_1, \dots, \delta_{\ell}$, and $n_1, n_2, \dots, n_{\ell}$ with $\delta_{i+1} = \eta_{\delta_{i}}(n_{i+1})$, then 
	for any $j < \ell$
	$w_{\delta_\ell} \supseteq \ran(\eta_{\delta_{j}} \restriction n_{j+1})$,
	in particular $\delta_\ell > \eta_{\delta_j}(n_{j+1}-1)$, so
	$$(\forall j < \ell) \ \ \delta_\ell \in 	(\eta_{\delta_j}(n_{j+1}-1), \eta_{\delta_j}(n_{j+1})).$$
	\stepcounter{pmocou}
\end{enumerate}

\begin{definition}
	If $\bar \eta$ is a ladder system on a stationary set $S$, then the sequence $\langle h_\alpha: \ \alpha \in \omega_1 \rangle \in \ ^{\omega_1}\omega$ is permitted (for $\bar \eta$), if
	for each $\delta \in S$ the sequence $h_{\eta_\delta(n)}$ ($n < \omega$) is strictly increasing.
\end{definition}

\begin{claim}
	If $\bar \eta$ is a ladder system on a stationary set $S$ 	admitting $\aleph_0$-uniformization,  then there exist sequences $\bar m = \langle m_\delta: \ \delta \in S \rangle \in \ ^S\omega$, $\bar h = \langle h_\alpha: \ \alpha \in \omega_1 \rangle \in \ ^{\omega_1}\omega$, such that $\bar h$ is permitted for $\bar \eta' = \bar \eta - \bar m$.%, moreover, for each $\delta \in S$ $h_{\eta'_\delta(n)} = k_\delta+n$ for some $k_\delta \in \omega$.
\end{claim}
\begin{PROOF}{}
	The proof is an easy application of $\aleph_0$-uniformization, and it is left to the reader.
\end{PROOF}

\begin{definition} \label{hisdef}
	Let $\bar \eta$ be a ladder system on the stationary set $S$. 
	\begin{enumerate}[label = $\arabic*)$, ref =$\arabic*)$]
		\item  We say that $\bar \delta = (\delta_i,\dots,\delta_{i+\ell})$ is a $\bar \eta$-history, if $i \in \bbZ$, $\ell >0$, and for each $0 \leq k < \ell$:
	$\delta_{i+ k+1} = \eta_{\delta_{i+k}}(n_{i+k+1})$ for some $n_{i+k+1} \in \omega$.
	\item 	We let $\bar n^{\bar \delta} = \langle n_{i+k+1}: \ k < \ell \rangle$ denote this sequence associated to $\bar \delta$.
	\item We call the $\bar \eta$-history $\bar \delta = (\delta_i,\dots,\delta_{i+\ell})$ maximal, if there is no $\bar \eta$-history $\bar \delta^*$ properly extending $\bar \delta$ (i.e. no suitable $\delta_{i-1}$, nor $\delta_{i+\ell+1}$ does exist),
	\item We call the $\bar \eta$-history $\bar \delta = (\delta_i,\dots,\delta_{i+\ell})$ long, if there is no $\delta_{i-1}$ such that $(\delta_{i-1}, \delta_i,\dots,\delta_{i+\ell})$ is a $\bar \eta$-history.
	\end{enumerate}
	
\end{definition}

Observe that
\begin{enumerate}[label = $(\intercal)_{\arabic*}$, ref = $(\intercal)_{\arabic*}$]
	\setcounter{enumi}{\value{pmocou}}
	\item \label{t13} if $\bar \eta$ is a ladder system on $S$, $\langle A_n: \ n \in \omega \rangle$  is a 
	$\bar \eta$-rank (Definition $\ref{Arank}$), then for each $\bar \eta$-history $\bar \delta = (\delta_i,\dots,\delta_{i+\ell})$ necessarily $\delta_{i+k} \in A_{j+k}$ for some $j \in \omega$ for each $k \leq \ell$.
	\stepcounter{pmocou}
\end{enumerate}

Moreover, the following also holds.
\begin{claim}\label{zartw*sag}
	Assume that
	\begin{enumerate}[label = $(\Alph*)$, ref =$(\Alph*)$]
		\item $\bar \eta$ is a ladder system on $S$, $\langle A_n: \ n \in \omega \rangle$  is a 
		$\bar \eta$-rank (Definition $\ref{Arank}$),

		\item  $\gp = (S, \bar \eta, \bar v, \Psi)$ is a very special $S$-uniformization problem that is nice, 
		\item 	 $\bar w = \langle w_\alpha: \ \alpha < \omega_1 \rangle$ is a $\gp$-pressed down system,
	\end{enumerate}
	 \then
	\begin{enumerate}
		\item for any $\delta < \omega_1$ the set
		$w_{\delta}$ is $\gp$-closed,
		\item and for any maximal $\bar \eta$-history $\bar \delta = \langle \delta_i, \delta_{i+1}, \dots, \delta_{i+\ell} \rangle$ and for any choice of $t_{i+k} \in \omega$ ($k < \ell$) the set	
		$$w^* = w_{\delta_\ell} \cup \cup\{[\delta_k, t_{k}] : k \in [i,i+\ell)  \}$$
		is $\gp$-closed.
	\end{enumerate} 
		
\end{claim} 
\begin{PROOF}{Claim \ref{zartw*sag}}
	The fact that $w_\delta$ is $\gp$-closed is just part of the definition of being $\gp$-pressed down.
	Second, use that $\ran (\eta_\delta) \subseteq \Omega$ (by \ref{limr} as $\gp$ is nice), while for a fixed history $\bar \delta = \langle \delta_i, \delta_{i+1}, \dots, \delta_{i+\ell} \rangle$, $v_{\delta_{i+k}} \subseteq w_{\delta_{i+k}} \subseteq w_{\delta_{i+\ell}}$,
	so
	\[ v_{\delta_{i+k+1}} \cup \{\delta_{i+k+1} \} \subseteq w^* \cap \delta_{i+k} \subseteq \eta_{\delta_{i+k}}(n_{i+k+1}+1) \] 
	\[ \text{ (where } \delta_{i+k+1}= \eta_{\delta_{i+k}}(n_{i+k+1}))  \]
	
	 implies that $w^* = w_{\delta_\ell} \cup \cup\{[\delta_k, t_{k}] : k \in [i,i+\ell)  \}$ is indeed $\gp$-closed.
\end{PROOF}

\begin{definition} \label{w*}
	Let $\bar \eta$ be a ladder system on the stationary set $S$, assume that $\bar \delta = (\delta_i,\dots,\delta_{i+\ell})$ is an $\bar \eta$-history, and $\langle h_\alpha: \ \alpha < \omega_1 \rangle$ is permitted for $\bar \eta$, assume that $\langle w_\alpha: \ \alpha < \omega_1 \rangle$ is given.
	Then we let
	 $$w^*_{\bar \delta} = w^*_{\bar \delta}(\bar \eta, \bar h, \bar w) = w_{\delta_{i+\ell}} \cup (\cup\{[\delta_k, \delta_k+ h_{\delta_{k+1}}] : k \in [i,i+\ell)  \}).$$  	
\end{definition}

\begin{lemma} \label{ultuni}
	Assume that 
	\begin{enumerate}[label = $(\Alph*)$, ref = $(\Alph*)$]
		\item $\bar \eta$ is a ladder system on the stationary set $S$ (with $\aleph_0$-uniformization on $\bar \eta$), admitting the rank $\bar A$, and $\langle h_\alpha: \ \alpha < \omega_1 \rangle$ is permitted.
		\item 	$\gp = (S, \bar \eta, \bar v, \Psi)$ is a very special $S$-uniformization problem that is nice, $\bar w = \langle w_\alpha: \ \alpha <\omega_1 \rangle$ forms a $\gp$-pressed down system.
	\end{enumerate}
	Then there exists $\langle \psi_\alpha: \ \alpha < \omega_1 \rangle$, $\langle \nu^*_\alpha: \ \alpha < \omega_1\rangle$, $\langle h^*_\alpha: \ \alpha < \omega_1 \rangle$, $\langle p_\alpha: \ \alpha < \omega_1 \rangle$, $\bar m = \langle m_\alpha: \ \alpha \in S \rangle$, $\langle \nu^{**}_\alpha: \ \alpha < \omega_1\rangle$ such that
	letting $\bar \eta' = \bar \eta - \bar m$ the following holds for every $\gamma < \omega_1$:
	\begin{enumerate}[label = $(\blacktriangle)_{\arabic*}$, ref  = $(\blacktriangle)_{\arabic*}$]
		\item \label{egy} if $\gamma = \eta'_\delta(n)$ for some $\delta \in S$, then $\psi_\gamma \in \bigcup_{k \in \omega} (\Upsilon^+)^\gp_k$ (from Claim $\ref{Upsn}$),
		\item \label{ket}  for  every long $\bar \eta' = \bar \eta - \bar m$-history $\bar \delta = (\delta_0, \delta_1,\dots,\delta_{\ell} = \gamma)$, the sequence $h^*_\delta \in \ ^{\omega>}\omega$ satisfies
		$$h^*_\delta = \langle h_{\delta_1}, \dots, h_{\delta_\ell} = h_{\gamma} \rangle,$$
		and if $|w_{\bar \delta}^*| = k$, then
		$$  \psi_\gamma = \{ g \circ \text{OP}_{ w^*_{\bar \delta}, k}: \ g \in (\Psi^+)^{\gp}(w^*_{\bar \delta}) \}  \in (\Upsilon^+_k)^\gp$$
		(here we insist on $(\Psi^+)^{\gp}$ for technical reasons, and do not write $(\Psi^+)^{\gp \setminus \bar m}$  instead, which would  not be the same, as $w^*_{\bar \delta}$ may not be $\gp - \bar m$-closed),
		
		\item \label{har} for each $\delta \in S$ with $\eta'_\delta(n) = \gamma$ for some $n$: $\nu^*_{\gamma} = \eta'_\delta \restriction n \in \ ^n (\eta'_\delta(n))$.
			\item \label{kett} $\nu^{**}_\gamma = \langle \nu^{**}_{\gamma,k}: k < \ell \rangle \in \ ^{<\omega}(^{<\omega}\gamma)$, where for 	every long $\bar \eta' = \bar \eta - \bar m$-history $\bar \delta = (\delta_0, \delta_1,\dots,\delta_j = \gamma)$ we have $j = \ell$ (i.e.\ 
			$\lh(\bar \delta) = \ell + 1 = \lh(\nu^{**}_\gamma)+1$), and 
			 $$(\forall k < \ell): \ \nu^{**}_{\gamma,k} = \nu^*_{\delta_k},$$ 
			 (so if $\delta_k = \eta_{\delta_{k-1}}(n)$, $k <\ell$ then $\nu^{**}_{\gamma,k} = \eta'_{\delta_{k-1}} \restriction n$),
		\item \label{dont} $p_\gamma \in \ ^\ell\{0,1\}$, where for every long $\bar \eta'$-history $\bar \delta = (\delta_0, \delta_1,\dots,\delta_{j} = \gamma)$ we have $j = \ell$ (i.e.\ $p_\gamma = \langle p_\gamma(0), p_\gamma(1), \ldots, p_\gamma(j-1) \rangle$, so $\lh(\bar \delta) = \lh(p_\gamma)+1$), moreover for $0<k < \ell$
		 \[ p_{\gamma} (k) = 1 \ \iff \ (\exists \delta' \neq \delta''): \ \delta_k \in \ran(\eta'_{\delta'}) \cap \ran(\eta'_{\delta''}).\]
	\end{enumerate}
	\end{lemma}
\begin{PROOF}{Lemma \ref{ultuni}}
	We will proceed by induction on $i < \omega$, and define 
	\begin{enumerate}[label = (\greek*), ref = (\greek*)]
		\item \label{i0} $\langle m_\delta: \ \delta \in (S \cap \bigcup_{j<i} A_j) \rangle$, 
		\item \label{i1} $\langle \psi_\alpha: \ \alpha \in \bigcup_{j\leq i} A_j \rangle$, $\langle \nu^*_\alpha: \ \alpha \in \bigcup_{j\leq i} A_j  \rangle$, $\langle h^*_\alpha: \ \alpha \in \bigcup_{j\leq i} A_j  \rangle$, $\langle p_\alpha: \ \alpha  \in \bigcup_{j\leq i} A_j  \rangle$.
	\end{enumerate}
		Recall that for each $\varp \in A_i$ if $\varp = \eta_\delta(n)$ then $\delta \in A_{i-1}$ by Definition $\ref{Arank}$, therefore
		if $\varp \in A_0$, then $\langle \varp \rangle$ is a long history, 
		let $h^*_{\varp} = \nu^*_{\varp} = \nu^{**} = \langle \rangle$ (the empty sequence)
		and if $k = |w^*_{\langle \varp \rangle}|$, then simply let
		 $$\psi_{\varp} =   \{ g \circ \text{OP}_{w^*_{\langle \varp \rangle}, k}: \ g \in (\Psi^+)^\gp(w^*_{\langle \varp \rangle})\}. $$
		 
			Suppose that $i \geq 0$ is fixed, and the objects in \ref{i0}, \ref{i1} are already defined, so if $\gamma \in \bigcup_{j \leq i} A_j$, then $h^*_\gamma$, $\psi_\gamma$, $\nu^*_\gamma$, $\nu^{**}_\gamma$, $p_\gamma$ satisfy clauses \ref{egy}-\ref{dont} (recall that under $\eta_\delta'(n)$ we mean $ \eta_\delta(n+m_\delta)$, note that $m_\delta$ is  defined if $\delta \in \bigcup_{j <i} A_j$).
			
			First we choose $\langle \nu^{**}_\gamma, \psi_\gamma, h^*_\gamma, p_\gamma: \ \gamma \in A_{i+1} \cap (\bigcup_{\delta \in (S \cap A_i)} \ran(\eta_\delta)) \rangle$, then we are able to define $\langle \nu^*_\gamma: \ \gamma \in A_{i+1} \rangle$, $\langle m_\delta: \ \delta \in A_i \rangle$, which  determines $\bar \eta' \restriction A_i = (\bar \eta - \bar m) \restriction A_i)$ (note that there may be some
			 $$\gamma \in A_{i+1} \cap \left(\bigcup_{\delta \in (S \cap A_i)} \ran(\eta_\delta)\right) \setminus \left. \left(\bigcup_{\delta \in (S \cap A_i)} \ran(\eta'_\delta)\right) \right).$$ Finally we can choose 
			$\langle \nu^{**}_\gamma, \psi_\gamma, h^*_\gamma, p_\gamma: \ \gamma \in A_{i+1} \setminus (\bigcup_{\delta \in (S \cap A_i)} \ran(\eta'_\delta))\rangle$, similarly to the case of $\gamma \in A_0$.
			
			Before the induction step we have to observe the following corollaries of the induction hypothesis \ref{ket}:
			\newcounter{teklecou} \setcounter{teklecou}{0}
			\begin{enumerate}[label = $(\bigstar_{\arabic*})$, ref = $(\bigstar_{\arabic*})$]
				\item \label{abst} if $\varp \in \bigcup_{j \leq i} A_j$, then  for the long $\bar \eta'$-histories $\bar \delta = \langle \delta_0, \delta_1, \dots, \delta_\ell = \varp \rangle$, $\bar \delta' = \langle \delta'_0, \delta'_1, \dots, \delta'_{\ell'}  = \varp \rangle$ we always have  
				\[ h^*_\varp = \langle h_{\delta_0}, h_{\delta_1}, \dots, h_{\delta_\ell} (= h_\varp )\rangle =  \langle h_{\delta'_0}, h_{\delta'_1}, \dots, h_{\delta'_{\ell'}} (= h_\varp) \rangle,\]
				in particular
				$$\ell = \ell',$$
				
				\stepcounter{teklecou}
			\end{enumerate}	 
			moreover, 
			\begin{enumerate}[label = $(\bigstar_{\arabic*})$, ref = $(\bigstar_{\arabic*})$]
				\setcounter{enumi}{\value{teklecou}}
				\item \label{abs2}if $\varp \in \bigcup_{j \leq i} A_j$, $\bar \delta$, $\bar \delta'$ are as above in \ref{abst}, $n \in \omega$, then 
				$$w^*_{\bar \delta \tieconcat \langle \eta_\varp(n) \rangle} \cap \varp = w^*_{\bar \delta' \tieconcat \langle \eta_\varp(n) \rangle} \cap \varp $$
				as $w_{\delta_i} \subseteq w_{\eta_{\delta_i(n)}}$ by Definition \ref{wdf}
				(even 
				\begin{equation} \label{varpalatt} w^*_{\bar \delta \tieconcat \langle \eta_\varp(n) \rangle} \cap (\varp+\omega) = w^*_{\bar \delta' \tieconcat \langle \eta_\varp(n) \rangle} \cap (\varp+\omega)  \end{equation}
				is true).
				Moreover,
				\[ w^*_{\bar \delta \tieconcat \langle \eta_\varp(n) \rangle} \setminus w^*_{\bar \delta} =  w^*_{\bar \delta' \tieconcat \langle \eta_\varp(n) \rangle} \setminus w^*_{\bar \delta'} = \]
				\[ = [\varp+1, \varp + h_{\eta_\varp(n)}] \cup \{ \eta_\varp(n)\} \cup (w_{\eta_\varp(n)} \setminus w_\varp) \]
				(since $\delta_0 > \delta_1 > \dots > \delta_\ell = \varp> \eta_\varp(n)$, and
				$\delta'_0 > \delta'_1 > \dots > \delta'_\ell = \varp> \eta_\varp(n)$, and $w_\alpha \subseteq \alpha$ by \ref{w1} in Definition \ref{wdf}),
				in particular, 
				$$(w^*_{\bar \delta \tieconcat \langle \eta_\varp(n) \rangle} \setminus w^*_{\bar \delta}) \cap \eta_\varp(n) = (w^*_{\bar \delta' \tieconcat \langle \eta_\varp(n) \rangle} \setminus w^*_{\bar \delta}) \cap \eta_\varp(n)  = w_{\eta_\varp(n)} \setminus w_\varp, $$
				\stepcounter{teklecou}
			\end{enumerate}	 
			which easily implies that
			\begin{enumerate}[label = $(\bigstar_{\arabic*})$, ref = $(\bigstar_{\arabic*})$]
				\setcounter{enumi}{\value{teklecou}}
				\item \label{lekep} if $\varp \in \bigcup_{j \leq i} A_j$, $\bar \delta$, $\bar \delta'$ are as above in \ref{abst}, $n \in \omega$, then for any $k \leq \ell$
				\[ (\text{OP}_{w^*_{\bar \delta \tieconcat \langle \eta_\varp(n) \rangle}, w_{\bar \delta' \tieconcat \langle \eta_\varp(n) \rangle}}) ``[\delta'_k, \delta'_k+ h_{\delta'_{k+1}}] = [\delta_k, \delta_k+ h_{\delta_{k+1}}],\]
				and of course
				\[ \text{OP}_{w^*_{\bar \delta \tieconcat \langle \eta_\varp(n) \rangle},w_{\bar \delta' \tieconcat \langle \eta_\varp(n) \rangle}} (\alpha) = \alpha, \ \text{ if } \alpha \leq \varp, \]
				hence
				\begin{equation} \label{delta'delta}
					\text{OP}_{w^*_{\bar \delta \tieconcat \langle \eta_\varp(n) \rangle}, w^*_{\bar \delta' \tieconcat \langle \eta_\varp(n) \rangle}} \restriction  w^*_{\bar \delta'} = \text{OP}_{w^*_{\bar \delta}, w^*_{\bar \delta'}}.
				\end{equation}
				\stepcounter{teklecou}
			\end{enumerate}

			Now fix $\varp \in A_i$,			
				define the functions $c_\varp^{h^*}, c^{\psi}_\varp, c^p_\varp$ with  domain $\ran(\eta_\varp)$ as follows.
			\begin{enumerate}[label = $\bullet_{\alph*}$, ref = $\bullet_{\alph*}$]
				\item $c_{\varp}^{h^*} (\eta_{\varp}(n)) =  h^*_{\varp} \tieconcat \langle h_{\eta_\varp(n)} \rangle \in \ ^{<\omega}\omega$, 
				\item \label{nu**} $c_\varp^{\nu^{**}}(\eta_\varp(n)) = \nu_\varp^{**} \tieconcat \langle \nu^*_\varp \rangle$,
				\item for $c_\varp^{\psi}: \ran(\eta_\varp) \to \bigcup_{k < \omega} (\Upsilon^+_k)^\gp$ fix a long $\bar \eta' \restriction (S \cap \bigcup_{j<i} A_j)  =  (\bar \eta - \bar m)\restriction (S \cap \bigcup_{j<i} A_j)$-history $\bar \delta = \langle \delta_0, \delta_1, \dots, \delta_\ell = \varp \rangle$, and consider 
				$$w^*_{\bar \delta \tieconcat \langle \eta_\varp(n) \rangle} = w_{\eta_\varp(n)} \cup (\cup\{[\delta_k, \delta_k+ h_{\delta_{k+1}}] : k \in [0,\ell)  \} \cup [\varp, \varp + h_{\eta_\varp(n)}].$$
				First note that by our inductive hypothesis
				$$ h^*_\varp = \langle h_{\delta_1}, h_{\delta_2}, \dots, h_{\delta_\ell} (= \varp) \rangle,$$
				so (since $w_\varp \subseteq w_{\eta_\varp(n)} \subseteq \eta_\varp(n)$) $k= |w^*_{\bar \delta \tieconcat \langle \eta_\varp(n) \rangle}|$ does  not depend on the particular choice of $\bar \delta$, but we also shall argue later for
				\begin{equation} \label{psilo} c_\varp^{\psi}(\eta_\varp(n)) := \{ g \circ \text{OP}_{w^*_{\bar \delta \tieconcat \langle \eta_\varp(n) \rangle}, k}: \ g \in (\Psi^+)^\gp(w^*_{\bar \delta \tieconcat \langle \eta_\varp(n) \rangle}) \} \in \Upsilon^\gp_k \end{equation}
				(where $k = |w^*_{\bar \delta \tieconcat \langle \eta_\varp(n) \rangle}|$)
				 being independent of $\bar \delta$ (and only depending on $\varp$).
				\item Finally let $c_\varp^{p}(\eta_\varp(n)) =  p_\varp \tieconcat \langle j \rangle $, where $j =1$, if there are $\delta' \neq \delta''$ with $\varp \in \ran(\eta'_{\delta'}) \cap \ran(\eta'_{\delta''})$ (i.e.\ $\varp = \eta_{\delta'}(n') = \eta_{\delta''}(n'')$ for some $n' \geq m_{\delta'}$, $n'' \geq m_{\delta''}$), otherwise $j = 0$.
					
			\end{enumerate}

			\begin{claim} \label{izomt}
				The value $c_\varp^{\psi}(\eta_\varp(n))$ (defined in $\eqref{psilo}$ using the long history $\bar \delta = \langle \delta_0, \delta_1, \dots, \delta_\ell = \varp \rangle$) does not depend on the particular $\bar \delta$.
			\end{claim}
			\begin{PROOF}{Claim \ref{izomt}}
				Fix $\varp \in S \cap A_i$, $n \in \omega$, and 
				long $\eta'$-histories
				$\bar \xi = \langle \xi_0, \xi_1, \dots, \xi_\ell = \varp \rangle$, 
				$\bar \zeta = \langle \zeta_0, \zeta_1, \dots, \zeta_\ell = \varp \rangle$.
				
				We have to use the induction hypothesis on $\langle \psi_\delta: \ \delta \in \bigcup_{j \leq i} A_j \rangle$ and recall how $(\Psi^+)^\gp$ is defined from $(\Psi)^\gp$ in Claim $\ref{Upsn}$.
				Let 
				$$x_\xi := w^*_{\bar \xi \tieconcat \langle \eta_\varp(n) \rangle} = w_{\eta_\varp(n)} \cup (\cup\{[\xi_k, \xi_k+ h_{\xi_{k+1}}] : k \in [0,\ell)  \} \cup [\varp, \varp + h_{\eta_\varp(n)}],$$ 
				and 
				$$x_\zeta :=w^*_{\bar \zeta \tieconcat \langle \eta_\varp(n) \rangle} = w_{\eta_\varp(n)} \cup (\cup\{[\zeta_k, \zeta_k+ h_{\zeta_{k+1}}] : k \in [0,\ell)  \} \cup [\varp, \varp + h_{\eta_\varp(n)}].$$
				It is clear from the way we defined $w^*_{\bar \delta}$'s (Definition $\ref{w*}$) and the inductive assumption on $h^*_\varp$ (or just on $\psi_\varp$) that
				 $$ \left|w^*_{\bar \zeta \tieconcat \langle \eta_\varp(n) \rangle}\right| = \left|w^*_{\bar \xi \tieconcat \langle \eta_\varp(n) \rangle}\right|, $$ let $k$ denote this number.
				So suppose that the function $f_0$ satisfies 
				\begin{equation} \label{g0}	f_0=  g_0 \circ \text{OP}_{x_\xi, k} \ \text{ for some } g_0 \in (\Psi^+)^\gp(x_\xi), \end{equation}
					but 
					\begin{equation} \label{g1} g_1 = f_0 \circ \text{OP}_{k,x_\zeta}   \notin (\Psi^+)^\gp(x_\zeta)\end{equation}
					(so
					\begin{equation} \label{g0g1} g_0 = g_1 \circ \text{OP}_{x_\zeta,x_\xi}). \end{equation}
				This means (recalling \ref{veryfa1} from Fact $\ref{veryfa}$) that for some $b \in \cB^\gp$ and $b \subseteq x_\zeta$: 
				$$g_1 \restriction b \notin \Psi(b).$$
				Now again by  Fact $\ref{veryfa}$ this $b$ is of the form 
				$b^\delta_{n,s}= \bigcup_{j \leq n} (v_{\eta_\delta}(j) \cup \{ \eta_\delta(j)\}) \cup [\delta, \delta+s]$ for some $\delta \in S$, $n,s \in \omega$.
				
				 Now assume first that $\delta \leq \zeta_\ell = \varp$, but then $\eqref{varpalatt}$ in \ref{abst} implies 
				 that $b \subseteq x_\xi \cap x_\zeta$, and $\text{OP}_{x_\zeta,x_\xi} \restriction b = \text{id}_b$, so $g_0 \restriction b = g_1 \restriction b \notin \Psi^\gp(b)$, hence $g_0 \notin  (\Psi^+)^\gp(x_\xi)$ contradicting $\eqref{g0}$.
				
				Therefore we can assume that $\delta > \varp$			
				 (so $\delta = \zeta_j$ for some $j < \ell$), then  w.l.o.g.\ by \ref{pdc} from Definition $\ref{pprobd}$ we can assume that $n$, $s$ are maximal such that $b = b^\delta_{n,s} \subseteq x_\zeta$, so
				$$b = v_{\zeta_{j+1}} \cup \{ \zeta_{j+1}\} \cup [\zeta_j, \zeta_j + h_{\zeta_{j+1}}].$$
				But note that then 
				$$b \subseteq w^*_{\bar \zeta}$$ 
				(since $\{ \zeta_{j+1}\} \cup [\zeta_j, \zeta_j + h_{\zeta_{j+1}}] \cup w_{\zeta_\ell} \subseteq w^*_{\bar \zeta}$ by Definition $\ref{w*}$ and $w_{\zeta_\ell} \supseteq w_{\zeta_{j+1}} \supseteq v_{\zeta_{j+1}}$ by our demands on $\langle w_\alpha: \ \alpha < \omega_1 \rangle$), so
				$$g_1 \restriction w^*_{\bar \zeta} \notin (\Psi^\gp)^+(w^*_{\bar \zeta}).$$
				Now by the inductive assumption on $\psi_\varp$ in \ref{ket}
				necessarily 
				$$g^*:= (g_1 \restriction w^*_{\zeta}) \circ \text{OP}_{w^*_{\bar \zeta}, w^*_{\bar \xi}} \notin (\Psi^\gp)^+(w^*_{\bar \xi}),$$
				but as $\text{OP}_{x_\zeta,x_\xi} \restriction w^*_{\bar \xi}  = \text{OP}_{w^*_{\bar \zeta}, w^*_{\bar \xi}}$ by $\eqref{delta'delta}$ in \ref{lekep},
				 $$g^* = (g_1 \circ \text{OP}_{x_\zeta,x_\xi}) \restriction w^*_{\bar \xi} \notin (\Psi^\gp)^+(w^*_{\bar \xi}), $$
				 and by $\eqref{g0g1}$
				 $$		g_0  \restriction w^*_{\bar \xi} \notin (\Psi^\gp)^+(w^*_{\bar \xi})$$
				 implying $g_0 \notin 	(\Psi^\gp)^+(x_{\xi})$, contradicting $\eqref{g0}$.	
				
			\end{PROOF}
			Now having the claim proven we already saw that $c_\varp(\eta_\varp(n))$ does not depend on the long $\eta'$-history ending with $\varp$, once we manage to uniformize the system it will only depend on $\eta_\varp(n)$, not on $\varp$.
			
			It is easy to see that  for each fixed $\alpha$ the sets $\{c_\varp^{h^*} (\alpha): \ \alpha \in \dom(c_\varp^{h^*}) \}$, $\{c_\varp^{\psi} (\alpha): \ \alpha \in \dom(c_\varp^{\psi}) \}$ and $\{c_\varp^{p} (\alpha): \ \alpha \in \dom(c_\varp^{p}) \}$ are all countable, while 
			\begin{equation} \label{equ}  \{c_\varp^{\nu^{**}}(\alpha): \ \alpha \in \ran (\dom(\eta_\varp))\} \subseteq \ ^{<\omega}(^{<\omega}\alpha), \end{equation} which can be seen by the following: $c_\varp^{\nu^{**}}(\alpha) = \nu^{**}_\varp \tieconcat \langle \nu^*_\varp \rangle$ by its very definition in \ref{nu**},
			and by our inductive assumptions \ref{ket} for each long $\eta'$-history $\bar \delta = \langle \delta_0, \delta_1, \dots, \delta_\ell = \varp \rangle$ the sequence $c_\varp^{\nu^{**}}(\alpha)$ is equal to  $\langle \nu^*_{\delta_0}, \nu^*_{\delta_1}, \dots, \nu^*_{\delta_{\ell-1}} \rangle \tieconcat \langle \nu^*_{\delta_\ell} \rangle$. By \ref{kett} this latter is equal to 
			$$\langle \langle\rangle, (\eta'_{\delta_0} \restriction n_1), (\eta'_{\delta_1} \restriction n_2), \dots, (\eta'_{\delta_{\ell-1}} \restriction n_\ell) \rangle$$
			(where $\eta'_{\delta_k}(n_{k+1}) = \delta_{k+1}$, and $\nu^*_{\delta_0} = \langle \rangle$), so
			$$c_\varp^{\nu^{**}}(\alpha) \in \ ^{<\omega}\delta_1 \times \ ^{<\omega}\delta_2 \times \ \dots \times  \ ^{<\omega}\delta_\ell.$$
			But as $ \bar \delta \tieconcat \langle \delta_{\ell+1}:= \alpha \rangle$ is an $\eta$-history, so recalling \ref{fiok} if $k \leq \ell$ and $\beta \in \ran(\eta_{\delta_k}) \cap \delta_{k+1}$, then $\beta \in \alpha = \delta_{\ell+1}$, in particular,
			$$ \text{for each } k < \ell: \ \eta'_{\delta_k} \restriction n_{k+1} \text{ has range }\subseteq\alpha,$$
			we are done with proving $\eqref{equ}$.
			
			 Thus, we can uniformize all these functions, and pick 
			some suitable $d^{h^*}$, $d^{\psi}$, $d^{\nu^{**}}$ and $d^p$ that uniformize the respective mappings, so fix $m^0_\delta: \ \delta \in A_i$ so that $(n \geq m^0_\delta) \to d^{\text{t}}(\eta_\delta(n)) = c^{\text{t}}_\delta(\eta_\delta(n))$ for any $\text{t} \in \{h^*, \psi, \nu^{**}, p \}$.

			Now for obtaining $\nu^*$ apply Claim $\ref{nucl}$ to the ladder system $\langle \eta^0_\delta: \ \delta \in A_i \cap S \rangle$, $\eta^0_\delta(n) = \eta_\delta(n+m^0_\delta)$ giving us 
			\begin{enumerate}[label = $\bigodot$, ref = $\bigodot$ ]
				\item $\langle m^1_\delta: \ \delta \in A_i \cap S \rangle$, and
				$\langle \nu_\alpha: \ \alpha < \omega_1 \rangle$ with
				$$ \nu_{\eta^0_\delta(n+m^1_\delta)} = \nu_{\eta_\delta(n+m^0_\delta + m^1_\delta)} = $$ $$ = \eta^0 \restriction [m_\delta^1, m_\delta^1+n) = \eta \restriction [m_\delta^0 + m_\delta^1 , m_\delta^0 +  m_\delta^1+n). $$ 
			\end{enumerate}
		For each $\delta \in A_i \cap S$ we let $m_\delta = m^0_\delta + m^1_\delta$, so we defined the ladder system 
		$$ \bar \eta' \restriction (A_i \cap S), \ \  \eta'_\delta(n) = \eta_\delta(m^0_\delta+m^1_\delta+n) \ \ ( n \in \omega). $$
		
			Finally, for $\alpha \in A_{i+1}$ let
			$$ \begin{array}{rl} \nu^*_\alpha = & \left \{ \begin{array}{ll} \nu_\alpha, & \text{ if } \alpha \in \ran(\eta'_\delta) \text{ for some } \delta \in A_i \\
				\langle\rangle, & \text{ otherwise}, \end{array} \right. \\
			& \\
			
			\nu^{**}_\alpha = & \left \{ \begin{array}{ll} d^{\nu^{**}}(\alpha), & \text{ if } \alpha \in \ran(\eta'_\delta) \text{ for some } \delta \in A_i \\
				\langle\rangle, & \text{ otherwise}, \end{array} \right.  \\
			& \\
			
				\psi_\alpha = & \left \{ \begin{array}{ll} d^{\psi}(\alpha), & \text{ if } \alpha \in \ran(\eta'_\delta) \text{ for some } \delta \in A_i \\
				\langle\rangle, & \text{ otherwise}, \end{array} \right.  \\
			& \\
			p_\alpha = & \left \{ \begin{array}{ll} d^{p}(\alpha), & \text{ if } \alpha \in \ran(\eta'_\delta) \text{ for some } \delta \in A_i \\
				\langle\rangle, & \text{ otherwise}, \end{array} \right.  \\
			& \\
				h^*_\alpha = & \left \{ \begin{array}{ll} d^{h^*}(\alpha), & \text{ if } \alpha \in \ran(\eta'_\delta) \text{ for some } \delta \in A_i \\
				\langle\rangle, & \text{ otherwise}. \end{array} \right.  			
			
		\end{array} $$
			It is straightforward to check  all our demands in \ref{egy}-\ref{dont} for each $\delta \in S \cap A_i$, $\gamma \in A_{i+1}$.

\end{PROOF}

\begin{lemma} \label{ultuniz}
	Assume that 
	\begin{enumerate}[label = $(\Alph*)$, ref = $(\Alph*)$]
		\item $\bar \eta$ is a ladder system on the stationary set $S$ (with $\aleph_0$-uniformization on $\bar \eta$), admitting the rank $\bar A$, and $\langle h_\alpha: \ \alpha < \omega_1 \rangle$ is permitted.
		\item 	$\gq = (S, \bar \eta, \bar v, \Psi)$ is a nice very special $S$-uniformization problem, $\bar w = \langle w_\alpha: \ \alpha <\omega_1 \rangle$ is a $\gq$-pressed down system.
	\end{enumerate}
	Moreover, suppose that there exists $\bar \psi$, $\bar \nu^*$, $\bar p$, $\bar \nu^{**}$, $\bar h^*$ and $\bar m$ as in Lemma $\ref{ultuni}$ (and a derived ladder system $\bar \eta' = \bar \eta - \bar m$) 
	
	Fix $\varp < \omega_1$, $\bar \delta = \langle \delta_0, \dots, \delta_\ell = \varp \rangle$, $\bar \delta' = \langle \delta'_0, \dots, \delta'_\ell = \varp \rangle$ are  long $\bar \eta'$-histories, and suppose that $\theta < \varp$, $\bar \xi =  \langle \xi_0, \dots, \xi_k = \theta \rangle$ is a long $\bar \eta'$-history.
	Then
	\begin{enumerate}[label = $(\curlyvee)_{\arabic*}$, ref = $(\curlyvee)_{\arabic*}$]
		\item \label{hi0} $\{\delta_i: \ i \leq \ell\} \cap \{ \delta'_i: \ i \leq \ell\} = \{ \delta_i: \ i \geq j^\bullet \}$ for some $j^\bullet \leq \ell$,
		\item \label{hi1} $\{\delta_i: \ i \leq \ell\} \cap \{ \xi_i: \ i \leq k\}$ is either empty, or is of the form $\{ \delta_i: \ j_0 \leq i \leq j_1 \}$ for some $0 \leq j_0 \leq j_1 \leq \ell$, moreover,
		$$\{ \delta_i: \ j_0 \leq i \leq j_1 \} = \{ \xi_i: \ j_0 \leq i \leq j_1 \},$$
		\item \label{hi2} if  $j_0 \leq j_1$  are as in \ref{hi1} (so $\bar \delta$ and $\bar \xi$ intersect each other), then there exists a long $\bar \eta'$-history $\bar \xi' = \langle \xi'_0, \xi'_1, \dots, \xi'_{k}\rangle$ with 
		\begin{enumerate}
			\item $\xi'_{k} = \theta$, and
			\item $ \{\delta'_i: \ i \leq \ell\} \cap \{ \xi'_i: \ i \leq k\} = \{ \delta'_i: \ j_0 \leq i \leq j_1 \}$,
			\item moreover,
				$$ (\text{OP}_{w^*_{\bar \delta'}, w^*_{\bar \delta}})``w^*_{\bar \delta} \cap w^*_{\bar \xi} =  w^*_{\bar \delta'} \cap w^*_{\bar \xi'},$$
				and
					$$ (\text{OP}_{w^*_{\bar \xi'}, w^*_{\bar \xi}})``w^*_{\bar \delta} \cap w^*_{\bar \xi} =  w^*_{\bar \delta'} \cap w^*_{\bar \xi'},$$
			
		\end{enumerate}
			\item \label{hi3} there exists a long $\bar \eta'$-history $\bar \xi = \langle \xi_0, \xi_1, \dots, \xi_{k^*} \rangle$ such that $\xi_{k^*} < \varp$, $\varp \notin \{\xi_j: \ j \leq k^*\}$, and
		whenever  $\bar \vartheta = \langle \vartheta_0, \vartheta_1, \dots, \vartheta_{m} \rangle$ is a long $\bar \eta'$-history with $\vartheta_m < \varp$, $\varp = \delta_\ell \notin \{ \vartheta_j: \ j\leq m\}$, then for all $i<\ell$ $\delta_i \in \{\vartheta_j: \ j\leq m\}$ implies $\delta_i \in \{\xi_0, \xi_1, \dots, \xi_{k^*}\}$
	and
		$$ w^*_{\bar \vartheta} \cap [\delta_i, \delta_i+\omega) \subseteq w^*_{\bar \xi} \cap [\delta_i, \delta_i+\omega).$$
		
	\end{enumerate}
\end{lemma}
\begin{PROOF}{Lemma \ref{ultuniz}}
	First, assume on the contrary, that for  the long  $\bar \eta'$-histories $\bar \delta$, $\bar \delta'$ ($\delta_\ell = \delta'_\ell$) for some $i<\ell$ $\delta_{i} = \delta'_{i}$, but $\delta_{i+1} \neq \delta'_{i+1}$ (so necessarily $i+1 < \ell$). Then on the one hand  $\delta_{i+1}, \delta'_{i+1} \in \ran(\eta'_{\delta_{i}})$, so by \ref{har} $\nu^*_{\delta_{i+1}} \neq \nu^*_{\delta'_{i+1}}$ as one is a proper initial segment of the other.
	On the other hand by \ref{kett}
	\[ \nu^*_{\delta_{i+1}} = \nu^{**}_{\delta_\ell, i+1} = \nu^{**}_{\delta'_{\ell}, i+1} =  \nu^{*}_{\delta'_{i+1}}, \]
	a contradiction, this proves \ref{hi0}.

	For \ref{hi1} let $\delta_i = \xi_j$, but $\delta_{i-1} \neq \xi_{j-1}$ (so $\delta_{i-1} \notin \ran(\bar \xi)$ by \ref{t13}). Then apply \ref{hi0} to $\langle \delta_0, \delta_1, \dots, \delta_i \rangle$, $\langle \xi_0, \xi_1, \dots, \xi_j \rangle$. Finally, follows from \ref{ket} or \ref{kett} that $i=j$.
			
	To prove \ref{hi2} let $j^\bullet$ be given by \ref{hi0} let $j_0, j_1$ be given by \ref{hi1} for $\bar \delta$, $\bar \xi$, so that
	\begin{equation} \{ \delta_i: \ j_0 \leq i \leq j_1\} = \{\xi_i: \ j_0 \leq i \leq j_1\}. \end{equation}
	 Define $\xi'_{i} = \delta'_{i}$ for $i \in [j_0,j_1]$, and $\xi'_i = \xi_i$ for $i>j_1$ (and $i \leq k$).
	Before completing the definition of $\bar \xi'$ we check that this segment is a $\eta'$-history, which is clear once we have showed 
	\begin{equation} \label{oncwe} \xi_{j_1+1} \in \ran(\eta'_{\delta'_{j_1}}). \end{equation} Of course we can assume that $\delta'_{j_1} \neq \delta_{j_1}$ (as $\delta_{j_1} = \xi_{j_1}$ and so \eqref{oncwe} would follow), from which necessarily $j_1 < \ell$.
	
	We will need the following two observations.
	\begin{enumerate}[label = $(\boxdot)_1$, ref = $(\boxdot_1)$ ]
		\item \label{uj0} $\nu^*_{\delta_{j_1+1}} = \nu^*_{\delta'_{j_1+1}}$, either because $j_1 +1 = \ell$ (and then $\delta_{j_1+1} = \delta'_{j_1+1} = \varp$), or else if $j_1 +1 < \ell$, then
			$\nu^{**}_{\varp, j_1+1} = \nu^*_{\delta_{j_1+1}} = \nu^*_{\delta'_{j_1+1}}$ by \ref{kett}, so
			\[ \nu^*_{\delta'_{j_1+1}}  = \eta'_{\delta'_{j_1}} \restriction n =  \eta'_{\delta_{j_1}} \restriction n \text{ for some } n,\]
			where 
			\begin{equation} \label{deltaj1p0}\eta'_{\delta'_{j_1}}(n) = \delta'_{j_1+1}, \end{equation} and
			\begin{equation} \label{deltaj1p1} \eta'_{\delta_{j_1}}(n) = \delta_{j_1+1}. \end{equation}
	\end{enumerate}
	Moreover, as $\delta_{j_1} = \xi_{j_1}$ we have
	\begin{enumerate}[label = $(\boxdot)_2$, ref = $(\boxdot_2)$ ]
		\item \label{uja} $\nu^*_{\xi_{j_1+1}} = \eta'_{\xi_{j_1}} \restriction n^\bullet = \eta'_{\delta_{j_1}} \restriction n^\bullet$ for some $n^\bullet$   by \ref{ket}, and $\eta'_{\delta_{j_1}}(n^\bullet) = \eta'_{\xi_{j_1}}(n^\bullet) = \xi_{j_1 +1}$.
	\end{enumerate}
	So \ref{uj0} and \ref{uja} together mean that  it suffices to argue that 
	\begin{equation} \label{nbu} n^\bullet < n, \end{equation} because then 
	\begin{equation} \label{font} \eta'_{\delta'_{j_1}}(n^\bullet) = \eta'_{\delta_{j_1}}(n^\bullet) = \xi_{j_1+1}, \end{equation} i.e. \ref{oncwe} holds. As $ \eta'_{\delta_{j_1}}(n) = \delta_{j_1+1}$, and $\eta'_{\xi_{j_1}}(n^\bullet) = \xi_{j_1 +1}$ and $\delta_{j_1+1} \neq \xi_{j_1 +1}$ by our assumptions, so if $n^\bullet > n$, then it follows from \ref{uja} that
	$\eta'_{\xi_{j_1}}(n) =  \eta'_{\delta_{j_1}}(n) = \delta_{j_1+1}$ by $\eqref{deltaj1p1}$, and necessarily $\delta_{j_1+1} < \xi_{j_1 +1} = \eta'_{\xi_{j_1}}(n^\bullet)$. So applying \ref{fiok} to
	$\bar \xi$ we obtain that for each $j \leq k$ $\xi_j > \delta_{j_1+1} \geq \delta_\ell = \varp$ holds, contradicting $\theta < \varp$. 

	Now recalling \ref{hi1} we only have to extend $\langle \xi'_{j_0}, \xi'_{j_0+1}, \dots, \xi'_k \rangle$ to a long $\bar \eta'$-history $\langle \xi'_{0}, \xi'_1 \ldots,  \xi'_k \rangle$, such that $\xi'_{j_0-1} \neq \delta'_{j_0-1}$, So we can assume that $0<j_0$.  Note that we may assume that  $\delta_{j_0} \neq \delta_{\ell} = \varp$, i.e. $j_0 < \ell$, as otherwise letting $\bar \xi' = \bar \xi$ would work. Recall \ref{dont}, so $p_\varp(j_0)$ is defined, and $\delta_{j_0} \in \ran(\eta'_{\delta_{j_0-1}}) \cap \ran(\eta'_{\xi_{j_0-1}})$ implies that $p_\varp(j_0)=1$. Applying this fact to $\bar \delta'$ we obtain that 
	\[ \delta'_{j_0} \in \ran(\eta'_{\delta'_{j_0-1}}) \cap \ran(\eta'_{\gamma}) \text{ for some } \gamma \neq \delta'_{j_0-1}, \]
	let $\xi'_{j_0-1} = \gamma$, and extending it to a long history, we are done. (It follows from \ref{hi0}, that two long $\bar \eta'$-histories $\bar \xi$ and $\bar \xi'$ with the same ending must have the same length.)
	
	For the moreover part recall how we defined the $w^*$'s (Definition $\ref{w*}$):
	$$ w^*_{\bar \delta} = w_{\delta_\ell} \cup (\cup\{[\delta_i, \delta_i+ h_{\delta_{i+1}}] : i < \ell \}),$$
 and note that $\delta_{\ell} = \delta'_{\ell}$, so by \ref{ket} $h_{\delta_i} = h_{\delta'_i}$ for each $i \leq \ell$,
 and
 	$$ w^*_{\bar \xi} = w_{\xi_{k}} \cup (\cup\{[\xi_i, \xi_i+ h_{\xi_{i+1}}] : i < k \}),$$
 	similarly $\xi_{k} = \xi'_{k}$, so by \ref{ket} $h_{\xi_i} = h_{\xi'_i}$ for each $i \leq k$.
 	We recall that
 	$$ \textrm{max}(w_{\xi_k}) < \xi_k < \xi_{k-1} <\ldots < \xi_0,$$
		$$ \textrm{max}(w_{\xi'_k}) < \xi'_k < \xi'_{k-1} < \ldots < \xi'_0,$$
			$$ \textrm{max}(w_{\delta_\ell}) < \delta_\ell < \delta_{\ell-1} < \ldots <\delta_0,$$
					$$ \textrm{max}(w_{\delta'_\ell}) < \delta'_\ell <\delta'_{\ell-1} < \ldots <\delta'_0,$$
					(and $w_{\delta_\ell} = w_{\delta'_\ell}$),
					so one can easily check that the moreover part of \ref{hi2} would follow if we could argue that
					\begin{equation} \label{kelle}(\forall i \leq k): \  [\xi_i, \xi_i+ h_{\xi_{i+1}}] \cap w_{\delta_\ell} = [\xi'_i, \xi'_i+ h_{\xi'_{i+1}}] \cap w_{\delta_\ell}. \end{equation}
					Note that \eqref{kelle} would follow from the assertion
						\begin{equation} \label{kella} (\forall i \leq k): \ \xi_i \in w_{\delta_\ell} \ \iff \ \xi'_i\in w_{\delta_\ell}.  \end{equation}
			For \eqref{kella} note that by definition $\xi_j = \xi'_j$ for $j >j_1$, so we only have to care about $\xi_j$'s, $\xi'_j$'s for $j \leq j_1$.
			  Therefore, as $\eta'_{\delta_{j_1}}(n) = \delta_{j_1+1}$, $ \eta'_{\delta'_{j_1}}(n) = \delta'_{j_1+1}$, $n > n^\bullet$ (\eqref{deltaj1p0}, \eqref{deltaj1p1}, \eqref{nbu}),
			  and  $ \eta'_{\delta'_{j_1}}(n^\bullet) = \eta'_{\delta_{j_1}}(n^\bullet) = \xi_{j_1+1} = \xi'_{j_1+1}$,
			  (by \eqref{font}), 
			  and clearly $$\eta'_{\delta_{j_1}}(n) =  \delta_{j_1} = \xi_{j_1} > \delta_{j_1+1} > \ldots > \delta_{\ell} > \eta'_{\delta_{j_1}}(n^\bullet) = \xi_{j_1+1},$$
			  $$\eta'_{\delta'_{j_1}}(n) =  \delta'_{j_1} = \xi'_{j_1} > \delta'_{j_1+1} > \ldots > \delta'_{\ell} > \eta'_{\delta'_{j_1}}(n^\bullet) = \xi'_{j_1+1}.$$
			  This means that $\xi_{j_1}, \xi'_{j_1} > $max($w_{\delta_\ell}$), and we are done.

	For \ref{hi3} first let $k^\bullet = \ell$, if $\nu^*_{\varp} \neq \langle \rangle$, otherwise letting $k^\bullet< \ell$ to be the largest $j< \ell$ such that $\nu^{**}_{\varp,j} \neq \langle \rangle$ (equivalently, $\nu^{*}_{\delta_j} \neq \langle \rangle$) if such a $j$ exists, and let $k^\bullet=0$ if there is no such $j$.
	We will argue that choosing 
	\begin{equation} \bar \xi = \langle \xi_0, \xi_1, \ldots \rangle = \langle \delta_0, \delta_1, \dots, \delta_{k^\bullet-1}, \nu^*_{\delta_{k^\bullet}}(|\nu^*_{\delta_{k^\bullet}}|-1) \rangle \end{equation} works (or else, if $\nu^*_{\varp} = \langle \rangle$ and so there is no $j< \ell$ with $\nu^{*}_{\delta_j} \neq \langle \rangle$, then for any $\bar \eta'$-history $\bar \vartheta$ with last entry smaller than $\varp$, the set of entries of $\bar \vartheta$  does not meet that of $\bar \delta$). 
	No matter which case we are in we can assume that
	\begin{equation} \forall t > 0: \ \nu^*_{\delta_{k^\bullet+t}} = \langle\rangle.
	\end{equation}
	In the first case note that if $k^\bullet >0$, then $\nu^*_{\delta_{k^\bullet}} = \eta'_{\delta_{k^\bullet-1}} \restriction |\nu^*_{\delta_{k^\bullet}}|$ by \ref{har}, so $\bar \xi$ is indeed a long $\bar \eta'$-history. Also it is easy to see by \ref{kett} (or \ref{har} if $k^\bullet = \ell$) that for the last entry of $\bar \xi$: $$\xi_{k^\bullet} = \nu^*_{\delta_{k^\bullet}}(|\nu^*_{\delta_{k^\bullet}}|-1)< \varp.$$
	
	Clearly it is enough to verify that picking an arbitrary $\bar \eta'$-history $\bar \vartheta$ with $\vartheta_{|\vartheta|-1} < \varp$ and $\varp \notin  \{ \vartheta_0, \vartheta_1, \dots, \vartheta_{|\bar \vartheta|-1} \}$ necessarily
	$$ (\forall t \geq 0) \ \delta_{k^\bullet+t} \notin \{ \vartheta_0, \vartheta_1, \dots, \vartheta_{|\bar \vartheta|-1} \}.$$
	But  observe that if a long $\bar \eta'$-history $\bar \xi$ goes through $\delta_{k^\bullet+t}$, then it can go through $\delta_{k^\bullet+t-1}$, $\delta_{k^\bullet+t-2}$, $\dots$, $\delta_{k^\bullet}$ with the same ending point $\xi_{|\xi|-1}$.
	
	So it is enough to argue that if $\bar \vartheta$ is a long $\bar \eta'$-history with $\vartheta_j = \delta_{k^\bullet}$ for some $j$ then necessarily  $\vartheta_{|\vartheta|-1} \geq \varp = \delta_\ell$. Recall \ref{hi1}, pick the largest $t \geq 0$ such that $\vartheta_{j+t} = \delta_{k^\bullet+t}$ (so $k^\bullet+t < \ell$,  $\delta_{k^\bullet+t} \neq \varp$ as $\varp \notin \{ \vartheta_i: \ i < |\vartheta|\}$ by our assumptions)). Then $\vartheta_{j+t+1} \neq \delta_{k^\bullet+t+1}$ are of the form $\eta'_{\delta_{k^\bullet+t}}(n_1)$,   $\eta'_{\delta_{k^\bullet+t}}(n_2)$ for some $n_1,n_2 < \omega$, respectively, and as $\nu^*_{\delta_{k^\bullet+t+1}} = \langle \rangle = \eta'_{\delta_{k^\bullet+t}} \restriction n_2$ by\ref{har} we obtain that $n_2 < n_1$, so $\vartheta_{j+t+1} > \delta_{k^\bullet+t+1}$.
	But then it follows from \ref{fiok}, that all entries of $\vartheta$ are bigger than $\delta_{k^\bullet+t+1}$, so bigger than $\varp$.
	
	Since for every $\varrho \in S$ $h_{\eta'_\varrho(n)}$ is strictly increasing (by our assumptions) recalling the definition of $w^*_{\bar \vartheta} =  w_{\vartheta_{|\bar \vartheta|-1}} \cup (\cup\{[\delta_j, \delta_j+ h_{\vartheta_{j+1}}] : j \in [0, |\bar \vartheta|)  \})$ in Definition $\ref{w*}$ (and a similar reasoning applying \ref{fiok}) it is straightforward to check the moreover part.

\end{PROOF}

Now using Lemmas $\ref{ultuni}$ and $\ref{ultuniz}$ we will solve the nice $S$-uniformization problem $\gp$ with the aid of $\bar m$ and all the values resulting from Lemma $\ref{ultuniz}$.
\begin{claim} \label{Ff}
	Assume that $\bar \eta$ is a ladder system on the stationary set $S$, 
	$\langle A_i: \ i \in \omega \rangle$ is a $\bar \eta$-rank,
	$\bar w$ is a $\gp$-pressed down system, and Lemma $\ref{ultuni}$ holds with $\bar m = \langle m_\delta: \ \delta \in S \rangle$.
	
	Then there exists a function $F$ with $\dom(F)= \omega_1$, which is a solution of $\gp$.
\end{claim}
\begin{PROOF}{Claim \ref{Ff}}

We define
\newcounter{fobcou} \setcounter{fobcou}{0}
\begin{enumerate}[label = $(\circledast)_{\arabic*}$, ref = $(\circledast)_{\arabic*}$]
	\setcounter{enumi}{\value{fobcou}}
	\item $\Xi_\varepsilon = \{\bar \delta = \langle \delta_0, \delta_1, \ldots, \delta_{|\bar \delta|-1} \rangle: \bar \delta$ is a maximal $\bar \eta' = \bar \eta - \bar m$-history ending with $\varepsilon$, i.e.\ $\delta_{|\bar \delta|-1} = \varp\}$,
	\stepcounter{fobcou}
\end{enumerate}
let 
\begin{enumerate}[label = $(\circledast)_{\arabic*}$, ref = $(\circledast)_{\arabic*}$]
	\setcounter{enumi}{\value{fobcou}}
	\item $\Gamma = \{ \alpha < \omega_1: \ \Xi_\alpha \neq \emptyset\}$ (so $\Gamma = \bigcup_{\delta \in S} \left(\ran( \eta_\delta \restriction [m_\delta, \infty)\right) \setminus S)$.	
	\stepcounter{fobcou}
\end{enumerate}

\begin{enumerate}[label = $(\circledast)_{\arabic*}$, ref = $(\circledast)_{\arabic*}$]
	\setcounter{enumi}{\value{fobcou}}
	\item We choose for each $\zeta \in \Gamma$ a maximal history $\bar \delta^\zeta \in \Xi_\zeta$. 
		\stepcounter{fobcou}
	\item \label{fzeta} By induction on $\zeta \in \Gamma$ we choose a  function
	$f_\zeta$  such that
	\begin{enumerate}
		\item $\dom(f_\zeta) = w^*_{\bar \delta^\zeta}$, and
		\item\label{allo} 	$f_\zeta \in (\Psi^+)^\gp(w^*_{\bar \delta^\zeta})$, moreover,
		\item  introducing the object
		$$G_\zeta = \bigcup_{\alpha < \zeta} \bigcup_{\bar \vartheta \in \Xi_\alpha} g_{\bar \vartheta},$$ which we require to be a function, where by $g_{\bar \vartheta}$ we mean the   function with domain $w^*_{\bar \vartheta}$ defined as follows.
	\end{enumerate}
	\stepcounter{fobcou}	
\end{enumerate}

\begin{enumerate}[label = $(\circledast)_{\arabic*}$, ref = $(\circledast)_{\arabic*}$]
	\setcounter{enumi}{\value{fobcou}}
	\item \label{komp}   For $\alpha  \in \zeta \cap \Gamma$ 
	and $\bar \vartheta \in \Xi_\alpha$ we let
	$g_{\bar \vartheta} = f_\alpha \circ \text{OP}_{w^*_{\bar \delta^\alpha} ,w^*_{\bar \vartheta}}$ where $\dom(f_\alpha ) = w^*_{\bar \delta^\alpha}$
	(so $g_{\bar \vartheta}$ is a function with domain $w^*_{\bar \vartheta}$), and
	\item we let $B_\zeta = \bigcup_{\alpha < \zeta} \bigcup_{\bar \vartheta \in \Xi_\alpha} w^*_{\bar \vartheta}$. 	
	\stepcounter{fobcou}
		\stepcounter{fobcou}
\end{enumerate}

%\begin{enumerate}[label = $(\circledast)_{\arabic*}$, ref = $(\circledast)_{\arabic*}$]
%	\setcounter{enumi}{\value{fobcou}}
	
%\end{enumerate}

	Note that by \ref{komp}  
\begin{equation} \label{circ}(\text{for } \alpha < \zeta): \  \forall \bar \delta, \bar \delta' \in \Xi_\alpha: \ \ G_\zeta
	\restriction w^*_{\bar \delta'} \circ  (\text{OP}_{w^*_{\bar \delta'}, w^*_{\bar \delta}}) = G_\zeta \restriction  w^*_{\bar \delta},  \end{equation}
moreover, by our demands on $\bar m$  using \ref{ket} and \ref{allo} clearly
\begin{equation} \label{circko} \forall  \bar \delta' \in \Xi_\alpha: \ \ G_\zeta
	\restriction w^*_{\bar \delta'} \in (\Psi^\gp)^+(w^*_{\bar \delta'}), \end{equation}

First we argue that (provided we can carry out the induction) the resulted function $F = \bigcup_{\alpha < \omega_1} G_\alpha$ is a solution for $\gp$.
Note that (by \ref{veryfak3}/ \ref{k3av}  in Fact $\ref{veryfa}$, and recalling $u_\alpha \subseteq w_\alpha \subseteq \alpha$ in Definition $\ref{wdf}$)

\begin{enumerate}[label = $(\circledast)_{\arabic*}$, ref = $(\circledast)_{\arabic*}$]
	\setcounter{enumi}{\value{fobcou}}
	\item if $\delta \in S$, then $\bigcup_{n \in \omega} w_{\eta_\delta(m_\delta + n)} = \delta$,
	\stepcounter{fobcou}
\end{enumerate}
so by choosing for each $n$ a maximal $\bar\eta'$-history through $\eta'_\delta(m_\delta+n)$ and $\delta$ clearly $\delta \subseteq B_\delta$.
As $S$ is stationary $F$ is an entire function, defined on each ordinal $\alpha < \omega_1$.
We only have to check that for each $\delta$, and each $n,k < \omega$
$$ F\restriction b^\delta_{n,k} \in \Psi^\gp(b^\delta_{n,k}).$$
For this having $\delta$, $n,k$ fixed choose $n^\bullet$ large enough so that
$$ \eta'_\delta(n^\bullet) = \eta_\delta(m_\delta+n^\bullet) \geq \eta_\delta(n),$$
and 
$$h_{\eta'_\delta(n^\bullet)} \geq k.$$
Now if $\zeta \in \Gamma$ is such that there exists a long $\bar \eta'$-history $\bar \delta \in \Xi_\zeta$ through $\eta'_\delta(n^\bullet)$ and $\delta$, then by Definition $\ref{w*}$
we have 
$$w^*_{\bar \delta} \supseteq b^\delta_{m_\delta+n^\bullet,h_{\eta'_\delta(n^\bullet)}} \supseteq b^\delta_{n,k}, $$
and by $\eqref{circko}$ $$F \restriction w^*_{\bar \delta} = G_{\zeta+1} \restriction w^*_{\bar \delta} \in (\Psi^\gp)^+(w^*_{\bar \delta}),$$
and by \ref{pdb} $F \restriction b^\delta_{n,k} \in \Psi^\gp(b^\delta_{n,k})$, we are done.

This means that the next claim completes the proof of Claim $\ref{Ff}$, so that of our main theorem.

\begin{sclaim}\label{fzetak}
	We can construct the sequence $\langle f_\zeta: \ \zeta \in \Gamma \rangle$ satisfying \ref{fzeta}.
\end{sclaim}
\begin{PROOF}{Subclaim \ref{fzetak}}
	We proceed by induction, obviously only the successor step is interesting.
	So assume that $f_\beta$ is constructed for each $\zeta \cap \Gamma$, we need to construct a suitable $f_\zeta$.

	We have to characterize the set $B_\zeta \cap w^*_{\bar \delta^\zeta}$ (and $B_\zeta \cap w^*_{\bar \delta}$ for $\bar \delta \in \Xi_\zeta$), the set where our function is already defined, and we have to argue that the induction hypothesis ensures that the required compatibility in \ref{komp} can be achieved. 
	
	Appealing to \ref{hi3} of Lemma $\ref{ultuniz}$ with $\bar \delta^\zeta = \langle \delta^\zeta_0, \delta^\zeta_1, \dots, \delta^\zeta_\ell = \zeta \rangle$ there exists a long $\bar \eta'$-history $\bar \xi = \langle \xi_0, \xi_1, \dots, \xi_k \rangle$ with  $\xi_k < \zeta$,  and with the property that
	whenever for a maximal $\bar \eta'$-history $\bar \vartheta$ we have $\vartheta_j = \delta^\zeta_i$ for some $j < \lh(\bar \vartheta)$, $i < \ell$, but $\vartheta_{|\vartheta|-1} < \zeta$, then $w^*_{\bar \vartheta} \cap [\delta^\zeta_i, \delta^\zeta_i+\omega) \subseteq w^*_{\bar \xi}$, in particular, $\delta^\zeta_i = \xi_n$ for some $n<k$. This also implies that
	\begin{equation} \label{foeke}
		B_\zeta \cap w^*_{\bar \xi} \subseteq \zeta \cup w^*_{\bar \xi}.
	\end{equation}

	Recalling the definition of $w^*_{\bar \delta}$'s ($w^*_{\bar \delta} = w_{\delta_{\lh(\bar\delta)-1}} \cup \bigcup_{k<\lh(\bar\delta)-1} [\delta_k, \delta_k+ h_{\delta_{k+1}}]$ Definition $\ref{w*}$, where in fact the sequence of $h_{\delta_k}$'s only depend on $\delta_{\lh(\bar\delta)-1}$) this clearly implies that
	\begin{equation} \label{foek}
	B_\zeta \cap w^*_{\bar \xi} \subseteq \zeta \cup w^*_{\bar \xi}
\end{equation}
	(as $\bar \xi$ works for any $\bar \vartheta$ and $i$).
	
	By extending $\bar \xi$ to be a maximal $\bar \eta'$-history if necessary we can assume that $\rho := \xi_k \in \Gamma$.
	Now \ref{hi2} from Lemma $\ref{ultuniz}$ yields that for every $\bar \delta' \in \Xi_\zeta$ there exists $\bar \xi' \in \Xi_\rho$, such that 
		$$ (\text{OP}_{w^*_{\bar \xi'}, w^*_{\bar \xi}})``w^*_{\bar \delta^\zeta} \cap w^*_{\bar \xi} =  w^*_{\bar \delta'} \cap w^*_{\bar \xi'},$$
	and
		$$ (\text{OP}_{w^*_{\bar \delta'}, w^*_{\bar \delta^\zeta}})``w^*_{\bar \delta^\zeta} \cap w^*_{\bar \xi} =  w^*_{\bar \delta'} \cap w^*_{\bar \xi'},$$
		so 
		\begin{equation} \label{fo0ek}
			\text{OP}_{w^*_{\bar \xi'}, w^*_{\bar \xi}} \rest w^*_{\bar \delta^\zeta} \cap w^*_{\bar \xi} = \text{OP}_{w^*_{\bar \delta'}, w^*_{\bar \delta^\zeta}} \rest w^*_{\bar \delta^\zeta} \cap w^*_{\bar \xi}, 
		\end{equation}
		as both functions are order preserving.
		So \eqref{fo0ek} together  with \eqref{foeke} (and $\eqref{circ}$) imply that 
		\begin{equation}
			G_\zeta  \circ \text{OP}_{w^*_{\bar \delta^\zeta}, w^*_{\bar \delta'}} \restriction  (w^*_{\bar \delta^\zeta} \cap B_\zeta) = G_\zeta  \restriction (w^*_{\bar \delta'} \cap B_\zeta) 
		\end{equation}
	 (since $\text{OP}_{w^*_{\bar \delta^\zeta}, w^*_{\bar \delta'}}$ fixes each ordinal smaller or equal to $\zeta = \delta^\zeta_\ell = \delta'_\ell$).
		
		We claim that it is enough to extend $G_\zeta \restriction B_\zeta \cap w^*_{\bar \delta^\zeta}$ to a function $f_\zeta$ defined on $w^*_{\bar \delta^\zeta}$ with $f_\zeta \in (\Psi^\gp)^+(w^*_{\bar \delta^\zeta})$.
		So suppose that such $f_\zeta$ exists, we need that 
		$G_{\zeta+1} = G_\zeta \bigcup \cup\{g_{\bar \delta'}: \ \bar \delta' \in \Xi_\zeta\}$ is a function, which in turn would imply that 
		$$G_{\zeta+1} \circ \text{OP}_{w^*_{\bar \delta^\zeta}, w^*_{\bar \delta'}} = G_{\zeta+1} \rest w^*_{\bar \delta'}.$$ Suppose that $\beta \in w^*_{\bar \delta'} \cap w^*_{\bar \delta''}$ for some $\bar \delta', \bar \delta'' \in \Xi_\zeta$. Then recalling that $ w^*_{\bar \delta'} = w_{\delta'_{|\bar \delta'|-1}} \cup (\cup\{[\delta_i, \delta_i+ h_{\delta_{i+1}}] : i < |\bar \delta'|-1 \})$,  OP$_{w^*_{\bar \delta'}, w^*_{\bar \delta''}} =$OP$_{w^*_{\bar \delta'}, w^*_{\bar \delta^\zeta}} \circ $OP$_{w^*_{\bar \delta^\zeta}, w^*_{\bar \delta''}}$ fixes $\beta$, so OP$_{\bar \delta^\zeta, \bar \delta''}(\beta) =$ OP$_{\bar \delta^\zeta, \bar \delta'}(\beta)$, and we are done (remembering \ref{komp}).
		
	For this (i.e. for a suitable $f_\zeta$) we only have to prove (recalling \ref{pdd} from Definition $\ref{pprobd}$) that $B_\zeta \cap w^*_{\bar \delta^\zeta}$ is $\gp$-closed,	
	and check that $G_\zeta \restriction B_\zeta \cap w^*_{\bar \delta^\zeta} \in (\Psi^\gp)^+(B_\zeta \cap w^*_{\bar \delta^\zeta})$, i.e. for $\gp$-basic sets $b \subseteq w^*_{\bar \delta^\zeta}$ $G \restriction b \in \Psi^\gp(b)$.
	So fix $\delta^\bullet \in S$, $\delta^\bullet \in B_\zeta \cap w^*_{\bar \delta^\zeta}$.
	First assume that $\delta^\bullet < \zeta$, so $\delta^\bullet \in w_{\zeta}$ by the definition of $w^*_{\bar \delta^\zeta}$. But then as $w_\zeta$ is $\gp$-closed by our present assumptions on $\bar w$ and Definition $\ref{wdf}$, so for some $n$ and $k$
	$$ \eta_{\delta^\bullet}(n) \cup v_{\eta_\delta^\bullet(n)} \subseteq w_\zeta \cap \delta^\bullet \subseteq \eta_{\delta^\bullet}(n+1),$$
	and 
	$$ w_\zeta \cap [\delta^\bullet, \delta^\bullet+\omega) = [\delta^\bullet, \delta^\bullet+k].$$
	(note that here we use the original ladder system $\eta$).
	Then for some $n'$ $\eta'_{\delta^\bullet}(n') = \eta_{\delta^\bullet}(m_{\delta^\bullet}+n') \geq \eta_{\delta^\bullet}(n)$, and $h_{\eta'_{\delta^\bullet}(n')} \geq k$ (as $\bar h$ is permitted for $\bar \eta$).
	Then choosing a maximal $\bar \eta'$-history $\bar \xi$ going through 
	$\eta'_{\delta^\bullet}(n')$ and $\delta^\bullet$, the following necessarily holds:
	$$ b^{\delta^\bullet}_{n,k} \subseteq b^{\delta^\bullet}_{m_{\delta^\bullet}+ n',h_{\eta(m_{\delta^\bullet}+n')}} = \{ \eta_{\delta^\bullet}(m_{\delta^\bullet} + n') \} \cup v_{\eta_{\delta^\bullet}(m_{\delta^\bullet}+n')} \cup [\delta^\bullet,\delta^\bullet+h_{\eta(m_{\delta^\bullet}+n')}] \subseteq w^*_{\bar \xi}$$
	 (note that $b^{\delta^\bullet}_{m_{\delta^\bullet}+ n',h_{\eta(m_{\delta^\bullet}+n')}}$ is a $\gp$-basic set as well). 
	As $G_\zeta \restriction w^*_{\bar \xi} \in  (\Psi^\gp)^+(w^*_{\bar \xi})$ by $\eqref{circko}$, clearly $G_\zeta \restriction b^{\delta^\bullet}_{n,k} \in \Psi^\gp(b^{\delta^\bullet}_{n,k})$ (by \ref{pdc}), we are done.
	
	Now since $\zeta \in \Gamma$, so $\zeta \notin S$, the other possibility is that $\delta^\bullet > \zeta$.
	Then $\delta^\bullet \in w^*_{\bar \delta^\zeta}$ means that $\delta^\bullet = \delta^\zeta_i$ for some $i < \ell$ ($\delta^\bullet \in S$), and so $\delta^\zeta_{i+1} = \eta'_{\delta^\bullet}(n)$ for some $n$.
	Note that by \ref{fiok}
	\begin{equation} \label{eeg}
		\eta'_{\delta^\bullet}(n-1) < \zeta < \eta'_{\delta^\bullet}(n) = \delta^\zeta_{i+1}, 
	\end{equation}
and then
 	\begin{equation} \label{tar}
 		 w_{\delta^\zeta_{i+1}} \cup \{ \eta'_{\delta^\bullet}(n) \} \subseteq w_\zeta \cup \{ \eta'_{\delta^\bullet}(n) \} \subseteq w^*_{\bar \delta^\zeta} \cap \delta^\bullet \subseteq \delta^\zeta_{i+1} + 1 = \eta'_{\delta^\bullet}(n) +1
 	\end{equation}
 	by Definition $\ref{wdf}$ and the $\gp$-closedness of $w^*_{\bar \delta^\zeta}$,
 	and
 		\begin{equation} \label{tar+k}
 		[\delta^\bullet, \delta^\bullet + h_{\eta'_\delta(n)}] = w^*_{\bar \delta^\zeta} \cap [\delta^\bullet, \delta^\bullet + \omega).
 	\end{equation}
 	
	Now $\delta^\bullet \in B_\zeta$ implies that 
	\begin{enumerate}[label = $(\bullet)_1$, ref = $(\bullet)_1$]
		\item \label{elo} there is a maximal $\eta'$-history $\bar \xi = \langle \xi_0, \xi_1, \dots, \xi_k \rangle$ with $\xi_k < \zeta$ going through $\delta^\bullet$, so $\xi_j = \delta^\bullet$, and $\xi_{j+1} = \eta'_{\delta^\bullet}(n^\bullet)$ for some $j$ and $n^\bullet$ (where $m_\delta \geq n^\bullet$, since $\bar \xi$ is an $\bar \eta$-history).
	\end{enumerate} 
At this point let $n^\bullet$ be maximal such that \ref{elo} holds (i.e.\ for some $\bar \eta'$-history $\bar \xi = \langle \xi_0, \ldots, \xi_k \rangle$, $\bar \xi$ has $\delta^\bullet$ as an entry, and $\xi_k < \zeta$), then by finding a maximal history through $\eta'_{\delta^\bullet}(n^\bullet)$ we can assume that $\xi_{j+1} = \eta'_{\delta^\bullet}(n^\bullet)$, and $n^\bullet$ is the maximal $t$ such that $\eta'_{\delta^\bullet}(t) \in B_\zeta$.  Such maximal $n^\bullet$ exists by the following. We are going to argue that 
\begin{enumerate}[label = $(\bullet)_2$, ref = $(\bullet)_2$]
	\item \label{sufv}
	 either $n^\bullet = n-1$, or $n^\bullet =n$ where $n$ is 
	from $\eqref{eeg}$.
\end{enumerate}
First recall $\eqref{eeg}$, so either $\eta'_\delta(n-1) \in \Gamma$, or it belongs to $S$, and there is a  maximal $\eta'$-history through it and $\delta$, so $n^\bullet \geq n-1$.
On the other hand, $\zeta < \eta'_{\delta^\bullet}(n) < \eta'_{\delta^\bullet}(n+1)$, and again by \ref{fiok} any history through $\eta'_{\delta^\bullet}(n+t)$ for $t>0$ must have a last entry bigger than $\eta'_{\delta^\bullet}(n+t-1)$, so bigger than $\eta'_{\delta^\bullet}(n)$, contradicting the fact that so far we only considered maximal $\eta'$-histories with last entry smaller than $\zeta$.

Now in order to finish the proof of the subclaim it suffices to show that
\begin{enumerate}
	\item $(w_{\delta^\zeta}^* \cap B_\zeta) \cap \delta^\bullet \subseteq \eta'_{\delta^\bullet}(n^\bullet +1)$,
	\item $(w^*_{\bar \delta^\zeta} \cap B_\zeta) \cap [\delta^\bullet, \delta^\bullet+\omega) = [\delta^\bullet, \delta^\bullet+ h_{\eta'_{\delta^\bullet}(n^\bullet)}]$,
	\item   so $b^{\delta^\bullet}_{m_{\delta^\bullet} + l,k} \subseteq B_\zeta \cap w^*_{\bar \delta^\zeta}$ iff $l \leq n^\bullet$ and $k \leq h_{\eta'_{\delta^\bullet}(n^\bullet)}$,
	\item and $G_\zeta \rest b^{\delta^\bullet}_{m_{\delta^\bullet} + l,h_{\eta'_{\delta^\bullet}(n^\bullet)}} \in \Psi^\gp(b^{\delta^\bullet}_{m_{\delta^\bullet} + l,h_{\eta'_{\delta^\bullet}(n^\bullet)}})$,
\end{enumerate}
	this will imply that $w_{\delta^\zeta}^* \cap B_\zeta$ is $\gp$-closed and $G_\zeta \rest w_{\delta^\zeta}^* \cap B_\zeta \in (\Psi^+)^\gp$, as $\delta^\bullet$ was an arbitrary element of $(w_{\delta^\zeta}^* \cap B_\zeta) \cap S$.

	It follows from the maximality of $n^\bullet$ (and the definition of $w^*_{\bar \xi}$, and $\bar h$ being permitted)  that
	\begin{equation}
		B_\zeta \cap [\delta^\bullet, \delta^\bullet+\omega) = [\delta^\bullet, \delta^\bullet+ h_{\eta'_{\delta^\bullet}(n^\bullet)}] = w^*_{\bar \xi} \cap [\delta^\bullet, \delta^\bullet+\omega).
	\end{equation}
	
	Then similarly to $\eqref{tar}$, $\eqref{tar+k}$
	\begin{equation} \label{tartthe}
		w_{\eta'_{\delta^\bullet}(n^\bullet)} \cup \{ \eta'_{\delta^\bullet}(n^\bullet) \} \cup [\delta^\bullet, \delta^\bullet\textbf{}+ h_{\eta'_{\delta^\bullet}(n^\bullet)}] \subseteq w^*_{\bar \xi} \subseteq B_\zeta.
	\end{equation}

 If $n^\bullet = n-1$,  then by the maximality of $n^\bullet$ we would have $\eta'_{\delta^\bullet}(n) \notin B_\zeta$ so putting together $\eqref{tar}$, $\eqref{tar+k}$ and $\eqref{tartthe}$ we have
	\begin{equation} \label{lezarvegre} w_{\eta'_{\delta^\bullet}(n^\bullet)} \cup \{ \eta'_{\delta^\bullet}(n^\bullet) \} \subseteq (B_\zeta \cap w^*_{\bar \delta^\zeta}) \cap \delta^\bullet \subseteq \eta'_{\delta^\bullet}(n^\bullet +1),
		\end{equation}
	and
	\begin{equation} \label{lezarvegre2} (B_\zeta \cap w^*_{\bar \delta^\zeta}) \cap [\delta^\bullet, \delta^\bullet+\omega) = [\delta^\bullet, \delta^\bullet+ h_{\eta'_{\delta^\bullet}(n^\bullet)}]. \end{equation}
	On the other hand, if $n^\bullet = n$, then clearly $w^*_{\bar \delta^\zeta} \cap [\delta^\bullet, \delta^\bullet+\omega) = [\delta^\bullet, \delta^\bullet+ h_{\eta'_{\delta^\bullet}(n^\bullet)}]$, and $\eqref{lezarvegre}$, $\eqref{lezarvegre2}$ also hold.
	
	In each case  $b^\bullet := b^{\delta^\bullet}_{m_{\delta^\bullet}+n^\bullet, h_{\eta'_{\delta^\bullet}(n^\bullet)}}$ is a maximal $\gp$-basic set at $\delta^\bullet$ that is a subset of  $B_\zeta \cap w^*_{\bar \delta^\zeta}$ (since $v_{\eta'_{\delta^\bullet}(n^\bullet)} \subseteq w_{\eta'_{\delta^\bullet}(n^\bullet)}$ by Definition $\ref{wdf}$). It is also clear that $B_\zeta \cap w^*_{\bar \delta^\zeta}$ is $\gp$-closed. Moreover, $b^\bullet \subseteq w^*_{\xi}$ for the long $\bar \eta'$-history $\bar \xi$ going through $\delta^\bullet$ and $\eta'_{\delta^\bullet}(n^\bullet)$,  so $G_\zeta \restriction w^*_{\bar \xi} \in (\Psi^\gp)^+(w^*_{\bar \xi})$ by $\eqref{circko}$, so $G_\zeta \restriction b^\bullet \in \Psi^\gp(b^\bullet)$ by \ref{pdb}.

\end{PROOF}
\mn

\end{PROOF}

In order to finish the proof of Lemma \ref{Lebe}, suppose that $\gp$ is a very special $S$-uniformization system which is nice. By Claim \ref{elok} there exists some $\bar m$ and $\bar A = \langle a_n: \ n \in \omega \rangle$ such that $\bar A$ is an $\bar \eta' = \bar \eta - \bar m$-rank.
By Lemma \ref{wlem} for some $\bar m'$ the system $\bar w$ is a $\bar \eta'' = \bar \eta - \bar m'$-pressed down system.
Therefore, Lemma \ref{ultuni} and Claim \ref{Ff} apply to $\gp'' = (\gp - \bar m) - \bar m'$, and we have a solution to that (i.e. to $\gp''$).

\end{PROOF}

\end{PROOF}
\noindent
\section*{Acknowledgement}

The authors thank the anonymous referee for their numerous invaluable comments and suggestions and for helping us considerably improve the readability of the paper.

%\end{PROOF}

\bibliographystyle{amsalpha}
\bibliography{shlhetal,486}
\end{document}